\documentclass[11pt]{article}

\usepackage{amsmath, amssymb, bbm, xspace}

\newcommand{\Real}{\ensuremath{\mathbb{R}}}

\newcommand{\argminn}{\mathop{\hbox{\rm argmin}}}
\newcommand{\argmin}[1]{\displaystyle\argminn_{#1}}

\DeclareMathOperator*{\st}{subject\;to}

\def\half  {{\textstyle{1\over 2}}}

\def\norm#1{\|#1\|}
\def\spose#1{\hbox to 0pt{#1\hss}}

\def\text #1{\hbox{\quad#1\quad}}

\def\Nbold{{\mathbf{N}}}
\def\Mbold{{\mathbf{M}}}

\def\nthinsp{\mskip -2   mu}

\def\superstar{^{\raise 0.5pt\hbox{$\nthinsp *$}}}
\def\SUPERSTAR{^{\raise 0.5pt\hbox{$*$}}}

\def\lamstarT {\lambda^{\raise 0.5pt\hbox{$\nthinsp *$}T}}

\def\Cscr{{\cal C}}

\def\Fscr{{\cal F}}

\def\Kscr{{\cal K}}

\def\Gscr{{\cal G}}

\def\Nscr{{\cal N}}

\def\Nscr{{\cal N}}

\def\Xscr{{\cal X}}
\def\Yscr{{\cal Y}}

\def\hbar{\skew{4.2}\bar h}

		\def\bkE{{\rm I\kern-.17em E}}
		\def\bk1{{\rm 1\kern-.17em l}}
		\def\bkD{{\rm I\kern-.17em D}}
		\def\bkR{{\rm I\kern-.17em R}}
		\def\bkR{\mathbb{R}}
		\def\bkP{{\rm I\kern-.17em P}}
		\def\bkY{{\bf \kern-.17em Y}}
		\def\bkZ{{\bf \kern-.17em Z}}

		\def\beq{\begin{eqnarray}}
		\def\bc{\begin{center}}
		\def\be{\begin{enumerate}}
		\def\bi{\begin{itemize}}
		\def\bs{\begin{small}}
		\def\bS{\begin{slide}}
		\def\ec{\end{center}}
		\def\ee{\end{enumerate}}
		\def\ei{\end{itemize}}
		\def\es{\end{small}}
		\def\eS{\end{slide}}
		\def\eeq{\end{eqnarray}}
		\def\qed{\quad \vrule height7.5pt width4.17pt depth0pt}

	\def\cp2problem#1#2#3#4{\fbox
		 {\begin{tabular*}{0.9\textwidth}
			{@{}l@{\extracolsep{\fill}}l@{\extracolsep{6pt}}l@{\extracolsep{\fill}}c@{}}
				#1 & & $#4 $ 
			\end{tabular*}}}
\newcommand{\pmat}[1]{\begin{pmatrix} #1 \end{pmatrix}}
		
		\renewcommand{\emph}[1]{\textbf{#1}}

		\def\bkE{{\rm I\kern-.17em E}}
		\def\bk1{{\rm 1\kern-.17em l}}
		\def\bkD{{\rm I\kern-.17em D}}
		\def\bkP{{\rm I\kern-.17em P}}
		
		\def\bkZ{{\bf{Z}}}

\newcommand {\beeq}[1]{\begin{equation}\label{#1}}
\newcommand {\eeeq}{\end{equation}}
\newcommand {\bea}{\begin{eqnarray}}
\newcommand {\eea}{\end{eqnarray}}

\def\texitem#1{\par\smallskip\noindent\hangindent 25pt
               \hbox to 25pt {\hss #1 ~}\ignorespaces}

\def\st{\mbox{subject to}}

\newcommand{\tabincell}[2]{\begin{tabular}{@{}#1@{}}#2\end{tabular}}
\usepackage{graphicx}
\usepackage{mathrsfs}
	\usepackage{verbatim,color}	
	\usepackage{algorithm,epsfig,verbatim}
	\usepackage{amsmath,mathtools,amsthm}
\usepackage{tcolorbox}

\def\Gcal{\mathscr{G}}
\usepackage{caption} 
\usepackage{multirow}
\usepackage[para,online,flushleft]{threeparttable}
\usepackage{tikz}
\usetikzlibrary{calc,shapes.geometric,positioning,arrows,intersections}
\usepackage{pgfplots}
	\newtheorem{assumption}{Assumption}

\allowdisplaybreaks

\usepackage{marginnote}
\usepackage{algorithm,algorithmic}
\let\oldmarginpar\marginpar
\renewcommand\marginpar[1]{\-\oldmarginpar[\raggedleft\footnotesize #1]%
{\raggedright\footnotesize\color{red} #1}} 
\newtheorem{theorem}{Theorem}
\newtheorem{corollary}{Corollary}
\newtheorem{lemma}[theorem]{Lemma}
\newtheorem{definition}{Definition}

\newtheorem{proposition}{Proposition}

\newcommand{\captionfonts}{\small}
\makeatletter  
\long\def\@makecaption#1#2{%
  \vskip\abovecaptionskip
  \sbox\@tempboxa{{\captionfonts #1: #2}}%
  \ifdim \wd\@tempboxa >\hsize
    {\captionfonts #1: #2\par}
  \else
    \hbox to\hsize{\hfil\box\@tempboxa\hfil}%
  \fi
  \vskip\belowcaptionskip}
\makeatother   

\newcommand{\gb}{\tilde {\bf g}}
\newcommand{\x}{{\bf x}}
\newcommand{\p}{{\bf p}}
\newcommand{\vv}{{\bf v}}
\newcommand{\y}{{\bf y}}
\newcommand{\z}{{\bf z}}

\newcommand{\us}[1]{\textcolor{black}{#1}}

\newcommand{\usv}[1]{{\color{red}#1}}

\newcommand{\ssc}[1]{{\color{black}#1}}
\newcommand{\ic}[1]{{\color{black}#1}}

\begin{document}
\title{On the computation of equilibria in monotone and potential stochastic hierarchical games}
\author{Shisheng Cui \and Uday V. Shanbhag\thanks{Research was partially
supported by NSF CMMI-1538605 and
DOE ARPA-E award  DE-AR0001076 (Shanbhag). The authors are contactable at \texttt{suc256,udaybag@psu.edu}.}}
\date{\today}
\maketitle
\begin{abstract}
We consider a class of hierarchical noncooperative $\Nbold$-player games where
the $i$th player solves a parametrized stochastic mathematical program with equilibrium
constraints (MPEC) with the caveat that the implicit form of the $i$th player's
in MPEC is convex in player strategy, given rival decisions. 
Few, if any, general
purpose schemes exist for computing equilibria  even for deterministic
specializations of such games. We develop computational schemes in two distinct
regimes: (a) {\em Monotone regimes.} When player-specific implicit problems are
convex, then the necessary and sufficient equilibrium conditions are given by a
stochastic inclusion. Under a monotonicity assumption on the operator, we
develop a variance-reduced stochastic proximal-point scheme that achieves
deterministic rates of convergence in terms of solving proximal-point problems
in monotone/strongly monotone regimes and  the schemes are
characterized by optimal or near-optimal sample-complexity
guarantees.  Finally, the generated sequences are shown to be convergent to an equilibrium 
in an almost-sure sense in both monotone and strongly monotone
regimes; (b) {\em
Potentiality.} When the implicit form of the game admits a potential function,
we develop an asynchronous relaxed inexact smoothed proximal best-response
framework. However, any such avenue is impeded by the need to efficiently
compute an approximate solution of an MPEC with a strongly convex implicit
objective. To this end, we consider the smoothed counterpart of this game where
each player's problem is smoothed via randomized smoothing. Notably, under
suitable assumptions, we show that an $\eta$-smoothed game admits an
$\eta$-approximate Nash equilibrium of the original game. Our proposed scheme
produces a sequence that converges almost surely to an $\eta$-approximate Nash
equilibrium in both relaxed and unrelaxed settings. This scheme is reliant on computing the proximal problem, a
stochastic MPEC whose implicit form has a strongly convex objective, with
increasing accuracy in finite-time. The smoothing framework allows for
developing a variance-reduced zeroth-order scheme for such problems that admits
a fast rate of convergence. Numerical studies on a class of multi-leader
multi-follower games suggest that variance-reduced proximal schemes provide significantly better accuracy with far lower run-times. The relaxed best-response scheme scales well will problem size and generally displays more stability than its unrelaxed counterpart.
\end{abstract}

\section{Introduction}
In this paper, we consider the class of $\Nbold$-player noncooperative hierarchical games in uncertain regimes. We consider a class of $\Nbold$-player games in which the $i$th player solves the following parametrized problem. 
\begin{align}  \tag{Player$_i(\x^{-i})$}
    \begin{aligned}
    \min_{\x^i \in \Xscr_i} & \quad f_i(\x^i,\x^{-i}) \triangleq \mathbb{E}\left[\tilde{g}_i(\x^i,\x^{-i},\xi(\omega)) +\tilde{h}_i(\x^i,\y^i(\x,\xi(\omega)),\omega)\right] \\
        \st & \quad \mathbb{E}[\tilde{c}_i(\x^i,\xi(\omega))] \leq 0,
\end{aligned}
\end{align}
where $\x^i \subseteq \Xscr_i \subseteq \Real^{n_i}$, $\sum_{i=1}^{\Nbold} n_i = n$, $\xi: \Omega \to \Real^d$ represents the $d-$valued random variable, 
$\tilde{g}_i: \Real^n \times \Real^d \to \Real$, $\tilde{c}_i : \Real^{n_i} \times \Real^d \to \Real$, and $\tilde{h}_i: \Real^{n_i+m_i} \times \Real^d \to \Real$ are real-valued
functions, $\y^i: \Xscr_i \times \Real^d \to \Real^{m_i}$ is a single-valued mapping
corresponding to the  unique solution of the $i$th player's lower-level
problem, given $\x$, and $i \in \{1, \cdots, \Nbold\}$. Note that $\x^{-i} \triangleq \{\x^j\}_{j \neq i}$, $\Xscr_{-i} \triangleq \prod_{j\neq i} \Xscr_j$ and $\Xscr \triangleq \prod_{i=1}^{\Nbold} \Xscr_i$. In Section~\ref{sec:3.1}, we consider the generalization where the $i$th player's problem is additionally constrained by $\mathbb{E}[\tilde{c}_i(\x^i,\xi(\omega))] \leq 0$ where $\tilde{c}_i : \Real^{n_i} \times \Real^d \to \Real$. Suppose the associated probability space is $(\Omega,\mathbb{P}, \Fscr)$ and $\mathbb{E}[\bullet]$ represents the expectation with respect to the probability measure $\mathbb{P}$. In the remainder of the paper, we suppress the $\xi$ for expository clarity and refer to $\y^i(\x,\xi(\omega))$ by $\y^i(\x,\omega)$. Suppose the $\omega$-specific lower-level problem associated with player $i$ is defined as the unique solution to a parametrized variational inequality problem, defined as 
\begin{align*}
    \y^i(\x,\omega) \ =  \ \mbox{SOL}\left(\Yscr_i, F_i(\bullet,\x,\omega)\right), \mbox{ for } \omega \in \Omega
\end{align*}
where SOL$(\Yscr_i, F_i(\bullet,\x,\omega))$ denotes the solution set of a parametrized variational inequality problem VI$(\Yscr_i, F_i(\bullet,\x,\omega))$, $F_i: \Real^{m_i} \times \Real^n \times \Real^d \to \Real^{m_i}$ is a real-valued map, and $\Yscr_i \subseteq \Real^{m_i}$ is a closed and convex set. This is a flexible framework that subsumes a broad class of games as shown next. 

\medskip

\noindent{\bf \large 1.1. Convex hierarchical games under uncertainty.} Consider the proposed class of noncooperative convex hierarchical games denoted by $\Gscr^{\rm chl}$. For any game $\Gcal \in \Gscr^{\rm chl}$, for $i = 1, \cdots, \Nbold$,  the $i$th player's problem, denoted by (Player$_i(\x^{-i})$) is a convex program for every $\x^{-i} \in \Xscr_{-i}$. The class $\Gscr^{\rm chl}$ subsumes the following subclasses. We refer to this subclass of 

\noindent {(i) \em Single-level noncooperative games with expectation-valued
objectives.} When $h_i  \equiv 0$ for $i \in \{1, \cdots, \Nbold\}$ where $h_i(\x) \triangleq \mathbb{E}[\tilde{h}_i(\x^i,\y^i(\x,\omega),\omega)]$, this reduces to a class of single-level games with
expectation-valued objectives. This class of games has been extensively studied,
both in terms of analysis when the problem (Player$_i(\x{^{-i}})$) is convex for every
$\x{^{-i}} \in \Xscr_{-i}$~\cite{FPang09,ravat2011characterization} as well as computation when
the game admits suitable monotonicity~\cite{koshal2013regularized} or potentiality properties~\cite{monderer1996potential}, amongst others.

\medskip 

\noindent {(ii) \em  Multi-leader multi-follower games under uncertainty.}
Multi-leader multi-follower games arise when there is a collection of followers
that participate in a noncooperative game, parametrized by leader-level
decisions. Contingent on the equilibrium decisions of the followers, leaders
compete in a noncooperative game. This class of games, referred to as
multi-leader multi-follower games, have been
analyzed in stylized deterministic~\cite{sherali1984multiple} and
stochastic~\cite{demiguel2009stochastic,shanbhag11complementarity} (see recent survey in~\cite{aussel2020short}). While
existence of such equilibria in such games is by no means a given
(see~\cite{pang05quasi} for simple settings where equilibria fail to exist),
existence guarantees have been provided for subclasses in stylized
settings~\cite{sherali1984multiple,hu07epecs,su2007analysis} as well as under the
availability of a  potential function~\cite{kulkarni15existence}. Computation
of equilibria has focused on considering the associated complementarity
problems~\cite{leyffer05multi,demiguel2009stochastic}; prior {efforts} have
included smoothing approaches~\cite{steffensen18quadratic},  heuristic
approaches~\cite{kulkarni14shared,hu07epecs}, and sampling-based
approximations~\cite{demiguel2009stochastic}. However, almost all of the
approximation/smoothing approaches have focused on computing solutions to
necessary conditions~\cite{leyffer05multi,steffensen18quadratic}, rather than
equilibria.   To the best of our knowledge, {\em there are  no convergent schemes for such games or 
reasonable subclasses even in deterministic settings}.

\medskip

\noindent {(iii) \em Bilevel games under uncertainty.} We define bilevel games
as being a subclass of multi-leader multi-follower games in which the
lower-level problem is parametrized by $\x^i$ with no dependence on rival
decisions $\x^{-i}$. 

\medskip

\noindent {\large \bf 1.2. Focus of paper.} Our interest lies in the class of
hierarchical convex games,  where $\Xscr_i$ is a
closed and convex set in $\Real^{n_i}$, $g_i(\x^i,\x^{-i}) \triangleq \mathbb{E}[\tilde{g}_i(\x^i,\x^{-i},\omega)]$ is a convex function on
$\Xscr_i$, and  $f_i(\bullet, \x^{-i})$ is convex for every $\x^{-i} \in
\Xscr_{-i} \triangleq \prod_{j \neq i} \Xscr_j$. Our focus is on two subclasses
of such games.

\smallskip 

\noindent {\em (a) Monotone games.} Monotone games represent a subclass of hierarchical convex games,  where $T$ is a
monotone map on $\Xscr \triangleq \prod_{i=1}^\Nbold \Xscr_i$ and $ T(\x) \triangleq
\prod_{i=1}^\Nbold \partial_{\x^i} f_i(\x^i,\x^{-i}).$ Equilibria of this game
are entirely captured by the solution set of the monotone inclusion $0 \in
T(\x)$, where $T$ is expectation-valued.  Monotonicity of the game immediately
follows when the hierarchical term is ``private'' (i.e.  the lower-level
problem is independent of $\x^{-i}$) and $G$ is a monotone map where $G(\x)
\triangleq \prod_{i=1}^\Nbold \partial_{\x^i} g_i(\x^i,\x^{-i})$. However,
monotonicity also holds when the hierarchical term is not necessarily private
(cf.~\cite{demiguel2009stochastic}). We focus on developing techniques for
resolving monotone stochastic inclusions via Monte-Carlo sampling schemes, a
class of problems for which little is available when $T$ is both
expectation-valued and set-valued. 

\smallskip 

\noindent {\em (b) Potential games.} Potential games~\cite{monderer1996potential} represent a subclass of hierarchical convex games characterized by a potential function $P(\x)$ such that for any $i \in \{1,\cdots, \Nbold\}$ and any $\x^i \in \Xscr_i$,  
\begin{align*}
    P(\x^i, \x^{-i}) - P(\z^i,\x^{-i}) = f_i(\x^i,\x^{-i}) - f_i(\z^i,\x^{-i}), \mbox{ for all } \x^i, \z^i \in \Xscr_i. 
\end{align*}
Potentiality of the game immediately follows when the bilevel term is private and there exists a potential function $P$ such that $P(\x^i, \x^{-i}) - P(\z^i,\x^{-i}) = g_i(\x^i,\x^{-i}) - g_i(\z^i,\x^{-i})$ for all $ \x^i, \z^i \in \Xscr_i.$ In such cases, the original game has a potential function given by $P(\x) + \sum_{i=1}^{\Nbold} \mathbb{E}[\tilde{h}_i(\x^i,\y^i(\x^i,\omega),\omega)]$.  However, potentiality also follows in multi-leader multi-follower games where the bilevel term is not shared~\cite{kulkarni15existence,steffensen18quadratic}. Under a suitable potentiality assumption, we focus on developing efficient asynchronous best-response schemes for settings where $\y^i(\cdot,\omega)$ is a single-valued map.
\begin{tcolorbox} {\bf Research goal.} Our goal lies in developing provably convergent and efficient first-/zeroth-order schemes for computing equilibria for subclasses of games in (a) and (b).
\end{tcolorbox}

\noindent{\bf \large 1.3. Challenges and contributions.} Equilibria of
the most general forms of such games are challenging to compute, given the
inherent nonconvexity in player problems, the presence of expectations over
general measure spaces, and the lack of any underlying structure such as
potentiality or monotonicity. However, even when potentiality or monotonicity
of the Cartesian product of the subdifferential map of the implicit objectives,
computation of equilibria remains challenging for several reasons, some of which are specified next. (i) {\em
Expectation-valued and nonsmooth objectives and constraints.} Both the
objectives and constraints may be both nonsmooth and expectation-valued,
implying that standard projection-based schemes employed for deterministic
convex strategy sets cannot be employed.  (ii) {\em Hierarchical structure.}
The hierarchical structure significantly complicates the application of
available schemes. By replacing  the lower-level problem  using its necessary
and sufficient conditions leads to ill-posed nonlinear and nonconvex
optimization problems, i.e. mathematical programs with equilibrium constraints
(MPECs)~\cite{luo96mathematical}. While the implicit structure retains
convexity but resolving the resulting variational problem is complicated by the
presence of uncertainty and multi-valuedness. (iii) {\em Absence of structure
in inclusion problem.} Prior research on structured monotone inclusions has
relied on  single-valuedness and Lipschitz continuity in the expectation-valued
map, a property that is unavailable here. (iv) {\em Data privacy requirements.}
In some instances, player objectives and strategy sets are private and cannot
be shared, precluding the adoption of centralized schemes.   

\noindent {\bf Outline and contributions.} The remainder of the paper is
organized into five sections. In Section 2, we provide some preliminaries on
the hierarchical games of interest. Sections 3 and 4 focus on monotone and
potential variants of such games while Section 5 examines the numerical
behavior of the scheme. We conclude in Section 6. Our main contributions are
articulated next.

{\bf I. Monotone hierarchical games under uncertainty.}  In Section~\ref{sec:spp}, we present a
stochastic proximal-point framework for computing solutions to the necessary
and sufficient conditions of such games, compactly characterized by monotone
stochastic generalized equations where the operator is either strongly monotone
or maximal monotone. Notably, such claims can be extended to regimes with expectation-valued constraints under suitable properties.  In contrast with available stochastic proximal-point
schemes, we compute an inexact resolvent of the expectation
$T(\x) \triangleq \mathbb{E}[\Phi(\x,\y(\x,\omega),\omega)]$ via stochastic approximation. We show that when the
sample-size sequences are raised at a suitable rate, we prove that the
resulting sequence of iterates converges either at a linear rate (strongly
monotone) or at a rate of $\mathcal{O}(1/k)$ (maximal monotone) (in terms of a
suitable expectation-valued metric) matching the deterministic rate. This leads
to oracle complexities of $\mathcal{O}(1/\epsilon)$ under geometrically
increasing sample-sizes (strongly monotone) and
$\mathcal{O}(1/\epsilon^{2a+1})$ for $a>1$ when the sample-size is raised at
the rate of $\lceil (k+1)^{2a}\rceil$ (maximal monotone). These statements are
further supported by almost-sure convergence guarantees. Notable distinctions
with prior work are as follows: (i) The schemes achieve deterministic rates of
convergence, implying far better practical behaviour; (ii) The statements allow
for state-dependent noise, significantly widening the reach of such schemes;
and (iii) In strongly monotone regimes, the techniques allow for geometric rate
statements and optimal sample-complexities.

\smallskip

{\bf II. Potential hierarchical games under uncertainty.} We consider a
smoothing-based framework in which we consider the computation of equilibria of
an $\eta$-smoothed game characterized by a suitable potentiality requirement.
In fact, the equilibria of the original game can be related to that of the
smoothed game in terms of the best-response residual and under suitable conditions, limit points of the sequence of $\eta$-smoothed equilibria are equilibria of the original game. We then present an
asynchronous smoothed inexact proximal best-response scheme for computing an
equilibrium of the smoothed game. The scheme relies on leveraging a
zeroth-order scheme for computing an inexact best-reponse of a given player's
problem. In addition, we develop an relaxed counterpart where players average
between their current strategy and a best-response. Both the asynchronous
scheme and their relaxed counterpart are equipped with almost-sure convergence
guarantees to an approximate Nash equilibrium.

\smallskip

{\bf III. Numerical behavior.} Both sets of schemes are applied on a subclass
of multi-leader multi-follower games complicated by uncertainty. In monotone
settings, we observe that the proposed variance-reduced proximal-point schemes
provide solutions of superior accuracy in a fraction of the time required by
standard stochastic approximation schemes. Under a potentiality assumption,
both the asynchronous relaxed inexact smoothed best-response scheme and its
relaxed counterpart display convergent behavior but the relaxed scheme displays
a higher degree of stability later in the process.  

\section{Preliminaries} 
In this section, we begin by providing some preliminaries on bilevel convex
games in Section~\ref{sec:2.1}, followed by a description of hierarchical
monotone and potential games in Section~\ref{sec:2.2}. We conclude with a more
elaborate description of two prototypical hierarchical convex games in Section~\ref{sec:2.3}
where monotonicity and potentiality properties are highlighted. In Section~\ref{sec:2.4}, we conclude with a brief commentary on the assumptions of single-valuedness of $\y(\bullet,\omega)$ and the convexity of the implicit player-specific objective and provide a preliminary literature survey to show that these assumptions have relatively broad applicability. 

\subsection{A taxonomy of bilevel convex games} \label{sec:2.1}
We begin by considering a single-level $\Nbold$-player noncooperative game in which the $i$th player solves a parametrized optimization problem given by  
\begin{align} \tag{Agent$_i(\x^{-i})$} 
    \min_{\x^i \in \Xscr_i} \ {g_i}(\x^i,\x^{-i}), 
    \end{align}
    where $\Xscr_i \subseteq \Real^n$ is a closed and convex set and ${g_i}(\bullet,\x^{-i})$ is convex on $\Xscr_i$ for any $\x^{-i} \in \Xscr_{-i}$. We denote the class of convex single-level games by $\Gscr^{\rm csl}$.  It may be recalled that an $\epsilon$-Nash equilibrium is given by a tuple $\x^*_{\epsilon} \triangleq \{\x^{1,*}, \cdots, \x^{\Nbold,*}\}$ such that for $i = 1, \cdots, \Nbold$,  
    \begin{align} \tag{$\epsilon$-NE}
        {g_i}(\x^{i,*},\x^{-i,*}) \leq {g_i}(\x^i,\x^{-i,*}) + \epsilon, \qquad \forall \ \x^i \ \in \ \Xscr_i.
    \end{align}
    When $\epsilon = 0$, $\x^*_{\epsilon}$ reduces to the standard Nash equilibrium. In a hierarchical generalization of this game, the $i$th player's objective is modified by the addition of a term $h_i(\x^i,\y^i(\x))$ where $\y^i(\x)$ represents a solution to a lower-level variational inequality problem VI$(\Yscr_i,F_i(\x,\bullet))$. This problem requires a vector ${{\y}}^i \in \Yscr_i$ that satisfies 
    \begin{align}
        \tag{VI$(\Yscr_i,F_i(\x,\bullet))$} 
        {(\widehat{\y}^i - \y^i)^\mathsf{T} F_i(\x,\y^i)} \geq 0, \quad \forall \ \widehat{\y}^i \ \in \ \Yscr_i. 
    \end{align}
Consequently, the $i$th player in a hierarchical game solves the following parametrized problem.
\begin{align} \tag{Hier-Agent$_i(\x^{-i})$} 
    \min_{\x^i \in \Xscr_i} \ {f_i(\x^i,\x^{-i})\triangleq g_i(\x^i,\x^{-i})} + h_i(\x^i,\y^i(\x)). 
    \end{align}
We denote the subclass of hierarchical convex games by $\Gscr^{\rm chl}$ while
monotone and potential variants are referred to as $\Gscr^{\rm chl}_{\rm mon}$
and $\Gscr^{\rm chl}_{\rm pot}$, respectively. Each of these subclasses is
discussed in greater detail. Before proceeding, we make a well-posedness assumption on the existence of an equilibrium.  

\medskip 

\begin{tcolorbox}{\bf Ground Assumption (G1)}
Throughout this paper, we assume that the hierarchical convex game admits a Nash equilibrium.
\end{tcolorbox}

\smallskip

Naturally, there are instances when multi-leader multi-follower games fail to
admit an equilibrium. Pang and Fukushima~\cite{pang05quasi} provide precisely
such an instance. However, when player objectives are convex, given rival
decisions, and  strategy sets are compact, existence of equilibria follows from
fixed-point arguments~\cite{ichiishi83game}. Absent convexity, existence of equilibria in multi-leader multi-follower
games is more challenging to show and potentiality arguments have been adopted
to show that existence follows if a suitable optimization problem is
solvable~\cite{kulkarni15existence}. 

\subsection{Hierarchical monotone and potential
games}\label{sec:2.2}
We now consider the class of monotone convex single-level games (see Appendix {\bf A.1.} for a description of games, variational inequality problems, inclusions, and monotonicity), denoted by $\Gscr^{\rm csl}_{\rm mon}$. Any element of this class is characterized by monotonicity of the map $G$ on $\Xscr$ where 
$G(\x) \triangleq \prod_{i=1}^{\Nbold} \partial_{\x^i} g_i(\x)$ and $\prod_{i=1}^{\Nbold} \Xscr_i$ denotes the Cartesian product of sets $\Xscr_1, \cdots, \Xscr_{\Nbold}$.
In fact, we may relate a game $\Gcal \in \Gscr^{\rm csl}_{\rm mon}$ with a counterpart $\widehat{\Gcal} \in \Gscr^{\rm chl}_{\rm mon}$, where $\Gscr^{\rm chl}_{\rm mon}$ denotes the subclass of hierarchical convex games with a monotone map. The proofs are provided in Appendix.
\begin{proposition} \label{mono-g} \em Consider a game $\Gcal \in \Gscr^{\rm csl}_{\rm mon}$ where the $i$th player solves (Agent$_i(\x^{-i})$) for  $1, \cdots, \Nbold$ and $\y^i(\bullet)$ is a single-valued map. Furthermore, consider a game $\widehat{\Gcal}$ where the $i$th player solves (Hier-Agent$_i(\x^{-i})$) for  $1, \cdots, \Nbold$.  

    \noindent (a) Suppose $h_i$ is convex and $h_i(\x^i,\y^i(\x)) = h_i(\x^i,\y^i(\x^i))$ for $i = 1, \cdots, \Nbold$, i.e. the hierarchical term is private. 

    \noindent (b) Suppose $h_i(\x^i,\y^i(\x)) = {h}(\x)$ for $i = 1, \cdots, \Nbold$  where ${h}$ is a convex function on $\Xscr$, i.e. the hierarchical term is common across all players. 

\noindent Then we have that $\widehat{\Gcal} \in \Gscr^{\rm chl}_{\rm mon}$.
\end{proposition}

In short, monotone single-level games often induce monotone hierarchical games
when the hierarchical structure emerges in a particular fashion.  Next, we
consider the class of potential convex single-level games, denoted by
$\Gscr^{\rm csl}_{\rm pot}$. Corresponding to an element $\Gcal \in \Gscr^{\rm
csl}_{\rm pot}$ is a potential function $P(\x)$ such that for any given  $i \in
\{1, \cdots, \Nbold\}$,   we have  
\begin{align*}
    P(\x^i,\x^{-i}) - P(\tilde{\x}^i,\x^{-i}) = {g_i}(\x^i,\x^{-i}) - {g_i}(\tilde{\x}^i,\x^{-i}),
\end{align*}
for any $\x^i,\tilde{\x}^i \in \Xscr_i$ and any $\x^{-i} \in \Xscr_{-i}$. We may then develop a relationship between $\Gcal$ and a related game $\widehat{\Gcal}$ that lies in the subclass of hierarchical convex potential games.  
\begin{proposition} \label{pot-g} \em Consider a game $\Gcal \in \Gscr^{\rm csl}_{\rm pot}$ where the $i$th player solves (Agent$_i(\x^{-i})$) for  $1, \cdots, \Nbold$ and $P$ denotes its potential function. Furthermore, consider a game $\widehat{\Gcal}$ where the $i$th player solves (Hier-Agent$_i(\x^{-i})$) for  $1, \cdots, \Nbold$.  Then the following hold.

    \noindent (a) Suppose $h_i(\x^i,\y^i(\x)) = h_i(\x^i,\y^i(\x^i))$, i.e. the hierarchical term is private. Then $\widehat{P}(\x) \triangleq P(\x) + \sum_{i=1}^{\Nbold} h_i(\x^i,\y^i(\x^i))$ for any $\x \in \Xscr.$     

    \noindent (b) Suppose $h_i(\x^i,\y^i(\x)) = {h}(\x)$ where $h$ is a convex function on $\Xscr$, i.e. the hierarchical term is common across all players. Then $\widehat{P}(\x) \triangleq P(\x) + {h}(\x)$ for any $\x \in \Xscr.$

    \noindent Then $\widehat{\Gcal} \in \Gscr^{\rm chl}_{\rm pot}$ with an associated potential function $\widehat{P}$.
\end{proposition}

\subsection{Applications} \label{sec:2.3}

\noindent {\em (a) A subclass of monotone stochastic bilevel games.} 
Consider a class of bilevel games with $\Nbold$ players, denoted by $\Nscr \triangleq \{1, \cdots, \Nbold \}$, where each player's objective has two terms, the first of which is parametrized by rival decisions $\x^{-i}$ while the second is independent of rival decisions.  In
general, this class of games is challenging to analyze since the player
problems are nonconvex. Adopting a similar approach in examining a stochastic
generalization of a quadratic setting examined
in~\cite{steffensen18quadratic} with a single follower where $\Mbold = 1$, suppose the $\omega$-specific lower-level problem corresponding to player $i$ is 
\begin{align}\tag{\mbox{Lower$_i(\x^{-i},\omega)$}}
\min_{\y^i \, \geq \, \ell_i(\x^i,\omega)} & \quad \half (\y^i)^{\mathsf{T}} Q_i(\omega) \y^i - b_i(\x^i,\omega)^\mathsf{T} \y^i,\end{align}
where $Q_i(\omega)$ is a positive definite and diagonal matrix for every $\omega \in \Omega$, $b_i(\bullet,\omega)$ and $\ell_i(\bullet,\omega)$ are affine functions for every $\omega \in \Omega$.
Suppose the leaders compete in a noncooperative game in which the $i$th leader  solves 
\begin{align*}
    \min_{x_i \, \in \, \mathcal{X}_i} & \quad \mathbb{E}\left[\tilde{{g}}_i(\x^i,\x^{-i},\omega) + a_i(\omega)^\mathsf{T} \y^i(\x^i,\omega)\right], 
\end{align*}
where $\tilde{{g}}(\bullet,\x^{-i},\omega)$ is a convex function, $\y^i(\x^i,\omega)$ denotes the unique lower-level solution given the upper-level decision $\x^i$ and realization $\omega$ associated with player $i$ and scenario $\omega$, and $\mathcal{X}_i$ is a
closed and convex set in $\Real$. The lower-level solution set associated with player $i$ is denoted by $\y^i(\x^i,\omega)$, given leader-level decisions $\x^i$, can be derived by considering the necessary
and sufficient conditions of optimality: \begin{align*}
	\left\{\begin{aligned} 0 & \leq \lambda_i  \perp \y^i - \ell_i(\x^i,\omega) \geq 0 \\
	0 & =  Q_i(\omega) \y^i - b_i(\x^i,\omega) - \lambda_i
\end{aligned}\right\} \, \mbox{or} \, \left\{\y^i(\x^i,\omega) \triangleq \max\left\{ Q_i(\omega)^{-1} b_i(\x^i,\omega), \ell_i(\x^i,\omega)\right\}\right\}. 
\end{align*}
We may then eliminate the lower-level decision in the player's problem, leading to a nonsmooth stochastic Nash equilibrium problem given by the following:
\begin{align*}
    \min_{\x^i \in \mathcal{X}_i} & \mathbb{E}\left[ \tilde{{g}}_i(\x^i,\x^{-i},\omega) + \underbrace{a_i(\omega)^\mathsf{T} \max\{Q_i(\omega)^{-1}b_i(\x^i,\omega), \ell_i(\x^i,\omega)\}}_{\ \triangleq  \ {\tilde{h}}_i(\x^i,\omega)}\right]. 
\end{align*}
Under suitable assumptions ${\tilde{h}}_i(\bullet,\omega)$ is a convex function. For instance, it suffices if  $b_i(\bullet,\omega)$ and $\ell_i(\bullet,\omega)$ are convex for every $\omega$,  $Q_i(\omega)$ is a positive diagonal matrix and $a_i(\omega)$ is a nonnegative vector for every $\omega$. This follows  from observing that ${\tilde{h}}_i(\x^i,\omega)$  is a scaling of the maximum of two convex functions. Consequently, the necessary and sufficient equilibrium conditions of this game are given by 
$0 \in \partial_{\x^i} \mathbb{E}[\tilde{{g}}_i(\x^i,\x^{-i},\omega) + {\tilde{h}}_i(\x^i,\omega)] + \Nscr_{X_i}(\x^i)$ for $i = 1, \hdots, {\Nbold}.$ 
Since ${\tilde{h}}_i(\bullet,\omega)$ is a convex function in $\x^i$ for every $\omega$, then the necessary and sufficient equilibrium conditions are given by $0 \in T(\x)$, where
\begin{align}\tag{SGE-a}
    T(\x)  & \triangleq  \mathbb{E}\left[\Phi(\x,\omega)\right],  \\
    \notag \mbox{ where } \Phi(\x,\omega) & \triangleq \prod_{i=1}^{\Nbold} \left[\partial_{\x^i} \left[\tilde{{g}}_i(\x^i,\x^{-i},\omega) + {\tilde{h}}_i(\x^i,\omega)\right]+ \mathcal{N}_{\Xscr_i}(\x^i)\right]. 
\end{align}

\noindent (i) {\em Monotonicity of game.}  Suppose ${g}_i(\x^i,\x^{-i}) \triangleq \mathbb{E}[\tilde{{g}}_i(\x^i,\x^{-i})]$ and ${g}_i(\bullet,\x^{-i})$ is convex and C$^1$ on an open set containing $\Xscr_i$ and ${G}$ is a monotone map on $\Xscr$, where 
$$ {G}(\x) \triangleq \pmat{ \nabla_{\x^1} {g}_1(\x^1,\x^{-1}) \\
                            \vdots \\
                        \nabla_{\x^\Nbold} {g}_N(\x^\Nbold,\x^{-\Nbold})}. $$
Then the resulting hierarchical game is qualified as monotone. 

\noindent (ii) {\em Potentiality of game.} If the collection of objectives ${g}_1, \cdots, {g}_\Nbold$ admit a potential function $P(\x)$ satisfying the following for any $i \in \Nscr.$ 
$$ P(\x^i,\x^{-i}) - P(\z^i, \x^{-i}) = {g}_i(\x^i,\x^{-i}) - {g}_i(\z^i,\x^{-i}), \mbox{ for any } \x^i, \z^i \in \Xscr_i.$$
Then the hierarchical convex game has an associated potential function given by $\tilde{P}(\x) = P(\x) + \sum_{i=1}^\Nbold \mathbb{E}[{\tilde{h}}_i(\x^i, \omega)]$ and may be qualified as a potential game.  \\

\noindent {\em (b) A multi-leader stochastic Stackelberg-Nash-Cournot
equilibrium problem.} Consider an oligopolistic setting with $\Mbold+\Nbold$
firms where $\Mbold$ followers compete in a noncooperative game while $\Nbold$
leaders compete in a game subject to the equilibrium decisions of the
followers~\cite{sherali1984multiple,pang05quasi,su2007analysis,kulkarni15existence}.
Suppose the $j$th follower solves the following parametrized problem.
\begin{align}\tag{\mbox{Follower$_j(X,\y^{-j},\omega)$}}
\max_{\y^j \, \geq \, 0 } & \quad \left(p(\y^j + Y^{-j}+X,\omega) \y^j  - c_j(\y^j)\right), \end{align}
where $X = \x^i + \x^{-i}$. 
Suppose the inverse demand function $p(\cdot,\omega)$ is defined as $p(u,\omega) = a(\omega) - b(\omega)u$. Under this condition, the follower's objective can be shown to be strictly concave in $\y^j$~\cite{demiguel2009stochastic}. Consequently, the concatenated necessary and sufficient equilibrium conditions of the lower-level game are given by the following conditions.
\begin{align} \tag{Equil$_{\rm foll}(X,\omega)$}
     \begin{aligned}
         0 & \leq \y & \perp \nabla_{\y} c(\y) - p(X+Y,\omega) {\bf 1} - p'(X+Y,\omega) \y   \geq 0.
\end{aligned} 
\end{align}
We observe that (Equil$_{\rm foll}(X,\omega)$) is a strongly monotone variational inequality problem for $X \geq 0$ and for every $\omega \in \Omega$. Consequently, $\y: \Real_+ \times \Omega \to \Real_+^{\Mbold}$ is a single-valued map and is convex in its first argument for every $\omega$ if $c_j$ is quadratic and convex for $j = 1, \cdots, \Mbold$~\cite[Prop.~4.2]{demiguel2009stochastic}. In fact, it can be claimed that $\y(\cdot,\omega)$ is a piecewise C$^2$ and non-increasing function with $\partial_{\x^i} \y(X,\omega) \subset (-1,0]$ for $X \geq 0$. Consider the $i$th leader's problem, defined as 
\begin{align}\tag{Leader$_i(\x^{-i})$}
\max_{\x^i \geq 0} \ \left[\mathbb{E}\left[ p(\x^i+X^{-i} + Y(\x^i+X^{-i},\omega),\omega)\x^i \right] - C_i(\x^i) \right].
\end{align}
Consequently, we have that 
\begin{align*}
    0 \ni \x^i & \perp \mathbb{E}\left[-p (\x^i+X^{-i} + Y(\x^i+X^{-i},\omega),\omega) + (1 +  \partial_{\x^i} Y(X,\omega)) b(\omega) \x^i\right] \\
            & + \nabla_{\x^i} C_i(\x^i) \in 0. 
\end{align*}
By concatenating the problems for players $1, \cdots, \Nbold$, we obtain the following complementarity problem. 
\begin{align*}
    0 \ni \x & \perp \mathbb{E}\left[-p (X+ Y(X,\omega),\omega){\bf 1} \right] 
    + \pmat{\nabla_{\x^i} C_i(\x^i)}_{i=1}^\Nbold + \prod_{i=1}^{\Nbold} \{\mathbb{E}[(1 +  \partial_{\x^i} Y(X,\omega)) b(\omega) \x^i]\} \in 0. 
\end{align*}
This may be viewed as the following inclusion:
\begin{align*}
0  \in  T(\x) & \triangleq \mathbb{E}[\Phi(\x,\omega)], \\ 
        \mbox{ where } \Phi(\x,\omega) & \triangleq -p (X+ Y(X,\omega),\omega){\bf 1} 
                                     + \pmat{\nabla_{\x^i} C_i(\x^i)}_{i=1}^\Nbold + \prod_{i=1}^{\Nbold} \left\{[(1 +  \partial_{\x^i} Y(X,\omega)) b(\omega) \x^i]\right\} + \mathcal{N}_{\Real_n^+}.
\end{align*}
This map has been proven to be monotone in~\cite[Thm.~4.1]{demiguel2009stochastic}.

\subsection{A comment on the assumptions} \label{sec:2.4}
In this subsection, we briefly comment on the assumptions of uniqueness of lower-level problems and the convexity of the resulting ``implicit'' player-specific objective, focusing on the challenges associated with weakening these assumptions. Naturally, one may inquire as to whether such assumptions are far too restrictive to be employed in practice. We conclude this section with a preliminary literature survey where it can be observed that this is not the case and such avenues have found broad applicability across a range of settings. \\

\noindent (a) {\bf Uniqueness of $\y(\bullet,\omega)$.} As noted, we have imposed a suitably monotonicity requirement on the lower-level parametrized variational inequality problem that allows for claiming the uniqueness of the lower-level problem for a given $\x$ and $\omega$. Absent such an assumption, the player problem can be modeled either optimistically or pessimistically as follows~\cite{luo96mathematical}. 
            \begin{align*} \mbox{(Player$^{\rm optim}_{i}(\x^{-i})$)}
                & \ \left\{\begin{aligned}
                \min_{\x^i \in \Xscr^i} \min_{\y} &  \quad f_i(\x^i,\x^{-i}, \y) \\
                \st & \quad \y \in \mbox{SOL}(\Yscr, F(\x, \bullet)) 
        \end{aligned}\right\}. \\ 
\mbox{(Player$^{\rm pessim}_{i}(\x^{-i})$)}
                & \ \left\{\begin{aligned}
                \min_{\x^i \in \Xscr^i} \max_{\y} &  \quad f_i(\x^i,\x^{-i}, \y) \\
                \st & \quad \y \in \mbox{SOL}(\Yscr, F(\x, \bullet)) 
        \end{aligned}\right\}.  
        \end{align*}
        When $\y(\bullet,\omega)$ is a single-valued map, the above two
        problems coincide, but in general, both of the above parametrized
        problems are in general nonconvex optimization problems, falling within
        the category of mathematical programs with equilibrium
        constraints(MPECs)~\cite{luo96mathematical}. In such instances, the
        original hierarchical game reduces to a noncooperative game in which
        each player solves a mathematical program with equilibrium constraints
        (a nonconvex program).  Existence of equilibria to the original
        hierarchical game is not guaranteed
        (see~\cite{pang05quasi,kulkarni15existence} for simple instances where
        equilibria fail to exist) and there are no clean tractable conditions
        for expressing ``global'' Nash equilibria.  One could naturally
        ``regularize'' the lower-level when the map $F(\x,\bullet)$ is monotone
        on $\Yscr$ for every $\x$. However, such avenues need far more study
        since it has been discovered that regularized trajectories are not
        guaranteed to converge as noted in~\cite{caruso20regularization}. \\

        \noindent (b) {\bf Convexity of player problems.} We impose a convexity assumption on the
            implicit upper-level objective $f(\bullet,\y(\bullet))$ where $\y(\x)$
            is a unique solution to the lower-level problem given upper-level
            decision $x$. There are several issues with weakening convexity for
            the player problems.
\begin{enumerate}
\item[(i)] {\em Absence of equilibrium conditions for nonconvex games.} First, a Nash equilibrium is defined at a
            set of player-specific decisions at which no player has an
            incentive to deviate. Naturally, in nonconvex regimes, this
            requires that each player is at her global minimum, given rival
            decisions.  Yet, in general, there are no tractable equilibrium
            conditions for such a point and absent significant structure, we
            believe that the computation of such equilibria, while compelling
            and relevant, is currently out of reach.
\item[(ii)] {\em Nash-stationary equilibria.} Second, one could
            naturally employ stationarity conditions but the resulting
            solutions cannot be guaranteed to be equilibria. Our focus in this
        paper is {\bf on global equilibria and not Nash stationary equilibria}
    (as defined in~\cite{pang11nonconvex}).\\
\end{enumerate}
\noindent {\bf Applications.} Third, the assumptions imposed in (a) and (b) might be viewed as far too restrictive in practice. We believe that this may not be the case. In particular, the presence of lower-level uniqueness and upper-level convexity (in an implicit sense) are far more widespread and occur in a wide range of applications. Table~\ref{app:hier_games} provides a subset of such applications where such models have found applicability and it is seen that both (a) and (b) are seen to hold in the setting of interest. \\  
\begin{table}
    \scriptsize
    \centering
    \begin{tabular}{ p{2in}|c|c| c} 
        \hline 
        Topic & Uniqueness of $\y(\bullet,\omega)$ & Convexity of $f_i(\bullet, \x^{-i})$ & References \\ \hline  
 \hline Hierarchical Cournot games & $\checkmark$ & $\checkmark$ & ~\cite{sherali1984multiple,kulkarni14shared} \\ 
 \hline  Strategic behavior  in power markets & $\checkmark$& $\checkmark$& ~\cite{allaz93cournot,su2007analysis,shanbhag11complementary,demiguel2009stochastic,hu07epecs,hu2013existence} \\ 
 \hline Telecommunication markets& $\checkmark$& $\checkmark$& ~\cite{demiguel2009stochastic,hu2015multi}\\ 
 \hline Global emission control& $\checkmark$& $\checkmark$& ~\cite{mallozzi17multi}\\ 
 \hline Supply-chain networks& $\checkmark$& $\checkmark$& ~\cite{leleno1992leader} \\  
 \hline Generation capacity expansion games& $\checkmark$& $\checkmark$& ~\cite{murphy2005generation,wogrin2013open} \\ 
 \hline Gas markets & $\checkmark$ & $\checkmark$ & ~\cite{wolf97stochastic} \\
 \hline
\end{tabular}
\caption{Applications of hierarchical games}
\label{app:hier_games}
\end{table}


\section{VR proximal-point schemes for stochastic hierarchical monotone games}\label{sec:spp}
In this section, we will consider the class of stochastic hierarchical
monotone games. In Section~\ref{sec:3.1}, we discuss how equilibrium conditions
of such games can be recast as inclusions in settings with and without expectation-valued constraints. An efficient
variance-reduced proximal scheme is developed for computing equilibria of such
games in Section~\ref{sec:3.2} (via resolving the associated inclusions). In
Section~\ref{sec:3.3}--\ref{sec:3.4}, we conclude this section with a discussion of the
convergence theory and rate statements for such schemes in monotone and
strongly monotone regimes and conclude with a comment on the broader applicability of the framework for monotone inclusions in Section~\ref{sec:3.5}.  

\subsection{Hierarchical games and monotone inclusions}\label{sec:3.1}
We recall the $\Nbold$-player convex hierarchical game $\Gcal \in \Gscr^{\rm chl}$ of interest in which the $i$th player solves the following hierarchical optimization problem parametrized by $\x^{-i}$. 
\begin{align}  \tag{Player$_i(\x^{-i})$}
    \begin{aligned}
    \min_{\x^i} & \quad\mathbb{E}[ \tilde{f}_i(\x^i,\y^i(\x,\omega), \x^{-i},\omega)]\\ 
        \st & \quad \x^i \in \Xscr_i,
\end{aligned}
\end{align}
where $\tilde{f}_i(\x^i,\y^i(\x,\omega), \x^{-i},\omega) \triangleq \tilde{g}_i(\x^i,\x^{-i},\xi(\omega)) +\tilde{h}_i(\x^i,\y^i(\x,\xi(\omega)),\omega)$. Under the assumption that for any $\omega \in \Omega$, $\tilde{f}_i(\x^i,\x^{-i},\y(\x,\omega),\omega)$ is convex in $\x^i$ over $\Xscr_i$ for any $\x^{-i} \in \Xscr^{-i} \triangleq \prod_{j\neq i} \Xscr_j$. Consequently, the necessary and sufficient conditions of the game are compactly captured by an inclusion problem. This is formalized next.

\begin{proposition}[{\bf Equivalence to a stochastic inclusion problem}]\rm  Consider a $\Nbold$-player game $\Gcal \in \Gscr^{\rm chl}$ in which the $i$th player solves the parametrized problem (Player$_i(\x^{-i})$) for $i = 1, \cdots, \Nbold.$ Then $\x^* \triangleq \{\x^{1,*}, \cdots, \x^{\Nbold,*}\}$ is an equilibrium of $\Gcal$ if and only if 
\begin{align}
	\notag 0 \in T(\x^*) &\triangleq \mathbb{E}\left[\Phi(\x^*,\y(\x^*,\omega),\omega)\right], \\
	 \mbox{where }\Phi(\x^*,\y(\x^*,\omega),\omega) &\triangleq \prod_{i=1}^{\Nbold} \left[ \partial_{\x^i} \tilde{f}_i(\x^{i,*},\x^{-i,*},\y^i(\x^*,\omega),\omega) + \mathcal{N}_{\Xscr_i}(\x^{i,*})\right]. \label{def-T}
\end{align}   
\end{proposition}
\begin{proof}
By the convexity of the player-specific problems, $\x^* \triangleq \{\x^{1,*}, \cdots, \x^{\Nbold,*}\}$ is an equilibrium of $\Gcal$ if and only if  $\x^* \triangleq \{\x^{1,*}, \cdots, \x^{\Nbold,*}\}$ collectively solves this set of generalized equations 
\begin{align}\left\{
\begin{aligned}
0&  \in \partial_{\x^1} \left[\mathbb{E}[\tilde{f}_1(\x^1,\x^{-1},\y^1(\x,\omega),\omega)]\right] + \mathcal{N}_{\Xscr_1}(\x^1) \\
 & \qquad \quad \vdots \\
0&  \in \partial_{\x^\Nbold} \left[\mathbb{E}[\tilde{f}_{\Nbold}(\x^{\Nbold},\x^{-{\Nbold}},\y^\Nbold(\x,\omega),\omega)]\right]  + \mathcal{N}_{\Xscr_{\Nbold}}(\x^{\Nbold})
\end{aligned}\right\}.\label{SGE_N}
\end{align}
Via~\cite[Prop. 1.4.2]{facchinei2007finite}, it can then be shown that $\x^* \triangleq \{\x^{1,*}, \cdots, \x^{\Nbold,*}\}$ is a solution of \eqref{SGE_N} if and only if $\x^* \triangleq \{\x^{1,*}, \cdots, \x^{\Nbold,*}\}$ is a solution of 
\begin{align}
0 \in T(\x^*) &\triangleq \mathbb{E}\left[\Phi(\x^*,\y(\x^*,\omega),\omega)\right].
\notag
\end{align}
\end{proof}

We now consider the extension of such problems where the player problems
    have private expectation-valued constraints; in particular, suppose the
    $i$th player solves the expectation-valued constrained counterpart of
    (Player$_i(\x^{-i})$), denoted by (Player$^{\rm con}_i(\x^{-i})$),  for
    $i=1, \cdots, \Nbold$. We define (Player$^{\rm con}_i(\x^{-i})$) as follows.
    \begin{align}  \tag{Player$^{\rm con}_i(\x^{-i})$}
    \begin{aligned}
    \min_{\x^i \in \Xscr_i} & \quad\mathbb{E}\left[ \tilde{f}_i(\x^i,\y^i(\x,\omega), \x^{-i},\omega)\right]\\ 
        \st & \quad  \mathbb{E}\left[\tilde{c}_i(\x^i,\omega)\right] \leq 0. 
\end{aligned}
\end{align}
Under the additional assumption that $\mathbb{E}\left[\tilde{c}_i(\bullet,\omega)\right]$ is convex in $\x^i$ on $\Xscr_i$ and under a suitable regularity condition, $\x^{i,*}$ is an optimal solution of (Player$_i^{\rm con}(\x^{-i})$) if and only if $\{\x^{i,*}, \p^{i,*}\}$ is a primal-dual solution of the following system. 
\begin{align*}\left\{
\begin{aligned}
0&  \in \partial_{\x^i} \left[\mathbb{E}[\tilde{f}_i(\x^{i,*},\x^{-i},\y^i(\x^{i,*},\x^{-i},\omega),\omega)]\right] + \partial_{\x^i} \left[\mathbb{E}[\tilde{c}_i(\x^{i,*},\omega)\right]^T\p^{i,*} + \mathcal{N}_{\Xscr_i}(\x^i) \\
0 & \in -\mathbb{E}[\tilde{c}_i(\x^{i,*},\omega)] + \mathcal{N}_{\Real_{m_i}^+}(\p^{i,*}). 
\end{aligned}\right\}.
\end{align*}
This allows us to restate the necessary and sufficient equilibrium conditions of the hierarchical game with private expectation-valued constraints as follows.
\begin{align*}\left\{
\begin{aligned}
0&  \in \partial_{\x^1} \left[\mathbb{E}[\tilde{f}_1(\x^1,\x^{-1},\y^1(\x,\omega),\omega)]\right] + \partial_{\x^1} \left[\mathbb{E}[\tilde{c}_1(\x^{1},\omega)\right]^\mathsf{T}\p^{1}+ \mathcal{N}_{\Xscr_1}(\x^1) \\
0 & \in -\mathbb{E}[\tilde{c}_1(\x^{1},\omega)] + \mathcal{N}_{\Real_{m_1}^+}(\p^{1}) \\ 
 & \qquad \quad \vdots \\
0&  \in \partial_{\x^\Nbold} \left[\mathbb{E}[\tilde{f}_{\Nbold}(\x^{\Nbold},\x^{-{\Nbold}},\y^\Nbold(\x,\omega),\omega)]\right] + \partial_{\x^{\Nbold}} \left[\mathbb{E}[\tilde{c}_{\Nbold}(\x^{{\Nbold}},\omega)\right]^\mathsf{T}\p^{{\Nbold}} + \mathcal{N}_{\Xscr_{\Nbold}}(\x^{\Nbold}) \\
0 & \in -\mathbb{E}[\tilde{c}_{\Nbold}(\x^{{\Nbold}},\omega)] + \mathcal{N}_{\Real_{m_{\Nbold}}^+}(\p^{{\Nbold}}) 
\end{aligned}\right\}.
\end{align*}
 
\begin{proposition}[{\bf Hierarchical games with expectation-valued constraints and stochastic inclusions}]\rm  Consider a $\Nbold$-player game $\Gcal \in \Gscr^{\rm chl}$ in which the $i$th player solves the parametrized problem (Player$^{\rm con}_i(\x^{-i})$) for $i = 1, \cdots, \Nbold.$ For $i=1, \cdots, \Nbold$, suppose a regularity condition holds for player's problem at $\x^{i,*}$, given $\x^{-i,*}$. Then $\x^*$ is an equilibrium of $\Gcal$ if and only if $\{\x^*,\p^*\}$ is the solution of the following inclusion problem   
\begin{align*}
	0 \in \Psi(\x^*,\p^*) \triangleq \mathbb{E}\left[\Lambda(\x^*,\y(\x^*,\omega),\p^*, \omega)\right],
\end{align*}   
where $\Lambda(\x,\y(\x,\omega),\p,\omega)$ is defined as 
\begin{align*}\notag
\Lambda(\x,\y(\x,\omega),\p,\omega) & \triangleq \prod_{i=1}^{\Nbold} \left\{ \partial_{\x^i} \left[\mathbb{E}[\tilde{f}_{i}(\x^i,\x^{-i},\y^i(\x,\omega),\omega)]\right] + \partial_{\x^{i}} \left[\mathbb{E}[\tilde{c}_{i}(\x^{i},\omega)\right]^\mathsf{T}\p^{i} + \mathcal{N}_{\Xscr_{i}}(\x^i)\right\} \times \\
				& \quad \prod_{i=1}^{\Nbold}  
\left\{-\mathbb{E}[\tilde{c}_i(\x^{i},\omega)] + \mathcal{N}_{\Real_{m_i}^+}(\p^{i})\right\},
\end{align*} 
$\x^* \triangleq \{\x^{1,*}, \cdots, \x^{\Nbold,*}\}$, and $\p^*\triangleq \{\p^{1,*}, \cdots, \p^{\Nbold,*}\}$, respectively.
\end{proposition}
Recall that in general, the map $\Phi$ (and $\Lambda$) is not necessarily monotone.
However, there are many instances both in non-hierarchical~\cite{Scutari2012,scutari14real} and hierarchical~\cite{hu2013existence,hu2011variational,demiguel2009stochastic,steffensen18quadratic,su2007analysis,sherali1984multiple} regimes where monotonicity of $\Phi$ (or its single-valued variant) does
    indeed hold and that represents our focus. However, resolving such
    inclusion problems is by no means a simple propoposition since  $\Phi$ is
    an expectation-valued and possibly set-valued monotone map. Unfortunately,
    there are no efficient existing schemes in general settings for resolving
    such problems and we present a variance-reduced proximal-point framework
for this problem.

\subsection{Variance-reduced proximal-point framework for hierarchical monotone games}\label{sec:3.2} 
In this subsection, we present a variance-reduced
proximal-point method for stochastic inclusion problems of
the form $0 \in T(\x) = \mathbb{E}[\Phi(\x,\y(\x,\omega),\omega)]$ where $\Phi(\x,\y(\x,\omega),\omega)$ is defined in \eqref{def-T}.
Throughout this subsection, $\x^*$ denotes a solution of $0 \in T(\x)$, implying that $0 \in T(\x^*)$ or $\x^*\in T^{-1}(0)$.
Deterministic proximal-point methods require computing 
$(I+\lambda T)^{-1}$ at every step, a challenging proposition since the expectation is 
unavailable in closed form. Our scheme retains the
expectation-valued $T(\x)$ in the resolvent operator, which is \us{subsequently} approximated via Monte-Carlo sampling, leading to an error; in effect, we articulate the {\em resolvent} problem then utilize {\em sampling} to get an approximation. \us{Given $\x^0 \in \Real_n$, ({\bf VR-SPP}) generates a sequence $\{\x^k\}$, where $\x^{k+1}$ is updated as}
\begin{tcolorbox} {\bf Variance-reduced proximal-point method}
\begin{align*} \tag{{\bf VR-SPP}}
\x^{k+1}\coloneqq (I+\lambda_kT)^{-1}(\x^k)+e_k,
\end{align*}
\end{tcolorbox}
where $e_k$ denotes the random error in computing the resolvent $(I+\lambda_kT)^{-1}$ when employing Monte-Carlo sampling schemes. We review some preliminary results and assumptions in Sections~\ref{sec:3.2.1} and~\ref{sec:3.2.2}, respectively and then discuss a player-specific stochastic approximation framework for computing an inexact resolvent in Section~\ref{sec:3.2.3}. Subsequently, we analyze {\bf (VR-SPP)} for maximal monotone and strongly monotone regimes in Sections~\ref{sec:3.3} and ~\ref{sec:3.4}, respectively.

\subsubsection{Preliminaries on proximal-point schemes}\label{sec:3.2.1} 
Consider the generalized equation 
\begin{align}\tag{SGE}
 0 \in T(\x) \triangleq \mathbb{E}[\Phi(\x,\y(\x,\omega),\omega)], \label{SGE}
\end{align} 
where $T$ is a set-valued maximal monotone map and $\Phi(\bullet,\y(\bullet,\omega),\omega)$ is defined in \eqref{def-T}. 
A standard scheme to solve \eqref{SGE} {in deterministic regimes}  is the proximal
point algorithm proposed in~\cite{martinet1972determination,rockafellar1976monotone,rockafellar1976augmented}. \us{Given an $\x^0,$}
\begin{align} \tag{PP}
\notag \x^{k+1}\coloneqq(I+\lambda_kT)^{-1}(\x^k),
\end{align}
where $\lambda_k$ denotes the parameter of the proximal operator. The map $(I+\lambda_kT)^{-1}$, referred to as the resolvent of
$T$, is denoted by
$J_{\lambda_k}^T\triangleq(I+\lambda_kT)^{-1}$~\cite{rockafellar1976monotone}.
The resolvent \us{of $T$} is a single-valued, nonexpansive map for a monotone $T$; the
domain of $J_{\lambda_k}^T$ is equal to $\Real^n$ if $T$ is maximal
monotone~\cite{facchinei2007finite}. In~\cite{rockafellar1976monotone},
Rockafellar developed a proximal-point framework for generalized equations with
monotone operators, presenting a linear rate statement for strongly monotone
$T$. {This avenue has inspired several inexact proximal-point methods, including the classical inexact version~\cite{rockafellar1976monotone} and newer hybrid proximal extragradient (HPE) variants~\cite{solodov1999hybrid,monteiro2010complexity,monteiro2011complexity}.} More recently, in~\cite{corman2014generalized}, Corman and Yuan proved
that under maximal monotonicity, the proximal-point scheme produced
sequences which diminishes to
zero at the rate of $\mathcal{O}(1/k)$ under an appropriate metric while a linear rate can be
proven in strongly monotone regimes. In
this section, we develop a stochastic proximal
point framework  in which the resolvent of the
expectation-valued map, denoted by $(I+ \lambda_k
\mathbb{E}[\Phi(\x,\y(\x,\omega),\omega)])^{-1}$, is approximated with
increasing accuracy via a stochastic
approximation framework. Notably, a variance-reduced framework is proposed through which a linear rate (for strongly monotone $T$) and a sublinear rate $\mathcal{O}(1/k)$ (for monotone $T$) are derived with optimal or near-optimal sample-complexity. Notably, both schemes achieve deterministic iteration complexities in resolvent evaluations. When $T$ enjoys an amenable structure,
splitting-based approaches have emerged as an 
alternative.

\begin{table}[htb]
\begin{center}
\caption{Variance-reduced vs Stochastic proximal-point schemes}
\label{tab:spp}
\scriptsize
	\begin{tabular}[t]{  c | c | c  | c | c  } 
	\hline
	Alg/Prob. & Map & $\mathbb{E}[\|G(\x,\omega)\|^2] \leq $ & $\lambda_k$; $N_k$ & Statements  \\
\hline
\cite{patrascu2017nonasymptotic} (OPT) &  $f_L$ &  $M^2$ & $\mathcal{O}(1/k)$; 1 & {\scriptsize \tabincell{l}{ $\mathbb{E}[f(\bar{\x}^k)-f^*] \leq \mathcal{O}(\tfrac{1}{\sqrt{k}})$}} \\  
\hline
\cite{patrascu2017nonasymptotic} (OPT) &  $\sigma_{f,\omega}, (\nabla f)_L$ &  $M^2$ & $\mathcal{O}(1/k)$; 1 & {\scriptsize \tabincell{l}{ $\mathbb{E}[\|\x^k-\x^*\|^2] \leq \mathcal{O}(\tfrac{1}{k})$}} \\ 
\hline
 \cite{davis2019stochastic}   (OPT) & $f_L$  & $M^2$ & $\lambda$; 1 &  $\mathbb{E}][f(\bar{\x}^k)-f^*] \leq \mathcal{O}(\tfrac{1}{\sqrt{k}})$ \\\hline 
 \cite{davis2019stochastic}   (OPT) & $\sigma_f, f_L$  & $M^2$ & $\mathcal{O}(1/k)$; 1 &  $\mathbb{E}][f(\bar{\x}^k)-f^*] \leq \mathcal{O}(\tfrac{1}{k})$ \\ \hline 
\tabincell{c}{{\bf (VR-SPP)}\\ (SGE)} & $\sigma_T$ & $M_1^2\|\x\|^2 + M_2^2$ & \us{\tabincell{c}{$\lambda$;$\lceil \rho^{-2k}\rceil$\\ $\rho < 1$}}  & {\tiny \tabincell{l}{$\x^k \xrightarrow[a.s.]{k \to \infty} \x^* $ \\  $\mathbb{E}[\|\x^k-\x^*\|] \leq \mathcal{O}(q^k)$}} \\ \hline 
\tabincell{c}{{\bf (VR-SPP)} \\(SGE)} & MM  & $M_1^2\|\x\|^2 + M_2^2$ & \us{\tabincell{c}{$\lambda$;$\lceil k^{2a}\rceil$\\ $a > 1$}} & {\tiny \tabincell{l}{$\x^k \xrightarrow[a.s.]{k \to \infty} \x^* \in X^*$ \\  $\mathbb{E}[\|T_{\lambda}(\x^k)-\x^*\|] \leq \mathcal{O}(\tfrac{1}{k})$}} \\ \hline 
\end{tabular}

\texttt{$f_L, (\nabla_f)_L$: Lipschitz constants of convex $f$, $\nabla_f$; $\sigma_f, \sigma_{f,\omega}$: strong convexity constant of $f$, $f(\cdot,\omega)$, \\$\sigma_T$: strong monotonicity constant of $T$, MM: Maximal monotone, $G(\cdot,\omega)$: subgradient of $f(\cdot,\omega)$}
\vspace{-0.2in}
\end{center}
\end{table}
While stochastic counterpart of the proximal gradient method
(and its accelerated counterpart) have received much
attention~\cite{scholschmidt2011convergence,ghadimi15,jalilzadeh2018smoothed,jofre2019variance},
stochastic generalizations of the proximal-point method have been
less studied. Koshal, Nedi{\'c} and Shanbhag~\cite{koshal2013regularized} presented one of the first instances of a stochastic iterative proximal-point method for strictly monotone stochastic variational inequality problems and provided almost-sure convergence. In the context of minimizing
$\mathbb{E}[f(\x,\omega)]$, Ryu and Boyd~\cite{ryu14spi} proved
that the stochastic proximal scheme (defined as (SPI) below)
admitted a rate of convergence $\mathcal{O}(1/k)$ in mean-squared
error when $f(\cdot,\omega)$ is C$^2$, $L(\omega)$-smooth, and
strongly convex where $\mathbb{E}[L^2(\omega)] < \infty$.
\begin{align} \tag{SPI}
\x^{k+1}\coloneqq \mathrm{arg}\hspace{-0.02in}\min_{\x\in X}\left\{f(\x,\omega_k)+\tfrac{1}{2\lambda_k}\|\x-\x^k\|_2^2\right\}.
\end{align}
These statements were extended to model-based regimes by Asi and
Duchi~\cite{asi2018stochastic} where $f(\cdot,\omega_k)$ is
replaced by an appropriate model function. Subsequently, Patrascu
and Necoara~\cite{patrascu2017nonasymptotic} imposed a constraint
$\x \in \cap_{k} X_{\omega_k}$ and employed an additional
projection step onto $X_{\omega_k}$ at each step. Rate statements
are provided for both convex and strongly convex regimes without
the smoothness requirements. More recently, Davis and
Drusvyatskiy~\cite{davis2019stochastic} provided similar
statements in convex regimes while extending the rate statements
to weakly convex regimes. Our focus is on the stochastic generalized equation which requires an $\x \in \Real^n$ such that 
\begin{align*}
0 \in T(\x) \triangleq \mathbb{E}[\Phi(\x,\y(\x,\omega),\omega)], 
\end{align*}
where the components of the map $\Phi$ are denoted by $\Phi_i$, $i=1,\dots,n$,  $T_i:
\mathbb{R}^n \times \Omega \rightrightarrows \mathbb{R}^n$ is a set-valued map, $\mathbb{E}[\cdot]$ denotes the expectation, and the associated
probability space is given by
$(\Omega, {\cal F}, \mathbb{P})$. The only
related work is that by Bianchi~\cite{bianchi2016ergodic}; he proves a.s. convergence of a stochastic proximal-point (SPP) scheme under maximal monotonicity and requires computing the resolvent of the sampled map $T(\cdot,\omega)$ at each step, as defined next.
\begin{align*}
\tag{SPP}
\begin{aligned}
\x^{k+1}\coloneqq (I+\lambda_k \Phi(\bullet,\y(\bullet,\omega_k), \omega_k))^{-1}(\x^k).
\end{aligned}
\end{align*}
Rate statements (available for stochastic optimization) are summarized in~Table~\ref{tab:spp}. When the operator $T$  may be cast as the sum of two operators $A$ and $B$, there has been significant study of splitting methods~\cite{douglas1956numerical,peaceman1955numerical,lions1979splitting,passty1979ergodic} when the expectation-valued operator is single-valued in the optimization regime~\cite{scholschmidt2011convergence,ghadimi15,jalilzadeh2018smoothed,jofre2019variance,rosasco2019convergence} and more generally~\cite{combettes2016stochastic,rosasco2016stochastic} when $A$ is Lipschitz and
expectation-valued while $B$ is maximal monotone. Sample-average approximation techniques have also been developed~\cite{chen2012stochastic,shapiro2008stochastic} as a form of approximation framework. 

\noindent \hspace{0.2in} {\bf Gaps in stochastic proximal schemes.}
Several gaps emerge in studying prior work.
(i) {\em Gap between deterministic and stochastic rates.}
Deterministic schemes for strongly monotone
and monotone generalized equations display linear and
$\mathcal{O}(1/k)$ rate in resolvent operations
while stochastic analogs display rates of
$\mathcal{O}(1/k)$ and
$\mathcal{O}(1\sqrt{k})$, respectively. This leads
to far poorer practical behavior particularly when the
resolvent is challenging to compute, e.g., in
strongly monotone regimes, the complexity in resolvent
operations can improve from $\mathcal{O}(1/\epsilon)$ to
$\mathcal{O}(\log(1/\epsilon))$.  (ii) {\em Absence of rate
statements for monotone operators.} To the
best of our knowledge, there appear to be no non-asymptotic
rate statements available in monotone regimes.  (iii)  {\em State-dependent bounds
on subgradients and second moments.} Many subgradient and
stochastic approximation schemes impose bounds of the form
$\mathbb{E}[\|G(\x,\omega)\|^2] \leq M^2$ where $G(\x,\omega)
\in \partial f(\x,\omega)$ or $\mathbb{E}[\|w\|^2 \mid \x]
\leq \nu^2$ where $w = \nabla_\x f(\x,\omega) - \nabla_\x
f(\x)$. Both sets of assumptions are often challenging to
impose non-compact regimes. 

\noindent \hspace{0.2in} {\bf Motivation in developing (VR-SPP).} We draw
motivation from these gaps in developing variance-reduced
proximal schemes that can (a) achieve
deterministic rates of convergence with either identical or
slightly worse oracle complexities in both monotone and
strongly monotone regimes; (b) accommodate state-dependent
bounds to allow for non-compact domains; and (c) allow for
possibly biased oracles in select settings. Collectively, these schemes have provably better iteration complexity in resolvent operations, leading to superior empirical behavior.

\subsubsection{Assumptions and supporting results}\label{sec:3.2.2}
Throughout this section, we assume that the game $\Gcal$ admits the following ground assumption. 
\begin{tcolorbox}{\bf Ground Assumption (G2)}
Consider the $\Nbold$-player game $\Gcal$ in which the $i$th player is defined as (Player$_i(\x^{-i}))$ for $i=1, \cdots, \Nbold$. For $i=1, \cdots, \Nbold$, the parametrized lower-level mapping $F_i(\bullet,\x,\omega)$ is a strongly monotone map for $\x \in \Xscr$ and for every $\omega \in \Omega$. The associated map $T$, defined as \eqref{def-T}, is monotone. 
\end{tcolorbox}
We formalize an assumption on $T$ which is useful when providing convergence guarantees.
\begin{assumption} \label{max-mon} \em
The mapping $T$ is maximal monotone.
\end{assumption}

{While our original game is assumed to induce a monotone stochastic inclusion, our framework relies on solving a sequence of strongly monotone problems. To this end, the following assumption specifies} a strong monotonicity assumption on $T$.
\begin{assumption} \label{smonotonet} \em
The mapping $T$ is $\sigma$-strongly monotone, i.e. there exists $\sigma > 0$ such that
$(u-v)^\mathsf{T}(\x-\y)\ge \sigma\|\x-\y\|^2,\quad\forall  \x,\y \in \mathbb{R}^n,u \in T(\x), v \in T(\y)$.
\end{assumption}
Next, we define the Yosida approximation operator~\cite{brezis1973operateurs}.
\begin{definition}[{\bf Yosida approximation}] \em
For a set-valued maximal monotone operator $T: \Real^n \to \Real^n$ and for $\lambda > 0$, the Yosida approximation operator is
denoted as $T_{\lambda}(\bullet)$ and is defined as
$T_{\lambda} \triangleq \tfrac{1}{\lambda}(I-J_{\lambda}^T). $
\end{definition}

We now provide some properties of $J_{\lambda}^T$ and $T_{\lambda}$.
\begin{lemma}[{\bf Properties of $T_{\lambda}$ and $J_{\lambda}^T$}]\label{prop-res}\em \cite{rockafellar1976monotone,corman2014generalized} Given a maximal monotone map $T$ and a positive scalar $\lambda > 0$, the following hold. 
\begin{enumerate}
\item[(a)] $\x \in T^{-1}(0)$ if and only if $\x$ is a zero of $T_{\lambda}$, i.e. $0\in T(\x) \iff T_{\lambda}(\x)=0.$
\item[(b)] {$T_{\lambda}$ is a single-valued and $\tfrac{1}{\lambda}$-Lipschitz continuous map.}
\item[(c)] $J_{\lambda}^T$ is a single-valued and non-expansive map.
\end{enumerate}
\end{lemma}

Next, we assume the existence of a stochastic first-order oracle that can
provide estimator of $T(\x)$, given by $v(\x,\omega) \in \Phi(\x,\omega)$ {that
satisfies suitable moment bounds under state-dependent noise. Note that the
state-dependence assumption is crucial since it allows for dealing with the
unconstrained settings where compactness of the iterates cannot be guaranteed
via projection, for instance.}
\begin{assumption}[{\bf Stochastic first-order oracle
for $T$ with state-dependent bounds}]\label{ass_sfo} \em There exists a stochastic first-order oracle that
given an $\x$ produces $v(\x,\omega)$ such that
$\mathbb{E}[v(\x,\omega) \mid \x] = v(\x)$ and $\mathbb{E}[\us{\|v(\x,\omega)\|^2} \mid \x] \leq M_1^2
\|\x\|^2 + M_2^2$ a.s., where $v(\x) \in
T(\x)$  for all $\x$ and $v(\x,\omega) \in
\Phi(\x,\y(\x,\omega),\omega)$. \end{assumption}

\subsubsection{Approximating $(I+\lambda_k \mathbb{E}[\Phi(\x,\y(\x,\omega),\omega)])^{-1}$ via stochasic approximation}\label{sec:3.2.3} 
Our framework relies on computing inexact resolvents  with error $e_k$ via ({\bf VR-SPP}). Recall that the the resolvent problem can be rewritten as follows with $\lambda_k = \lambda$. 
\begin{align}\label{exact-prox}
\left[\tilde{\x}^{k+1} = (I + \lambda T)^{-1}(\x^k)\right] \equiv \left[ 0 \in T(\tilde{\x}^{k+1}) + \tfrac{1}{\lambda}(\tilde{\x}^{k+1}-\x^k)\right].   
\end{align}
\noindent However, $T$ is an expectation-valued map where evaluating $\Phi(\x,\y(\x,\omega),\omega)$ requires solving a lower-level problem with solution $\y(\x,\omega)$. Therefore exact solutions of \eqref{exact-prox} can generally not be provided in finite time. We now discuss how one may compute approximate solutions of such problem in finite time. We begin by defining $F_k(\bullet)$ as 
\begin{align}\label{def-Fk}
	0 \in  F_k (\z) \triangleq T(\z) + \tfrac{1}{\lambda}(\z-\x^k), \mbox{ where } T(\z) = \mathbb{E}[\Phi(\z,\y(\z,\omega),\omega)], 
\end{align}

We observe that $F_k$ is $\tfrac{1}{\lambda}$-monotone. Let $u \in \tilde{F}_k(\z,\omega) \triangleq \Phi(\z,\y(\z,\omega),\omega) + \tfrac{1}{\lambda}(\z-\x^k)$ and it follows that $u = v + \tfrac{1}{\lambda}(\z-\x^k)$ where $v \in \Phi(\x,\y(\x,\omega),\omega)$. In addition, we have $\mathbb{E}[\|v\|^2] \leq M_1^2\|\z\|^2 + M_2^2$ (by Assumption~\ref{ass_sfo}). 
We remind the reader that $\z, u,$ and $v$ are defined as the tuple of the analogous player-specific counterparts, defined as
$$ \z = \pmat{\z^1 \\ \vdots \\ \z^{\Nbold}}, u = \pmat{u^1 \\ \vdots \\u^{\Nbold}}, \mbox{ and } v = \pmat{v^1 \\ \vdots \\ v^{\Nbold}},$$ respectively.
Therefore, we have that
\begin{align}
 \notag\mathbb{E}[\|u\|^2] & \leq 2\mathbb{E}[\|v\|^2] + \tfrac{2}{\lambda^2} \mathbb{E}[\|\z-\x^k\|^2] \leq  2M_1^2\|\z\|^2 + 2M_2^2 + \tfrac{2}{\lambda^2} \mathbb{E}[\|\z-\x^k\|^2] 
\\
		& \leq 
   4M_1^2\|\x^k\|^2 + 2M_2^2 + (4M_1^2+\tfrac{2}{\lambda^2}) \mathbb{E}[\|\z-\x^k\|^2]. \label{vsa1}
\end{align}
If $\z^{k}_{0} = \x^k$, an inexact solution can be computed by taking $N_k$ steps of the update rule (SA), defined as follows where $\alpha_j$ denotes the steplength.
\begin{align}\tag{SA}
	\z^{k}_{j+1} & \coloneqq \z^{k}_{j} - \alpha_j u^k_j,\mbox{ for $j = 0, \cdots, N_{k}-1$,} \\
\mbox{where }\notag u^k_j & = v^{k}_{j} +  \tfrac{\z^{k}_{j}-\x^k}{\lambda} \ \mbox{and} \  v^{k}_{j} \in \Phi(\z^{k}_{j},\y(\z^{k}_{j},\omega_{j,k}),\omega_{k,j}).
\end{align}
The update rule (SA) can be explicitly written for each player as follows for $j=0, \cdots, N_k-1$. 
\begin{align}\label{sa-player} 
\left\{\begin{aligned}	\z_{j+1}^{k,1} & \coloneqq \z^{k,1}_{j} - \alpha_j u_{j}^{k,1} \\
		& \vdots \\
	\z_{j+1}^{k,\Nbold} & \coloneqq \z^{k,\Nbold}_{j} - \alpha_j u_{j}^{k,\Nbold}\end{aligned} \right\},  \mbox{ where } \left\{ \begin{aligned} v_{j}^{k,i} &\in \partial_{\x^{i}} \tilde{f}_{i}(\x^{k,i}, \x^{k,-i}, \y^i(\x^{k},\omega_{k,j}),\omega_{k,j}) + \mathcal{N}_{\Xscr_{i}}(\x^{k,i}), \\
u_{j}^{k,i}  & = v_{j}^{k,i} + \tfrac{\z_{j}^{k,i} - \x^{k,i}}{\lambda}  \mbox{ for } i =  1, \cdots, \Nbold. 
\end{aligned} \right\}
\end{align}
Upon termination after $N_k$ steps, $\x^{k+1,i} \coloneqq \z^{k,i}_{N_k}$ for $i=1, \cdots, \Nbold$. 
As part of the proposed scheme, we generate  $N_0$, $N_1$ 
$\cdots$, $N_{K-1}$ samples from the first-order oracle, \us{where $N_k$ samples are used at the $k$th step.} Consequently, we define $\mathcal{F}_{k}$ as the history up to iteration $k$ as follows. 
$$\mathcal{F}_{k} \triangleq
\left\{\{\x^{0,i}\}_{i=1}^{\Nbold},\{\{v_{j}^{0,i}\}_{i=1}^{\Nbold}\}_{j=0}^{N_0-1},  \cdots,  \{\{v_{j}^{k-1,i}\}_{i=1}^{\Nbold}\}_{j=0}^{N_{k-1}-1}
\right\}.$$
We define the history $\mathcal{F}_{k,j}$ at iteration \us{$j \geq 1$} of the inner scheme as follows.
\begin{align*} \mathcal{F}_{k,j} \triangleq \mathcal{F}_{k-1} \cup \left\{\{\{v_{\ell}^{k,i}\}_{i=1}^{\Nbold}\}_{\ell=0}^{j-1}\right\}. \end{align*}
We are now ready to formally define the variance-reduced proximal-point scheme for hierarchical monotone games.
\begin{tcolorbox}{{\bf Var-reduced proximal-point scheme for hierarchical monotone games ({\bf VR-SPP})}}
\be
\item[(0)] Let $k = 0$, $\x^{0,i} \in \Xscr_i$ for $i = 1, \cdots, \Nbold$. Given $K$, $\lambda$,  $\{N_k\}_{k=0}^{K-1}$, and $\{\alpha_j\}_{j=0}^{N_k-1}$. 
\item[(1)] While $k < K$,  
\item[(2)] Let $\z^{k,i}_{0} = \x^{k,i}$ for $i = 1, \cdots, \Nbold$. Generate $\{\{\z^{k,i}_{j}\}_{i=1}^{\Nbold}\}_{j=1}^{N_k}$ by updating for $j = 1, \cdots, N_k$.
    \begin{align*} 
\hspace{-0.23in}\left\{\begin{aligned}	\z_{j+1}^{k,1} & := \z^{k,1}_{j} - \alpha_j u_{j}^{k,1} \\
		& \vdots \\
	\z_{j+1}^{k,\Nbold} & := \z^{k,\Nbold}_{j} - \alpha_j u_{j}^{k,\Nbold}\end{aligned} \right\},  \mbox{ where } \left\{ \begin{aligned} v_{j}^{k,i} &\in \partial_{\x^{i}} \tilde{f}_{i}(\x^{k,i},\x^{k,-i}, \y^i(\x^k,\omega_{k,j}),\omega_{k,j}) + \mathcal{N}_{\Xscr_{i}}(\x^{k,i}), \\
u_{j}^{k,i}  & = v_{j}^{k,i} + \tfrac{\z_{j}^{k,i} - \x^{k,i}}{\lambda}  \mbox{ for } i =  1, \cdots, \Nbold.
\end{aligned} \right\}
\end{align*}
\item[(3)] Let $\x^{k+1,i} = \z^{k,i}_{N_k}$ for $i =1, \cdots, \Nbold$.
\item[(4)] Set $k \coloneqq k+1$ and go to (1). 
\ee
\end{tcolorbox}

\noindent {\bf Computing an element of $\partial_{\x^i}
\tilde{f}_i(\x,\y(\x,\omega),\omega)$.} The reader will observe that the scheme
requires computing $v_{j}^{k,i} \in \partial_{\x^i}
\tilde{f}_i(\x^{k,i},\x^{k,-i},\y(\x^k,\omega_{k,j}),\omega_{k,j})$. 
\begin{enumerate}
    \item[(i)] {\em Closed-form expression entirely in terms of $\x$.} 
In the first application in Section~\ref{sec:2.3},
$\y(\x,\omega)$ can be expressed in closed form in terms of $\x$; for instance, in this case
$\y_i(\x^i,\omega)$ =\\$\max\{Q_i(\omega)^{-i}b_i(\x^i,\omega),
\ell_i(\x^i,\omega))$. 
    \item[(ii)] {\em Expression in terms of $\y(\x,\omega)$ and $\partial_{\x} \y(\x,\omega).$} In the second application in Section~\ref{sec:2.3},
$\partial_{\x^i} \tilde{f}_i(\x^{k,i},\x^{k,-i},\y^i(\x^k,\omega_{k,j}),\omega_{k,j})$ is {\bf
not} available in closed form but can be expressed in terms of $\y(\x,\omega)$
and $\partial_{\x} \y(\x,\omega)$. We observe that the hierarchical structure
emerges because $\y(\x,\omega)$ is not available in closed form and requires
solving the lower-level problem (Equil$_{\rm foll}(X,\omega)$). In short, the structure of $\partial_{\x} \y(\x,\omega)$ needs to be derived via the model.
    \item[(iii)] {\em No problem structure.} Approaches (i) and (ii) require leveraging problem structure. In the absence of such structure, we would need to employ smoothing  and then compute a (sampled) gradient. However, the resulting gradient estimator is conditionally biased. It remains an open question as to whether this bias can be addressed within the above framework since no schemes exist to the best of our knowledge for resolving such problems with possibly biased oracles.
\end{enumerate}

In the remainder of this subsection, we will provide a rigorous rationale for why $\z_{N_k}^k$ satisfies a suitable error bound in an expectation-valued sense. We utilize the following lemma in the next proposition, both of which are proved in the Appendix.
\begin{lemma}\label{sa-rate-quad}\em
Given positive scalars $c, \mathcal{M}, \theta$ and $J_1 \in \mathbb{Z}_+$, consider 
the recursive inequality given by 
$\mathcal{A}_{j+1} \leq (1-2c\alpha_j) \mathcal{A}_j + \tfrac{\alpha_j^2 \mathcal{M}^2}{2}$ \ssc{where $\alpha_j=\tfrac{\theta}{j}$}  for $j \ge J_1$ and $\mathcal{A}_j \geq 0$ for all $j$. Suppose $J_2 \triangleq \lceil 2c\theta\rceil$, $J \triangleq \max\{J_1,J_2\}$, and \us{$B \triangleq\tfrac{\theta^2\pi^2}{12}$}.
Then for $j \geq J$, we have that
$\mathcal{A}_{j} \leq \frac{\max\left\{\tfrac{\mathcal{M}^2\theta^2}{2(2c\theta-1)}, J \mathcal{A}_\ssc{J}\right\}}{j} \leq \frac{\tfrac{\mathcal{M}^2\theta^2}{2(2c\theta-1)}+ J \us{(\mathcal{A}_{\us{1}}+B\mathcal{M}^2)}}{j}.$ 
\end{lemma}

\begin{proposition}\label{prop-res-error}\em
Consider a $\tfrac{1}{\lambda}$-strongly monotone map $F_k$ defined as \eqref{def-Fk}. 
Suppose Assumption~\ref{ass_sfo} holds and $0 \in F_k(J^T_{\lambda}(\x^k))$. If $J_1 \triangleq \lceil 2\lambda\theta (4M_1^2+\tfrac{2}{\lambda^2})\rceil$, $J_2 \triangleq \lceil \tfrac{\theta}{\lambda}\rceil$, $J \triangleq \max\{J_1,J_2\}$, then
\begin{align*}
\mathbb{E}[\|\z_j^k-J^T_{\lambda}(\x^k)\|^2 \mid \mathcal{F}_k] \leq  
\tfrac{\nu_1^2 \|\x^k\|^2 + \nu_2^2}{2j} \mbox{ for } j \geq J,   
\end{align*}
where $\nu_1^2$ and $\nu_2^2$ are defined as
\begin{align*}
\nu_1^2  &  \triangleq 
\left(\left(\tfrac{\theta^2}{2(2c\theta-1)}+JB\right)\left(\us{136}M_1^2+\tfrac{64}{\lambda^2}\right)+\us{8J}\right)\mbox{ and } \\  
 \nu_2^2 &  \triangleq  4\left(\tfrac{\theta^2}{2(2c\theta-1)}+J \us{B}\right)M_2^2  
   +  8\left(\left(\tfrac{\theta^2}{2(2c\theta-1)}+J\us{B}\right)\left(16M_1^2+\tfrac{8}{\lambda^2}\right)+J\right)\|\x^*\|^2. 
\end{align*}
\end{proposition}
The super-martingale convergence lemma is also employed in our analysis~\cite{polyak1987introduction}.
\begin{lemma}\label{robbins} \em
Let $r_k$, $u_k$, $\delta_k$, $\psi_k$ be nonnegative random variables adapted to $\sigma$-algebra $\mathcal{F}_k$, and let the following relations hold almost surely:
\begin{align}
\notag\mathbb{E}[r_{k+1}\mid\mathcal{F}_k]\leq(1+u_k)r_k-\delta_k+\psi_k, \quad\forall k; \quad \sum_{k=0}^{\infty}u_k<\infty,\mbox{ and } \sum_{k=0}^{\infty}\psi_k<\infty.
\end{align}
Then a.s., 
$\lim_{k\to\infty}r_k=r$ and $\sum_{k=0}^{\infty}\delta_k<\infty,$
where $r\ge0$ is a random variable. 
\end{lemma}

\subsection{Convergence analysis  under monotonicity}\label{sec:3.3}
We begin with a result from~\cite{corman2014generalized} and subsequently recall a bound on the sequence of iterates produced by a deterministic exact proximal-point scheme~\cite[Lemma~2.5]{corman2014generalized}.
\begin{lemma}\em\cite{corman2014generalized}\label{ysd} 
Given a set-valued maximal monotone operator $T$: $\Real^n \rightrightarrows \Real^n$, let $J_{\lambda}^T$ denote the resolvent operator while $T_{\lambda}$ denotes the Yosida approximation operator of $T$. Then 
$T_{\lambda}(x)\in T(J_{\lambda}^T(x))$ for all $x\in \Real^n.$ 
\end{lemma}
\begin{lemma}\label{PP2} \em
Let Assumption~\ref{max-mon} hold.  Consider any sequence generated by ({\bf VR-SPP}). Then the following holds for all $k>0$:
\begin{align*}
\|J_{\lambda}^T(\x^k)-\x^*\|^2=\|\x^k-\x^*\|^2-\lambda^2\|T_{\lambda}(\x^k)\|^2-2\lambda T_{\lambda}(\x^k)^\mathsf{T}(J_{\lambda}^T(\x^k)-\x^*).
\end{align*}
\end{lemma}
The next lemma allows for proving convergence of iterates generated by ({\bf VR-SPP}).
\begin{lemma} \label{vr-l} \em
Let Assumptions~\ref{max-mon} and \ref{ass_sfo} hold. Suppose $\lambda>0$ and $N_k \triangleq \lceil (k+1)^{2a} \rceil$ for all $k > 0$, where  $a>1$. Consider a sequence generated $\{\x^k\}$ generated by $(${\bf VR-SPP}$)$. Then $\{\|\x^k-\x^*\|\}$ is convergent almost surely.
\end{lemma}
\begin{proof}
From non-expansivity of $J_{\lambda}^T$~\cite{rockafellar1976monotone}, we obtain the following relation
\begin{align}
\|J_{\lambda}^T(\x^k)-\x^*\|^2
\le\|\x^{k}-\x^*\|^2. \label{pps-y2}
\end{align}
By adding and subtracting $J_{\lambda}^T(\x^k)$, we may bound $\mathbb{E}[\|\x^{k+1}-\x^*\|]$ as follows.
\begin{align}
& \notag\mathbb{E}[\|\x^{k+1}-\x^*\|\mid\mathcal{F}_k]  \le\mathbb{E}[\|J_{\lambda}^T(\x^k)-\x^*\|\mid\mathcal{F}_k]+\mathbb{E}[\|\x^{k+1}-J_{\lambda}^T(\x^k)\|\mid\mathcal{F}_k] \\
 \notag&\overset{\tiny (\mbox{Prop.}~\ref{prop-res-error})}{\le}\|J_{\lambda}^T(\x^k)-\x^*\|+\tfrac{\sqrt{\nu_1^2\|\x^k\|^2+\nu_2^2}}{\sqrt{N_{k}}}\\
& \le\|J_{\lambda}^T(\x^k)-\x^*\|+\tfrac{\sqrt{\nu_1^2\|\x^k\|^2+\nu_2^2+2\nu_1\nu_2\|\x^k\|}}{\sqrt{N_{k}}} \\
&=\|J_{\lambda}^T(\x^k)-J_{\lambda}^T(\x^*)\|+\tfrac{\nu_1\|\x^k\|+\nu_2}{\sqrt{N_{k}}} 
 \overset{(\tiny \mbox{Lemma}~\ref{prop-res})}{\le}\|\x^k-\x^*\|+\tfrac{\nu_1(\|\x^k-\x^*\|+\|\x^*\|)+\nu_2}{\sqrt{N_{k}}} \label{ieq1} \\
\notag&=(1+\tfrac{\nu_1}{\sqrt{N_k}})\|\x^k-\x^*\|+\tfrac{\nu_1\|\x^*\|+\nu_2}{\sqrt{N_{k}}}  \ = \  (1+\tfrac{\nu_1}{\sqrt{N_k}})v_k-\delta_k+\psi_k,
\end{align}
where $v_k$, $\delta_k$,  and $\psi_k$ are nonnegative random variables defined as 
 $v_k \triangleq \|\x^k-\x^* \|$, $\delta_k  \triangleq 0$, 
 and $\psi_k  \triangleq \tfrac{\nu_1\|\x^*\|+\nu_2}{\sqrt{N_{k}}} $.
By Lemma \ref{robbins}, 
$v_k \to \bar v \geq 0$ almost surely.\end{proof}

\begin{proposition}[{\bf a.s. convergence of (VR-SPP)}] \label{egvr} \em
Consider a sequence $\{\x^k\}$ generated by {\bf (VR-SPP)}. Let Assumptions~\ref{max-mon} and \ref{ass_sfo} hold. Suppose $\lambda>0$ and $N_k \triangleq \lceil (k+1)^{2a} \rceil$ for all $k > 0$, where  $a>1$. Then for any $\x^0$, $\x^k \xrightarrow[a.s.]{k \to \infty} \x^*\in X^*$ where $X^*$ denotes the solution set of (SGE).
\end{proposition}
\begin{proof}
From Lemma~\ref{vr-l}, $\{\|\x^k-\x^*\|\}$ is convergent a.s. implying that there exists $C$ such that a.s., $\|\x^k-\x^*\|^2\le C^2$ for all $k$. Recall that $\|\x^{k+1}-\x^*\|^2$ can be bounded as follows:
\begin{align}
\notag\|\x^{k+1}-\x^*\|^2  
&\le \|J_{\lambda}^T(\x^k)-\x^*\|^2+\|\x^{k+1}-J_{\lambda}^T(\x^k)\|^2\notag \\
& +2\|J_{\lambda}^T(\x^k)-\x^*\|\|\x^{k+1}-J_{\lambda}^T(\x^k)\|. \label{pps-6}
\end{align}
By Lemma \ref{ysd},  
\begin{align}
\notag  & \quad  (T_{\lambda}(\x^k)-T_{\lambda}(\x^*))^\mathsf{T}(J_{\lambda}^T(\x^k)-J_{\lambda}^T(\x^*)) \ge 0,\\
\equiv  &  \quad (T_{\lambda}(\x^k))^\mathsf{T}(J_{\lambda}^T(\x^k)-\x^*) \ge 0, \label{pps-y3}
\end{align}
by noticing that $T_{\lambda}(\x^*)  = 0 $ and $J_{\lambda}^T(\x^*)=\x^*$. By substituting \eqref{pps-y3} in Lemma \ref{PP2}, 
\begin{align}
\notag\|J_{\lambda}^T(\x^k)-\x^*\|^2&=\|\x^{k}-\x^*\|^2-\lambda^2\|T_{\lambda}(\x^k)\|^2-2\lambda T_{\lambda}(\x^k)^\mathsf{T}(J_{\lambda}^T(\x^k)-\x^*) \\
&\le\|\x^{k}-\x^*\|^2-\lambda^2\|T_{\lambda}(\x^k)\|^2. \label{pps-y4}
\end{align}
By substituting the bound~\eqref{pps-y4} in~\eqref{pps-6} and taking expectations conditioned  on $\mathcal{F}_k$, we obtain the following bound.
\begin{align}
\notag\mathbb{E}[&\|\x^{k+1}-\x^*\|^2\mid\mathcal{F}_k]\le\|J_{\lambda}^T(\x^k)-\x^*\|^2+\mathbb{E}[\|\x^{k+1}-J_{\lambda}^T(\x^k)\|^2\mid\mathcal{F}_k] \\
&+2\|J_{\lambda}^T(\x^k)-\x^*\|\mathbb{E}[\|\x^{k+1}-J_{\lambda}^T(\x^k)\|\mid\mathcal{F}_k] \label{pps-61} \\
&\le\|\x^{k}-\x^*\|^2-\lambda^2\|T_{\lambda}(\x^k)\|^2+\tfrac{\nu_1^2\|\x^k\|^2+\nu_2^2}{N_k}+2\|\x^k-\x^*\|\left(\tfrac{\nu_1\|\x^k\|+\nu_2}{\sqrt{N_k}}\right)\\
\notag&\le\|\x^{k}-\x^*\|^2-\lambda^2\|T_{\lambda}(\x^k)\|^2+\tfrac{\nu_1^2(2\|\x^k-\x^*\|^2+2\|\x^*\|^2)+\nu_2^2}{N_k}\\\notag&+2\|\x^k-\x^*\|\left(\tfrac{\nu_1(\|\x^k-\x^*\|+\|\x^*\|)+\nu_2}{\sqrt{N_k}}\right) \\
\notag&=(1+\tfrac{\ssc{2}\nu_1^2}{N_k}\ssc{+\tfrac{\ssc{2}\nu_1}{\sqrt{N_k}}})\|\x^k-\x^*\|^2-\lambda^2\|T_{\lambda}(\x^k)\|^2+\tfrac{2\nu_1^2\|\x^*\|^2+\nu_2^2}{N_k} \\
\notag&+2\|\x^k-\x^*\|\left(\tfrac{\nu_1\|\x^*\|+\nu_2}{\sqrt{N_k}}\right) \\
\hspace{-0.9in} \notag&\le (1+\tfrac{2\nu_1^2}{N_k}\ssc{+\tfrac{\ssc{2}\nu_1}{\sqrt{N_k}}})\|\x^k-\x^*\|^2-\lambda^2\|T_{\lambda}(\x^k)\|^2+\tfrac{2\nu_1^2\|\x^*\|^2+\nu_2^2}{N_k} 
+2C\left(\tfrac{\nu_1\|\x^*\|+\nu_2}{\sqrt{N_k}}\right).
\end{align}
By definition of $N_k$, $\sum_k\tfrac{1}{N_k} <\sum_k \tfrac{1}{\sqrt{N_k}} < \infty$. By Lemma~\ref{robbins}, $\{\|\x^k-\x^*\|\}$ is
convergent and $\sum_{k}\lambda^2\|T_{\lambda}(\x^k)\|^2$ $<\infty$ in an a.s.
sense. Therefore, in an a.s.
	sense, we have 
$\lim_{k \to \infty}\|T_{\lambda}(\x^k)\|^2 = 0. $
Since $\{\|\x^k-\x^*\|^2\}$ is a convergent sequence in an
a.s. sense, $\{\x^k\}$ is bounded a.s. and has a convergent
subsequence. Consider any convergent subsequence of
$\{\x^k\}$ with index set denoted by ${\cal K}$. Suppose its
limit point is $\bar{\x}$. Consequently, by the continuity of $T_{\lambda}$, we have that $\lim_{k \in {\cal
K}} T_{\lambda}(\x^k) = T_{\lambda}({\bar \x}) = 0$. It follows that
$\bar \x$ is a solution to $0\in T(\x)$. Consequently, we have that $\lim_{k \in \cal K} \x^k = \bar{\x} \in X^*$, in an a.s. sense. It follows that $\{\|\x^k-\bar \x\|^2\}$ is convergent and its
unique limit point is zero. Thus every subsequence of $\{\x^k\}$
converges a.s. to $\bar \x$, implying that the entire
sequence of $\{\x^k\}$ converges to $\bar x$ almost surely.  \end{proof}

We conclude this subsection with a rate statement for ({\bf VR-SPP}). 
\begin{proposition}[{\bf Rate of convergence of (VR-SPP) under maximal monotonicity}] \label{ratePP1} \em
Let Assumptions~\ref{max-mon} and \ref{ass_sfo} hold, $\lambda>0$, and $N_k \triangleq \lceil (k+1)^{2a} \rceil$ for all $k > 0$, where  $a>1$. Consider a sequence $\{\x^k\}$ generated by $(${\bf VR-SPP}$)$. \\ (a) For any $k \geq 0$, we have that $\mathbb{E}[\|T_{\lambda}(\x^k)\|^2]=\mathcal{O}\left(\tfrac{1}{k+1}\right).$ \\
\noindent (b) Suppose $\x^{K+1}$ satisfies $\mathbb{E}[\|T_{\lambda}(\x^{K+1})\|^2] \leq \epsilon$. Then the oracle complexity of computing such an $\x^{K+1}$ satisfies 
$\sum_{k=0}^KN_k=\mathcal{O}\left(\tfrac{1}{\epsilon^{2a+1}}\right).$
\end{proposition}
\begin{proof}
(a) By taking unconditional expectations on \eqref{ieq1},
\begin{align}
 \notag\mathbb{E}[\|\x^{k+1}-\x^*\|]&\le\mathbb{E}[\|\x^k-\x^*\|]+\tfrac{\nu_1(\mathbb{E}[\|\x^k-\x^*\|]+\|\x^*\|)+\nu_2}{\sqrt{N_{k}}}\\
 &\le  \mathbb{E}[\|\x^k-\x^*\|]+\tfrac{\nu_1(C+\|\x^*\|)+\nu_2}{\sqrt{N_{k}}} 
 \le \|\x^0-\x^*\|+\sum_{i=0}^\infty\tfrac{\nu_1(C+\|\x^*\|)+\nu_2}{\sqrt{N_i}}, \label{ppt-2}
\end{align}
where $\|\x^k-\x^*\| \leq C$ a.s. for all $k \geq 0$.  Taking unconditional expectations over \eqref{pps-61}, 
\begin{align}
\notag\mathbb{E}[\|\x^{k+1}-\x^*\|^2]& \leq \mathbb{E}[\|J_{\lambda}^T(\x^k)-\x^*\|^2]+\mathbb{E}[\|\x^{k+1}-J_{\lambda}^T(\x^k)\|^2] \\
\notag &+2\mathbb{E}[\|J_{\lambda}^T(\x^k)-\x^*\|]\mathbb{E}[\mathbb{E}[\|\x^{k+1}-J_{\lambda}^T(\x^k)\|\mid \mathcal{F}_k]].  \\
\notag		& \overset{\eqref{pps-y4}}{\leq} \mathbb{E}[\|\x^k-\x^*\|^2] - \lambda^2\mathbb{E}[\|T_{\lambda}(\x^k)\|^] + \mathbb{E}[\|\x^{k+1}-J_{\lambda}^T(\x^k)\|^2] \\
&+2\mathbb{E}[\|J_{\lambda}^T(\x^k)-\x^*\|]\mathbb{E}[\mathbb{E}[\|\x^{k+1}-J_{\lambda}^T(\x^k)\|\mid \mathcal{F}_k]]. \label{ppt-61} 
\end{align}
By non-expansivity of $J_{\lambda}^T$ and by substituting \eqref{ppt-2} in \eqref{pps-y2}, 
\begin{align}
\mathbb{E}[\|J_{\lambda}^T(\x^k)-\x^*\|]\le\mathbb{E}[\|\x^k-\x^*\|]\le\|\x^0-\x^*\|+\sum_{i=0}^\infty\tfrac{\nu_1(C+\|\x^*\|)+\nu_2}{\sqrt{N_i}}. \label{ppt-62}
\end{align}
Inserting \eqref{ppt-62} into \eqref{ppt-61}, we obtain the following bound:
\begin{align}
\notag\mathbb{E}[\|\x^{k+1}-\x^*\|^2] &\le \mathbb{E}[\|\x^{k}-\x^*\|^2]-\lambda^2\mathbb{E}[\|T_{\lambda}(\x^k)\|^2]+\tfrac{\nu_1^2\|\x^k\|^2+\nu_2^2}{N_{k}} \\
\notag&+\tfrac{2\nu_1\|\x^k\|+\nu_2}{\sqrt{N_{k}}}\left(\|\x^0-\x^*\|+\sum_{i=0}^\infty\tfrac{\nu_1(C+\|\x^*\|)+\nu_2}{\sqrt{N_i}}\right) \\
\notag&\le \mathbb{E}[\|\x^{k}-\x^*\|^2]-\lambda^2\mathbb{E}[\|T_{\lambda}(\x^k)\|^2]+\tfrac{\nu_1^2(2C^2+2\|\x^*\|^2)+\nu_2^2}{N_{k}} \\
\notag&+\tfrac{2\nu_1(C+\|\x^*\|)+\nu_2}{\sqrt{N_{k}}}\left(\|\x^0-\x^*\|+\sum_{i=0}^\infty\tfrac{\nu_1(B+\|\x^*\|)+\nu_2}{\sqrt{N_i}}\right).
\end{align}
Defining $E_1\triangleq \sum_{i=0}^\infty\frac{1}{\sqrt{N_i}}$, $E_2\triangleq \sum_{i=0}^\infty\frac{1}{N_i}$, $D_1\triangleq \nu_1^2(2C^2+2\|\x^*\|^2)+\nu_2^2$ and $D_2\triangleq \nu_1(C+\|\x^*\|)+\nu_2$, and summing from $i=0,\cdots,k$, we get
\begin{align}\notag
\lambda^2\sum_{i=0}^k\mathbb{E}[\|T_{\lambda}(x_i)\|^2]& \le \|\x^0-\x^*\|^2-\mathbb{E}[\|\x^{k+1}-\x^*\|^2]\\
\notag
& +\sum_{i=0}^k\left(\tfrac{D_1}{N_{i}}+\tfrac{2D_2}{\sqrt{N_{i}}}\left(\|\x^0-\x^*\|+D_2\sum_{\ell=0}^\infty\tfrac{1}{\sqrt{N_{\ell}}}\right)\right)\\
\notag&\le  \|\x^0-\x^*\|^2+D_1E_2+2D_2 E_1\|\x^0-\x^*\|+2D_2^2E_1^2 \\
\label{bddt-2}
&=\left(\|\x^0-\x^*\|+D_2E_1\right)^2+D_1E_2+D_2^2E_1^2.
\end{align}
We now proceed to analyze $\sum_{i=0}^k \mathbb{E}[\|T_{\lambda}(x_i)\|^2]$ by noting that
\begin{align}
\notag T_\lambda(\x^k)&=\tfrac{1}{\lambda}(\x^k-J_\lambda^T(\x^k))=\tfrac{1}{\lambda}(\x^{k+1}-J_\lambda^T(\x^k)-(\x^{k+1}-\x^k)) \\
\notag&=\tfrac{1}{\lambda}(\x^{k+1}-J_\lambda^T(\x^k))-\tfrac{1}{\lambda}(J_\lambda^T(\x^{k+1})-J_\lambda^T(\x^k))-(T_\lambda(\x^{k+1})-T_\lambda(\x^k))
\end{align}
It follows that
\begin{align}\notag
& \notag(T_\lambda(\x^{k+1})-T_\lambda(\x^k))^\mathsf{T}T_\lambda(\x^k)= \tfrac{1}{\lambda}(T_\lambda(\x^{k+1})-T_\lambda(\x^k))^\mathsf{T}(\x^{k+1}-J_\lambda^T(\x^k))\\&-\underbrace{\tfrac{1}{\lambda}(T_\lambda(\x^{k+1})-T_\lambda(\x^k))^\mathsf{T}(J_\lambda^T(\x^{k+1})-J_\lambda^T(\x^k))}_{\ \geq \ 0}-\|T_\lambda(\x^{k+1})-T_\lambda(\x^k)\|^2\notag\\
& \leq \tfrac{1}{\lambda}(T_\lambda(\x^{k+1})-T_\lambda(\x^k))^\mathsf{T}(\x^{k+1}-J_\lambda^T(\x^k))-\|T_\lambda(\x^{k+1})-T_\lambda(\x^k)\|^2. \label{ppt-63}
\end{align}
Then we have
\begin{align}
 \notag\|T_\lambda(\x^{k+1})\|^2 &  =\|T_\lambda(\x^{k})\|^2+\|T_\lambda(\x^{k+1})-T_\lambda(\x^k)\|^2+2(T_\lambda(\x^{k+1})-T_\lambda(\x^k))^\mathsf{T}T_\lambda(\x^k)\\
\notag&\overset{\eqref{ppt-63}}{\le}\|T_\lambda(\x^{k})\|^2-\|T_\lambda(\x^{k+1})-T_\lambda(\x^k)\|^2 \\
\notag &+\tfrac{2}{\lambda}\|T_\lambda(\x^{k+1})-T_\lambda(\x^k)\|\|\x^{k+1}-J_\lambda^T(\x^k)\| \\
\notag&\le \|T_\lambda(\x^{k})\|^2-\|T_\lambda(\x^{k+1})-T_\lambda(\x^k)\|^2+\|T_\lambda(\x^{k+1})-T_\lambda(\x^k)\|^2+\tfrac{D_1}{\lambda^2N_{k}} \\
& =\|T_\lambda(\x^{k})\|^2+\tfrac{D_1}{\lambda^2N_{k}}. \label{ppt-64}
\end{align}
By \eqref{ppt-64}, we have the following relationship.
\begin{align}
\|T_\lambda(\x^{k})\|^2\le \|T_\lambda(x_{i})\|^2+\sum_{j=i}^{k-1}\tfrac{D_1}{\lambda^2N_{j}}, \quad \forall i=0,\cdots,k-1. \label{bddt-1}
\end{align}
Thus, we have $(k+1)\|T_\lambda(\x^{k})\|^2\le \sum_{i=0}^k \|T_\lambda(x_{i})\|^2+\sum_{i=0}^k \sum_{j=i}^{k-1}\tfrac{D_1}{\lambda^2N_{j}}$, implying that
\begin{align}
\notag\mathbb{E}[\|T_{\lambda}(\x^k)\|^2] 
&\overset{\eqref{bddt-2},\eqref{bddt-1}}{\le}\tfrac{\left(\|\x^0-\x^*\|+D_2E_1\right)^2+D_1E_2+D_2^2E_1^2}{\lambda^2(k+1)}+\tfrac{D_1\sum_{i=0}^{k}\sum_{j=i}^{k-1}\frac{1}{N_{j}}}{\lambda^2(k+1)}.
\end{align}
Recalling $N_k=\lceil (k+1)^{2a} \rceil, \ a>1$, it follows that
\begin{align*}
\sum_{i=0}^{k}\sum_{j=i}^{k-1}\tfrac{1}{N_{j}}&=\sum_{i=0}^{k}\sum_{j=i}^{k-1}\tfrac{1}{\lceil (j+1)^{2a} \rceil} \le \sum_{i=0}^{k}\sum_{j=i}^{k-1}\tfrac{1}{(j+1)^{2a}} 
\le \int_{0}^{k+1}\int_{y}^{k+1}\tfrac{dxdy}{(x+1)^{2a}} \le \tfrac{1}{(2a-1)(a-1)}.
\end{align*}
Since $\sum_{i=0}^\infty\tfrac{1}{\sqrt{N_{i+1}}}<+\infty$ and $\sum_{i=0}^\infty\frac{1}{N_{i+1}}<+\infty$, we have 
$\mathbb{E}[\|T_{\lambda}(\x^k)\|^2]  \leq \tfrac{\widehat C}{k+1} = \mathcal{O}\left(\frac{1}{k+1}\right),$ where 
$ \widehat{C} \triangleq \tfrac{\left(\|\x^0-\x^*\|+D_2E_1\right)^2+D_1E_2+D_2^2E_1^2}{\lambda^2}+\tfrac{D_1}{\lambda^2(a-1)(2a-1)}.$

(b) Suppose $\x^{K+1}$ is such that $\mathbb{E}[\|T_{\lambda}(\x^{K+1})\|^2] \leq \epsilon$. From (a), for sufficiently small $\epsilon$,
\begin{align*}
\sum_{k=0}^KN_k&\le\sum_{k=0}^{\lceil \widehat{C}/\epsilon \rceil-1}N_k=\sum_{k=0}^{\lceil \widehat{C}/\epsilon \rceil-1}\lceil (k+1)^{2a} \rceil \le 2\sum_{k=0}^{\lceil \widehat{C}/\epsilon \rceil-1}(k+1)^{2a} \\
&\le  2\int_{x=0}^{\widehat{C}/\epsilon}(x+1)^{2a} \ dx \le \tfrac{2(\widehat{C}/\epsilon+1)^{2a+1}}{2a+1}\le\left(\tfrac{\widehat{C}}{\epsilon^{2a+1}}\right).
\end{align*}
\end{proof}

\subsection{Convergence analysis of (VR-SPP) under strong monotonicity}\label{sec:3.4} Next, we derive a rate statement under a strong monotonicity assumption on $T$. We begin by deriving a bound on $\|J_{\lambda}^T(\x^k)-\x^*\|$, akin to \cite[Prop.~3]{corman2014generalized}.

\begin{lemma}\label{pps1t} \em
Let Assumption~\ref{smonotonet} hold and let $\lambda>0$. Assume $\x^*\in T^{-1}(0)$ is a solution. Then we have the following for all $k$:
\begin{align*}
\|J_{\lambda}^T(\x^k)-\x^*\|\le(1+\sigma\lambda)^{-1}\|\x^k-\x^*\|.
\end{align*}
\end{lemma}
\begin{proof}
Suppose $y_{k+1}=J_{\lambda}^T(\x^k)$ or $\x^k = (I+\lambda T)(y_{k+1}) = y_{k+1} + \lambda v_{k+1}$ where $v_{k+1} \in T(y_{k+1})$. In addition, $\x^* = J_{\lambda}^T(\x^*)$ or $\x^* = \x^*+\lambda v^*$ where $0 = v^* \in T(\x^*)$. Since $T$ is $\sigma$-strongly monotone, we have that 
\begin{align*}
\|v_{k+1}-v^*\|\|y_{k+1}-\x^*\| & \geq (v_{k+1}-v^*)^\mathsf{T}(y_{k+1}-\x^*) \geq \sigma \|y_{k+1}-\x^*\|^2 \\
\implies \|v_{k+1}-v^*\| & \geq \sigma \|y_{k+1}-\x^*\|.
\end{align*}
Consequently, we may bound $\|\x^k-\x^*\|^2$ from below as follows.
\begin{align*}
\|\x^k-\x^*\|^2& \notag =\|y_{k+1}+\lambda v_{k+1}-(\x^*+\lambda v^*)\|^2\\
		& = \|y_{k+1}-\x^*\|^2 +\lambda^2\|v_{k+1}-v^*\|^2 + 2\lambda (y_{k+1}-\x^*)^\mathsf{T}(v_{k+1}-v^*) \\
		&  \geq (1+2\sigma\lambda) \|y_{k+1}-\x^*\|^2+\lambda^2\|v_{k+1}-v^*\|^2
         \geq (1+\sigma \lambda)^2\|y_{k+1}-\x^*\|^2,
\end{align*}
where the first inequality follows from the 
strong monotonicity of $T$ and the second inequality is a consequence of $\|v_{k+1}-v^*\|^2\ge\sigma^2\|y_{k+1}-\x^*\|^2$.  It follows that
\begin{align*}
\|\x^k-\x^*\|^2\ge(1+\sigma\lambda)^2\|y_{k+1}-\x^*\|^2=(1+\sigma\lambda)^2\|J_{\lambda}^T(\x^k)-\x^*\|^2.
\end{align*}
\end{proof}

We conclude by deriving a rate 
under a strong monotonicity requirement.

\begin{proposition}[{\bf Linear convergence of ({\bf VR-SPP}) under strong monotonicity}]\label{ratePvs1} \em
Let Assumptions~\ref{smonotonet} and \ref{ass_sfo} hold. Suppose $\{\x^k\}$ denotes a sequence generated by ({\bf VR-SPP}) and $\x^*$ denotes a unique solution to $0 \in T(\x)$. Furthermore, suppose $\|\x^0-\x^*\| \leq M$. Then the following hold. \\
(a) Suppose $N_{k}=\lfloor \us{\rho^{-(k+1)}} \rfloor$ where $0<\rho<1$ and $q\triangleq \tfrac{1+d}{(1+\sigma\lambda)^{2}} < 1$ for $d$ sufficiently small. Then $\mathbb{E}[\|\x^k-\x^*\|^2] \leq \tilde{D} \tilde{\rho}^k$ where $\tilde D > 0$  and $\tilde \rho = \max\{q,\rho\}$ if $q \neq \rho$ and $\tilde \rho \in (q,1)$ if $q = \rho$. 
\\ (b) The oracle complexity to ensure that $\mathbb{E}[\|\x^{K+1}-x^*\|^2] \le \epsilon$ satisfies $\sum_{k=0}^KN_k \le
\mathcal{O}\left(\frac{1}{\epsilon}\right).$
\end{proposition}
\begin{proof}
(a) By invoking Lemma~\ref{pps1t} and Prop.~\ref{prop-res-error},  we obtain the following:
\us{ \begin{align}
\mathbb{E}[\|\x^{k+1}&-\x^*\|^2]  \leq (1+d)\mathbb{E}[\|y_{k+1}-\x^*\|^2] + (1+\tfrac{1}{d})\mathbb{E}[\|\x^{k+1}-y_{k+1}\|^2]\notag \\
	& \le q\mathbb{E}[\|\x^k-\x^*\|^2]+(1+\tfrac{1}{d})\tfrac{D}{{N_{k}}}, \label{eq58p-t}
q \triangleq  \tfrac{(1+d)}{(1+\sigma\lambda)^{2}}, D \triangleq \nu_1^2(2C^2+2\|\x^*\|^2)+\nu_2^2, 
\end{align}
and $d > 0$ is chosen such that $\tfrac{(1+d)}{(1+\sigma \lambda)^2}<1$.}
Recall that $N_k$ can be bounded as seen next.
\us{\begin{align}
N_k= \lfloor \rho^{-(k+1)} \rfloor \ge \left\lceil \tfrac{1}{2}\rho^{-(k+1)} \right\rceil \ge \tfrac{1}{2}\rho^{-(k+1)}. \label{eq59p-t}
\end{align}}
We now consider three cases. \\
\noindent  (i): $q<\rho<1$. \us{Using \eqref{eq59p-t} in \eqref{eq58p-t} and defining $\bar D=2(1+\tfrac{1}{2})D$, $\tilde D \triangleq (M+\tfrac{\bar D}{1-q/p})$, we obtain} 
\begin{align*}
\notag\mathbb{E}[\|&\x^{k+1}-\x^*\|^2] \le q\mathbb{E}[\|\x^k-\x^*\|^2]+\tfrac{(1+\frac{1}{d})D}{N_{k}}\leq\us{q}\mathbb{E}[\|\x^k-\x^*\|^2]+\bar D\rho^{k+1} \\
&\le q^{k+1}\|\x^0-\x^*\|+\bar D\sum_{j=1}^{k+1}q^{k+1-j}\rho^j 
\le Mq^{k+1}+\bar D\rho^{k+1}\sum_{j=1}^{k+1}(\tfrac{q}{\rho})^{k+1-j}\le \tilde{D}\ssc{\rho}^{k+1}.
\end{align*}
\noindent  (ii): $\rho<q<1$. Akin to (i) and defining $\tilde D$ apprioriately,  $\mathbb{E}[\|\x^{k+1}-\x^*\|^2] \le\tilde{D}\ssc{q}^{k+1}$. \\
\noindent (iii): $\rho=q<1$. 
If $\tilde{\rho} \in (q,1)$ and $\widehat{D} > \tfrac{1}{\ln(\tilde{\rho}/q)^e}$, proceeding similarly we obtain
\begin{align*}
\notag\mathbb{E}[&\|\x^{k+1}-\x^*\|^2] \le q^{k+1}\mathbb{E}[\|\x^0-\x^*\|^2]+\bar D\sum_{j=1}^{k+1}q^{k+1}\le Mq^{k+1}+\bar D\sum_{j=1}^{k+1}q^{k+1}\\
& = Mq^{k+1}+\bar D(k+1)q^{k+1} 
\overset{\tiny \cite[\mbox{Lemma}~4]{ahmadi2016analysis}}{\le}  \tilde{D} \tilde{\rho}^{k+1}, \mbox{ where } \tilde D \triangleq (M+\widehat{D}). 
\end{align*}
\noindent Thus, $\{\x^k\}$ converges linearly in an expected-value sense. \\ 
(b) Case (i): If $q<\rho<1$. From (a), it follows that
\begin{align*}
\mathbb{E}[\|\x^{K+1}-\x^*\|^2] &\le \tilde{D}\ssc{\rho}^{K+1}\leq \  \epsilon \Longrightarrow  K \ge  \log_{1/\ssc{\rho}}(\tilde{D}/\epsilon) - 1 .
\end{align*} 
If $K = \lceil \log_{1/\ssc{\rho}}(\tilde{D}/\epsilon)\rceil - 1$, then ({\bf VR-SPP}) requires $\sum_{k=0}^KN_k$ evaluations. Since $N_k=\lfloor \rho^{-(k+1)} \rfloor \le \rho^{-(k+1)}$, then we have 
\us{\begin {align*}
 \quad \sum_{k=0}^{\lceil \log_{1/\ssc{\rho} }(\tilde{D}/\epsilon)\rceil -1}\rho^{-(k+1)}  =
\sum_{t=1}^{\lceil \log_{1/\ssc{\rho} }(\tilde{D}/\epsilon)\rceil}\rho^{-t} 
 \le \tfrac{1}{\rho^2\left(\tfrac{1}{\rho}-1\right)}\left(\tfrac{1}{\rho}\right)^{\lceil \log_{1/\ssc{\rho}}(\tilde{D}/\epsilon)\rceil}\\ 
 \le
\tfrac{1}{\rho\left(\tfrac{1}{\rho}-1\right)}\left(\tfrac{1}{\rho}\right)^{
\log_{1/\ssc{\rho}}(\tilde{D}/\epsilon)+1} 
 \le \tfrac{1}{\left(1-\rho\right)}\left(\tfrac{1}{\rho}\right)^{\log_{1/\rho}(\tilde{D}/\epsilon)}
  \le \tfrac{1}{(1-\rho)}\left(\tfrac{\tilde{D}}{\epsilon}\right).
\end{align*}}
We omit cases (ii) and (iii) which lead to similar complexities.
\end{proof}
\noindent {\bf Remark 1.} \us{Several aspects deserve additional emphasis.\\ 
\noindent (a)  {\em Rates and asymptotics.} To the best of our knowledge,  we remain unaware of a.s. convergence (an exception being Bianchi~\cite{bianchi2016ergodic}) and rate statements under state-dependent noise requirements for either monotone or strongly monotone inclusions. Note that Bianchi~\cite{bianchi2016ergodic} develops a stochastic proximal-point scheme that does not come equipped with rate statements; however, since the resolvent requires computing at every step, its practical behavior for large-scale regimes tends to be poorer when the resolvent is challenging to compute.\\
\noindent (b) {\em Algorithm parameters.} The inner steplengths of the (SA) scheme utilize the user-specified proximal parameter while the outer steps in {\bf (VR-SPP)} employ a constant user-specified steplength. The sample-sizes are also free of algorithm parameters. The minimum number of steps $J_1$ in each inner step do require knowing $M_1$ but this may be possible to obviate by using an increasing sequence of minimal number of steps. \\
\noindent (c) {\em Lipschitzian parameters.} Unlike (SA) schemes, this scheme does not tend to be as hampered by ill-conditioning since outer steplengths are not contingent on Lipschitzian parameters while inner steps are also user-specified. 
}

\noindent (d) {\em Practical implementations of VR schemes.} To achieve an error of $\epsilon=10^{-3}$, ({\bf VR-SPP}) requires $\mathcal{O}(10^9)$ samples for the maximal monotone mapping or $\mathcal{O}(10^3)$ samples for the strongly monotone mapping, respectively. In a typical finite sum optimization problem $\min\mathbb{E}[f(\x)]:=\tfrac{1}{n}\sum_{i=1}^nf_i(\x)$, $n$ is larger than $10^9$, thus the number of samples needed in both of our schemes is not expensive. If a sharper rate is wanted in this case, we just set $N_k=n$ to prevent unboundedness of $N_k$.

\subsection{Broader applicability of scheme for monotone stochastic inclusions} \label{sec:3.5}
The variance-reduced proximal-point framework has broader applicability  in addressing the {\em stochastic} counterpart of generalized equations~\cite{robinson83generalized},
a class of problems that has seen recent study via sample-average
approximation (SAA) techniques~\cite{chen19convergence}.  Formally, the
stochastic generalized equation requires an $\x \in\Real^n$
such that 
\begin{align} \tag{SGE}
0 \in \mathbb{E}[T(\x,\xi(\omega))], 
\end{align}
where the components of the map $T$ are denoted by $T_i$, $i=1,\dots,n$, $\xi:\Omega \to \Real^d$ is a random variable, $T_i:
\Real^n \times \Omega \rightrightarrows \Real^n$ is a set-valued map, $\mathbb{E}[\cdot]$ denotes the expectation, and the associated
probability space is given by
$(\Omega, {\cal F}, \mathbb{P})$.The expectation of a set-valued map leverages the Aumann integral~\cite{aumann1965integrals} and is formally defined as
$\mathbb{E}[T_i(\x,\xi(\omega))]= \left\{ \int v_i(\omega)dP(\omega)\mid \quad v_i(\omega)\in T_i(\x,\xi(\omega)) \right\}.$
Consequently, the expectation $\mathbb{E}[T(\x,\omega)]$ can be defined as a Cartesian product of the sets $\mathbb{E}[T_i(\x,\omega)]$, defined as 
$\mathbb{E}[T(\x,\omega)]  \ \triangleq \ \prod_{i=1}^n\mathbb{E}[T_i(\x,\omega)].$ 
We motivate (SGE) by considering some examples.  Consider
the stochastic convex optimization
problem~\cite{dantzig2010linear,birge2011introduction,shapiro2014lectures} given by
${\displaystyle \min_{\x \in \mathcal{X}}} \, \mathbb{E}[g(\x,\omega)],$ where
$g(\bullet,\omega)$ is a convex function for every $\omega$
and $\mathcal{X}$ is a closed and convex set. Such a problem can be
equivalently stated as $0 \in T(\x) \triangleq
\mathbb{E}[G(\x,\omega)] + \mathcal{N}_\mathcal{X}(\x)$, where $G(\x,\omega) = \partial g(\x,\omega)$ and $N_\mathcal{X}(\x)$ denotes the normal cone of $\mathcal{X}$ at $\x$. In
fact, both the single-valued~\cite{jiang08stochastic,juditsky2011solving,shanbhag2013stochastic} and  multi-valued~\cite{ravat17existence} stochastic variational inequality problems
can be cast as stochastic inclusions as well as seen by $0 \in T(\x)
\triangleq \mathbb{E}[F(\x,\omega)] + \mathcal{N}_\mathcal{X}(\x)$,
where $F(\bullet,\omega)$ is either single-valued or set-valued. This
introduces a pathway for examining stochastic analogs of traffic
equilibrium~\cite{ravat17existence} and Nash equilibrium
problems~\cite{ravat2011characterization} as well as a host of other problems
subsumed by variational inequality problems~\cite{facchinei2007finite}.

\section{Partially distributed  schemes for hierarchical potential games}\label{sec:4}
In this section, we again consider an $\Nbold$-player noncooperative game $\Gcal$ where the $i$th player's problem is defined by the parametrized hierarchical problem (Player$_i(\x^{-i})$), defined in Section~\ref{sec:spp}, and restated next.  
\begin{align} \tag{Player$_i(\x^{-i})$}
    \min_{\x^i \in \Xscr^i} \ f_i(\x^i,\x^{-i}) \triangleq \mathbb{E}\left[\tilde{f}_i(\x^i, \y^i(\x,\omega),\x^{-i},\omega)\right], 
\end{align}
where $\tilde{f}_i(\x^i,\y^i(\x,\omega), \x^{-i},\omega) \triangleq \tilde{g}_i(\x^i,\x^{-i},\omega) +\tilde{h}_i(\x^i,\y^i(\x,\omega),\omega)$.
In this section, under the assumption that for any $\omega \in \Omega$, $\tilde{f}_i(\x^i,\x^{-i},\y^i(\x,\omega),\omega)$ is convex in $\x^i$ over $\Xscr_i$ for any $\x^{-i} \in \prod_{j\neq i} \Xscr_j$ and $\Gcal$ admits a suitable
potentiality assumption,   we propose and prove the
asymptotic convergence of an  asynchronous smoothed proximal best-response
scheme (and its relaxed counterpart) for computing an approximate  Nash
equilibrium in Section~\ref{sec:4.1}. This scheme relies on computing
increasingly accurate best-responses, which are provided via a zeroth-order
method that processes the implicit form of the hierarchical problem. To this
end,   we introduce and  discuss a zeroth-order framework for computing an
approximate solution of a hierarchical  stochastic convex program in
Section~\ref{sec:4.2}. 

\subsection{A smoothing-based framework for hierarchical games}\label{sec:4.1}
Recall that for $i \in \{1, \cdots, \Nbold\}$,  $f_i(\bullet,\x^{-i})$ is
convex but not necessarily $L$-smooth on $\Xscr_i$ for every $\x^{-i} \in
\Xscr^{-i}$. In fact, it may be recalled that $$ f_{i} (\x^i,\x^{-i})
\triangleq \mathbb{E}[ \tilde{f}_i(\x^i, \y_i(\x,\omega), \x^{-i},\omega)]$$ and computing even a subgradient is
not immediate. Instead, the function may be evaluated, suggesting the
development of a gradient-free method facilitated by  introducing a randomized
smoothing of $f_i$. This smoothing allows for both claiming the
$L_{i,\eta}$-smoothness of the smoothed function (for a suitable $L_{i,\eta}$)
and providing a relation between $f_i$ and $f_{i,\eta}$. Such smoothing
techniques have a storied history, traceable to the 1900s~\cite{steklov1} and
employed for resolving nonsmooth convex
optimization~\cite{yousefian10adaptive,nesterov17} and monotone
games~\cite{yousefian2013self}.  We formally define an $\eta$-smoothed game $\Gcal_{\eta}$, given a game $\Gcal \in \Gscr^{\rm chl}_{\rm pot}$. We make the following assumption.

\begin{assumption} \label{bound-grad} \em For $i = 1, \cdots, \Nbold$,  $f_i(\bullet,\x^{-i})$ has uniformly bounded subgradients over $\Xscr$, i.e. for every $\x \in \Xscr$, we have that $\|\tilde{d}_i\| \leq \mathcal{L}_0$ where $\tilde{d}_i \in \partial_{\x^i} f_i(\x^i,\x^{-i}).$
\end{assumption}
Naturally, one might ask if such an assumption is indeed valid in the current setting. Inspired by~\cite{patriksson99stochastic} and ~\cite[Prop.~1]{cui21zeroth}, we provide Prop.~\ref{prop-lip-func} in Appendix {\bf A.2.} that provides conditions under which the above assumption holds. 

\begin{definition}[{\bf An $\eta$-smoothed noncooperative game $\Gcal_{\eta}$}] \em Consider a game $\Gcal \in \Gscr^{\rm chl}_{\rm pot}$ in which the $i$th player solves (Player$_i(\x^{-i})$). Suppose $\Gcal_{\eta}$ denotes a related game in which for $i = 1, \cdots, \Nbold$, the $i$th player's smoothed problem is defined as
\begin{align} \tag{Player$_{i,\eta}(\x^{-i})$}
\min_{\x^i \in \Xscr^i} \ f_{i,\eta}(\x^i,\x^{-i}) \triangleq \mathbb{E}_{u_i \in \mathbb{B}_i} \left[ \mathbb{E}\left[ \tilde{f}_i(\x^i +\eta u_i,\y^i(\x^i +\eta u_i,\x^{-i},\omega), \x^{-i},\omega)\right] \right],
\end{align}
where 
$\mathbb{B}_i \subseteq \Real^{n_i}$ is a sphere centered at the origin and $u_i$ is independent of $\omega$.
\end{definition}

Absent such a smoothing, while techniques are available for resolving this
hierarchical problem (which is in effect an MPEC)
(cf.~\cite{luo96mathematical,outrata98nonsmooth}), we remain unaware of
techniques that can provide $\epsilon$-solutions of such problems in finite
time. In fact, in recent work~\cite{cui21zeroth}, we have developed a
zeroth-order framework for precisely such problems and in this paper, we
consider a variant of such a scheme for contending with the proximal best-response problem in Section~\ref{sec:4.2}. Next, we discuss the impact of smoothing on the convexity and Lipschitz continuity of the gradient  of $f_i(\bullet, \x^{-i})$ via a result from ~\cite[Lemma~1]{cui21zeroth}.   

\begin{lemma} [{\bf Convexity and smoothness of $f_{i,\eta}(\bullet,\x^{-i}$}]\label{lem:conv-heta} \em For $i = 1, \cdots, \Nbold$,  suppose  $f_{i,\eta}(\bullet,\x^{-i})$ is
    defined as \us{$f_{i,\eta} (\x^i,\x^{-i}) \triangleq \mathbb{E}_{u^i\in \mathbb{B}_i} [f_i(\x^i+\eta u_i,\x^{-i})]$} where $u_i$
    is uniformly distributed in a ball $\mathbb{B}_i \subseteq \Real^{n_i}$. Then
    there exists an $(\alpha_i,
    \beta_i)$ such that the following hold.

    \noindent (a) $f_{i,\eta}(\bullet,\x^{-i})$ is convex and  $\alpha_i/\eta$-smooth for every $\x^{-i} \in \Xscr^{-i}$, i.e. 
    $$\|\nabla_{\x^i} f_{i,\eta}(\x^i,\x^{-i})-\nabla_{\x^i} f_{i,\eta}(\y^i,\x^{-i}) \| \leq \tfrac{\alpha_i}{\eta} \|\x^i-\y^i\| \mbox{ for any } \x^i, \y^i \in \Xscr_i.$$

    \noindent (b) For any $\x \in \Xscr$,  \begin{align}
        \label{bd-smooth} \us{f_i(\x) \leq f_{i,\eta}(\x) \leq f_i(\x) + \eta \beta_i}.
\end{align} 
Furthermore, $\bar \alpha = \displaystyle \max_{i=1,\cdots, \Nbold} \alpha_i$ and $\bar \beta = \displaystyle \max_{i=1, \cdots, \Nbold} \beta_i$.
\qed
\end{lemma}
\begin{proof} We provide a proof sketch.  (a) While convexity of $f_{i,\eta}(\bullet,\x^{-i})$ for any $\x^{-i} \in \Xscr^{-i}$ follows from ~\cite[Lemma~2(a)]{yousefian10convex}, $\tfrac{\alpha_i}{\eta}-$smoothness of $f_{i,\eta}(\bullet,\x^{-i})$ follows from ~\cite[Lemma~2(c)]{yousefian10convex} by invoking the uniform boundedness of the subgradients. (b) The left-hand side of \eqref{bd-smooth} is a consequence of employing Jensen's inequality while the right-hand side is a result of the subgradient inequality and the uniform boundedness of subgradients.     
\end{proof}
\noindent {\bf Comment.} Note that if Assumption~\ref{bound-grad} is weakened to the uniform bound that  $\|\tilde{d}_i\|\leq \mathcal{B}_0\|\x\|^2 + \mathcal{L}_0$ where $\tilde{d}_i \in \partial_{\x^i} f_i(\x)$, we may still invoke this result under the requirement that $\Xscr_i$ is a bounded set for $i=1, \cdots, \Nbold$. \\

Throughout this section, we make the following ground assumption.
\begin{tcolorbox}{\bf Ground Assumption (G3)}
Consider the $\Nbold$-player game $\Gcal$ in which the $i$th player is defined as (Player$_i(\x^{-i}))$ for $i=1, \cdots, \Nbold$. For $i=1, \cdots, \Nbold$, the parametrized lower-level mapping $F_i(\bullet,\x,\omega)$ is a strongly monotone map for $\x \in \Xscr$ and for every $\omega \in \Omega$. Assumption~\ref{bound-grad} holds, $\mathscr{G}_{\eta}$ is a potential game for any $\eta > 0$, $P_{\eta}(\x)$ denotes its potential function, and $P_{\eta}(\x) \geq \tilde{P}_{\eta}$ for every $\x \in \Xscr$. 
\end{tcolorbox}

We should emphasize that in many settings, potentiality of $\Gcal$ implies potentiality of $\Gcal_{\eta}$. For purposes of brevity, we do not discuss this further.  Associated with $\Gcal$, we define the proximal best-response~\cite{pang09nash} of player $i$  as follows, given rival decisions $\x^{-i}$.
\begin{align} \tag{PBR$_i({\x})$}
    B_{i}(\x) & \triangleq \argmin{\vv^i \in \Xscr_i} \ \left[ f_i(\x^i,\x^{-i})+ \tfrac{c}{2}\|\vv^i-\x^{i}\|^2 \right] \\
    \mbox{ where } f_i(\x^i,\x^{-i}) & \triangleq 
\mathbb{E}\left[\tilde{f}_{i} (\x^i,\y^i(\x^i, \x^{-i},\omega), \x^{-i},\omega)\right]. \notag
\end{align}
Similarly, we may define  the $\eta$-{\em smoothed} proximal best-response of player $i$ as follows. 
\begin{align} \tag{SPBR$_{i,\eta}({\x})$}
    B_{i,\eta}(\x) & \triangleq \argmin{\vv^i \in \Xscr_i} \ \left[ f_{i,\eta}(\x^i,\x^{-i}) + \tfrac{c}{2}\|\vv^i-\x^{i}\|^2 \right] \\ 
    \mbox{ where } f_{i,\eta}(\x^i,\x^{-i}) & \triangleq \mathbb{E}_{u_i \in \mathbb{B}_i} \left[\mathbb{E}\left[\tilde{f}_{i} (\x^i+\eta u_i,\y^i(\x^i+\eta u_i, \x^{-i},\omega), \x^{-i},\omega)\right]\right]. \notag
\end{align}
 Our next result  provides a
 deeper understanding of the  relationship between $B_i(\x)$ and $B_{i,\eta}(\x)$.

    \begin{proposition}[{\bf Proximal best-response map (PBR) and its smoothed variant (SPBR$_{\eta}$)}]\label{prop-spbr} \em Consider a game $\Gcal \in \Gscr^{\rm chl}_{\rm pot}$. For any $i \in \{1, \cdots, \Nbold\}$, suppose $f_i(\bullet,\x^{-i})$ is a convex function for any $\x^{-i} \in \Xscr_{-i}$ and $\Xscr_i \subseteq \Real^{n_i}$ is a closed and convex set. For any $i \in \{1, \cdots, \Nbold\}$, suppose $f_{i,\eta}(\bullet,\x^{-i})$ denotes the $\eta-$smoothing of $f_i(\bullet,\x^{-i})$. Suppose B$_i(\x)$ and B$_{i,\eta}(\x)$ denote the proximal best-response and smoothed proximal-response for $i \in \{1, \cdots, \Nbold\}.$ Then the following hold for any $i \in \{1, \cdots, \Nbold\}$.\\

\noindent (a) Both $B_i(\x)$ and $B_{i,\eta}(\x)$ are single-valued maps for $\x \in \Xscr$, {where $\Xscr \triangleq \prod_{i=1}^\Nbold \Xscr_i$.}

\smallskip

\noindent (b)  For any $i \in \{1, \cdots, \Nbold\}$ and $\x^{-i} \in \prod_{j \neq i} \Xscr_j$, $f_{i,\eta}(\bullet,\x^{-i})$ converges continuously to $f_i(\bullet,\x^{-i})$, i.e. $f_{i,\eta}(\x^i_{\eta},\x^{-i}) \to f_i(\x^i,\x^{-i})$ for all $\x^i_{\eta} \to \x^i$ where $\x^i_{\eta} \in \Xscr_i$ and $\x^i \in \Xscr_i$. Further, $f_{i,\eta}(\bullet,\x^{-i})$ converges uniformly to $f_i(\bullet,\x^{-i})$ on every bounded subset of $\Real^{n_i}$.


    \end{proposition}
\begin{proof}
    (a) follows from strong convexity of the proximal problems while (b) is a consequence of ~\cite[Cor.~3.3]{ermoliev95minimization}. 
\end{proof}
In fact, it can be shown that a fixed-point of the $\eta$-smoothed proximal best-response map is an $\eta$-approximate Nash equilibrium.

\begin{proposition}[{\bf Fixed-point of (SPBR$_{\eta}$) is NE of $\Gcal_{\eta}$}] \em Consider an $\Nbold$-player noncooperative game $\Gcal$ where the $i$th player solves (Player$_i(\x^{-i})$), given rival decisions $\x^{-i}$. For $i = 1, \cdots, \Nbold$, suppose $f_i(\bullet,\x^{-i})$ is convex on $\Xscr_i$ for any $\x^{-i} \in \Xscr_{-i}$.  Suppose $\x^{\eta} \triangleq \{\x^{1,\eta}, \cdots, \x^{\Nbold,\eta}\}$ is a fixed point of the $\eta$-smoothed best-response map. Then the following hold.

    \noindent (a) $\x^{\eta}$ is a fixed point of (SPBR$_{\eta}(\bullet)$), i.e. $\x^{i,\eta} = B_{i,\eta}(\x^{i,\eta},\x^{-i,\eta})$ for $i = 1, \cdots, \Nbold$ if and only if $\x^{\eta}$ is a Nash equilibrium of $\Gcal_{\eta}$. 

    \noindent (b) If $\x^\eta$ is a fixed point of SPBR$_{\eta}(\bullet)$, then $\x^{\eta}$ is an ${\eta} \bar{\beta}$-Nash equilibrium of $\Gcal$ where $\bar{\beta} \triangleq \displaystyle \max_{i \in \{1, \cdots, \Nbold\}} \beta_i$.

\end{proposition}
\begin{proof} (a) follows directly from \cite[Prop.~1.5]{pang09nash}. We proceed to prove (b).  Suppose $\x^{\eta} \triangleq \{\x^{1,\eta}, \cdots, \x^{\Nbold,\eta}\}$ is a fixed point of the $\eta$-smoothed best-response map (SPBR$(\bullet)$). Then we have that 
    \begin{align*}
        \x^{j,\eta} & = B_{j,\eta}{(\x^{j,\eta},\x^{-j,\eta})}, \qquad j = 1, \cdots, \Nbold.
    \end{align*}
    From (a), we have that $\x^{\eta}$ is a Nash equilibrium of $\Gcal_{\eta}$.  It follows that 
    \begin{align}\label{nash-smooth}
     f_{j,\eta}(\x^{j,\eta},\x^{-j,\eta}) \leq f_{j,\eta}(\x^j, \x^{-j,\eta}), \qquad \forall \x^j \in \Xscr_j \mbox{ for }  j = 1, \cdots, \Nbold. 
    \end{align}
    By leveraging the property of the smoothed function $f_{j,\eta}(\bullet,\x^{-j})$, we have that  
    \begin{align*}
        f_{j}(\x^{j,\eta},\x^{-j,\eta}) & \overset{\eqref{bd-smooth}}{\leq} f_{j,\eta}(\x^{j,\eta},\x^{-j,\eta})  \\
                                        & \overset{\eqref{nash-smooth}}{\leq} f_{j,\eta}(\x^j, \x^{-j,\eta}), \qquad \quad \qquad \forall \x^j \in \Xscr_j \\
                                        & \overset{\eqref{bd-smooth}}{\leq} f_{j} (\x^j, \x^{-j,\eta})  + \eta \beta_j,  \qquad  \quad \forall \x^j \in \Xscr_j  
    \end{align*}
 It follows that $\x^{\eta}$ is an $\ic{\eta} \bar{\beta}$-Nash equilibrium of $\Gcal$ where $\bar{\beta} \triangleq \max_{i \in \{1, \cdots, \Nbold\}} \beta_i$.  
\end{proof}

We now turn to the question of deriving error bounds on the best-response residual for the original game by leveraging solutions of the $\eta$-smoothed game. This avenue requires proving a simple result  that relates $B_i(\x)$ and $B_{i,\eta}(\x)$ for any $i \in \{1, \cdots, \Nbold\}$ and $\eta > 0$. 

    \begin{proposition}[{\bf Relating equilibria of $\Gcal_{\eta}$ to Equilibria of $\Gcal$}]\em Suppose the conditions of Prop.~\ref{prop-spbr} hold. Then the following hold. 

\noindent (i) For any $\x \in \Xscr$, we have that 
$ \|B_i(\x) - B_{i,\eta}(\x)\|^2 \leq \tfrac{2\eta \beta}{c}.$
\smallskip

\noindent (ii)  Suppose $\x \in \Xscr$.  Then the best-response residual for the original game $\Gcal$ is bounded as follows. 
\begin{align*}
 \sum_{i=1}^{\Nbold} \|\x^i - B_i(\x)\|^2 & \leq  2\sum_{i=1}^{\Nbold} \|\x^i - B_{i,\eta}(\x)\|^2 + \tfrac{8\Nbold \eta^2 \beta^2}{c^2}. 
\end{align*}

\smallskip

\noindent (iii) Suppose $\x_{\eta}^{*} \triangleq \left\{\x^{1,*}_{\eta} \cdots, \x^{\Nbold,*}_{\eta}\right\}$ denotes an equilibrium of $\Gcal_{\eta}$. Then the best-response residual for the original game is bounded as follows. 
\begin{align*}
    \sum_{i=1}^{\Nbold} \|\x^{i,*}_\eta - B_i(\x^*_\eta)\|^2 & \leq  \tfrac{8\Nbold \eta^2 \beta^2}{c^2}. 
\end{align*}

\end{proposition}
\begin{proof}
    By strong convexity, we have that 
    \begin{align*}
        f_{i}(B_{i,\eta}(\x),\x^{-i}) & \geq f_i(B_i(\x),\x^{-i}) + \tfrac{c}{2} \|B_i(\x) - B_{i,\eta}(\x)\|^2 \\ 
            f_{i,\eta}(B_{i}(\x),\x^{-i}) & \geq f_{i,\eta}(B_{i,\eta}(\x),\x^{-i}) + \tfrac{c}{2} \|B_i(\x) - B_{i,\eta}(\x)\|^2. 
    \end{align*}
    Adding the above inequalities, we obtain the result as follows. 
    \begin{align*}
        c \|B_i(\x) - B_{i,\eta}(\x)\|^2 & \leq f_{i}(B_{i,\eta}(\x),\x^{-i})  - f_{i,\eta}(B_{i,\eta}(\x),\x^{-i}) + f_{i,\eta}(B_{i}(\x),\x^{-i}) - f_i(B_i(\x),\x^{-i}) \notag \\
                                         & \leq |f_{i}(B_{i,\eta}(\x),\x^{-i})  - f_{i,\eta}(B_{i,\eta}(\x),\x^{-i})| + |f_{i,\eta}(B_{i}(\x),\x^{-i}) - f_i(B_i(\x),\x^{-i})| \notag \\
                                          & \leq 2\eta \beta. 
    \end{align*}

\noindent (ii) This result follows by noting that 
\begin{align*}
\notag \sum_{i=1}^{\Nbold} \|\x^i - B_i(\x)\|^2 & \leq 2\sum_{i=1}^{\Nbold} \left(\|\x^i - B_{i,\eta}(\x)\|^2 +  \|B_i(\x)-B_{i,\eta}(\x)\|^2\right) \\
                                         & \leq 2\sum_{i=1}^{\Nbold} \|\x^i - B_{i,\eta}(\x)\|^2 + \tfrac{8\Nbold \eta^2 \beta^2}{c^2}. 
\end{align*}

\noindent (iii) This follows from (ii) and by noting that $\sum_{i=1}^{\Nbold}\|\x^i - B_{i,\eta}(\x)\|^2 = 0$ for $\x =\x_{\eta}^{*}$.  
\end{proof}

We will now examine the question of whether the sequence of equilibria $\{\x^*_{\eta}\}_{\eta \downarrow 0}$, where $\x^*_{\eta}$ is an equilibrium of the smoothed game $\Gcal_{\eta}$, converges to an equilibrium of $\Gcal$. We begin by providing the following definition for multi-epiconvergence of a collection of functions  from~\cite[Def.~1]{gurkan09approximations}.

  \begin{definition}[{\bf Multi-epiconvergence}] \em Suppose $f_{i,\eta}: \Real^{n_i} \to \Real$ for $i = 1, \cdots, \Nbold$. The family of functions $\{f_{i,\eta}\}_{i=1}^{\Nbold}$ multi-epiconverges to the functions $\{f_i\}_{i=1}^{\Nbold}$ on $\Xscr$ if the following two conditions hold for every $i = 1, \cdots, \Nbold$ and every $\x \in \Xscr$. 

        \smallskip

        \noindent ({\bf ME(i)}) For every sequence $\{\x^{-i}_{\eta}\} \subset \Xscr_{-i}$ converging to $\x^{-i}$, there exists a sequence $\{\x^{i}_{\eta}\} \subset \Xscr_i$ converging to $\x^i$ such that 
            \begin{align*}
                \limsup_{\eta \to 0} f_{i,\eta}(\x^{i}_{\eta}, \x^{-i}_{\eta}) \leq f_i(\x^i,\x^{-i}).
            \end{align*}

        \smallskip 
        \noindent ({\bf ME(ii)}) For every sequence $\{\x_{\eta}\} \subset \Xscr$ converging to $\x$, 
    \begin{align*}
                \liminf_{\eta \to 0} f_{i,\eta}(\x^{i}_{\eta}, \x^{-i}_{\eta}) \geq f_i(\x^i,\x^{-i}).
            \end{align*}

    \end{definition}
  In~\cite{gurkan09approximations}, by leveraging the property of multi-epiconvergence, convergence of the sequence of approximate equilibria to its true counterpart is proven.  We reproduce this result here. 

    \begin{theorem}[{\bf Convergence of approximate Nash equilibria~\cite[Thm.~1]{gurkan09approximations}}] \label{thm-conv-xeta}\em Consider the game $\Gcal$ and suppose the following hold.  
\smallskip 

\noindent ({\bf C.I.}) For $i=1, \cdots, \Nbold$, suppose $\Xscr_i \subseteq \Real^{n_i}$ is a closed and convex set. 

\smallskip

\noindent ({\bf C.II.}) Suppose that the family $\{\{ f_{i,\eta}\}_{i=1}^{\Nbold} \}$ multi-epiconverges to the functions $\{f_i\}_{i=1}^{\Nbold}$. 
\smallskip

If the sequence $\{\x^*_\eta\}$ converges to $\x^*$  where $\x^*_{\eta}$ is an equilibrium of $\Gcal_{\eta}$ with functions $\{f_{i,\eta}\}_{i=1}^{\Nbold}$, then $\x^*$ is a Nash equilibrium of $\Gcal$.  
    \end{theorem}

    Note that Theorem~\ref{thm-conv-xeta} does not necessitate even the
    convexity of the player-specific objectives. Naturally, this result
    provides asymptotic guarantees but does not address the computability of
    the $\eta$-smoothed equilibrium problem with nonconvex player-specific
    problems. Furthermore, in our case, our problem is blessed with convexity
    and consequently, we may employ a corollary of Theorem~\ref{thm-conv-xeta},
    restated next with an explicit prescription of the condition (Pc) from ~\cite[Cor.~1]{gurkan09approximations}.

  \begin{corollary}[{\bf Convergence of approximate Nash equilibria under convexity~\cite[Cor.~1]{gurkan09approximations}}]\label{cor-conv-xeta} \em The conclusions of Theorem~\ref{thm-conv-xeta} hold under the following conditions. 

        \smallskip

    \noindent ({\bf D.I.}) For every $i = 1, \cdots, \Nbold$ and every $\eta >0$, the function $f_{i,\eta}(\bullet,\x^{-i})$ is convex for every $\x^{-i} \in \Xscr^{-i}.$

    \smallskip

    \noindent ({\bf D.II.}) For every $i = 1, \cdots, \Nbold$, the following holds 
    \begin{align}
    \lim_{\eta \to 0} f_{i,\eta}(\x^i,\x^{-i}_{\eta}) = f_i(\x^i,\x^{-i}) 
    \end{align}
    for every $\x^i \in \Xscr_i$ and every sequence $\{\x^{-i}_{\eta}\} \subset \Xscr^{-i}$ converging to $\x^{-i} \in \Xscr^{-i}$.
\end{corollary}

{We now prove that Corollary~\ref{cor-conv-xeta} can be invoked under suitable requirements. } 

\begin{proposition}[{\bf Asymptotic convergence of $\{\x_{\eta}\}$}]\em  Consider the game $\Gcal \in \Gscr^{\rm chl}$, its  smoothed counterpart $\Gcal_{\eta}$, and the sequence $\{\x_{\eta}\}$. For $i =1, \cdots, \Nbold$,  suppose $f_i(\bullet,\x^{-i})$ is a strongly lower semicontinuous function on $\Xscr_{i}$ for every $\x^{-i} \in \Xscr_{-i}$ and  $f_i(\x^i,\bullet)$ is a continuous function for every $\x^i \in \Xscr_i$. Then $\x$ is an equilibrium of $\Gcal$. 
    \end{proposition}

    \begin{proof} To invoke Corollary~\ref{cor-conv-xeta}, it suffices to show that conditions ({\bf D.I.}) and ({\bf D.II.}) hold. 

        Since $\Gcal$ is a convex hierarchical game, we have that $f_i(\bullet,\x^{-i})$ is convex on $\Xscr_i$ for every $\x^{-i} \in \Xscr^{-i}$. We may then invoke  Lemma~\ref{lem:conv-heta} to claim that $f_{i,\eta}(\bullet,\x^{-i})$ is convex on $\Xscr_i$ for every $\x^{-i} \in \Xscr_{-i}$. Therefore, ({\bf D.I.}) holds. 

   Suppose $\x^i \in \Xscr_i$ and a sequence $\{\x^{-i}_{\eta}\} \subset \Xscr^{-i}$ converges to $\x^{-i} \in \Xscr^{-i}$. Since $f_i(\x) \leq f_{i,\eta}(\x)\leq f_i(\x) + \eta \beta_i$ (by Lemma~\ref{lem:conv-heta}(b)), it follows that
     \begin{align*}
         \lim_{\eta \to 0} \ f_i(\x^i,\x_\eta^{-i}) \leq  \lim_{\eta \to 0} \ f_{i,\eta}(\x^i,\x^{-i}_{\eta}) \leq \lim_{\eta \to 0} \ (f_i(\x^i,\x^{-i}_\eta)+\eta \beta_i). 
    \end{align*}
    By continuity of $f_i(\x^i,\bullet)$ and by noting that $\lim_{\eta \to 0} \x^{-i}_\eta = \x^{-i} \in \Xscr^{-i}$, we have that
      \begin{align*}
          f_i(\x^i,\x^{-i}) \leq  \lim_{\eta \to 0} \ f_{i,\eta}(\x^i,\x^{-i}_{\eta}) \leq  f_i(\x^i,\x^{-i}). 
    \end{align*}
Consequently, ({\bf D.II.}) holds.

    \end{proof}
\noindent {\bf Comment.}  We note that the convergence claim can be strengthened to a claim of subsequential convergence as long as $\Xscr$ is a compact set. This allows for claiming the existence of a convergent subsequence, whose limit point via the above result is the desired equilibrium of the original game.

\subsection{An asynchronous smoothed proximal best-response framework}\label{sec:4.2}
{Prior to presenting our asynchronous smoothed relaxed best-response
scheme, we provide some background.  Recall that in best-response schemes,   each player selects a
best-response (BR), given  current rival strategies
\cite{fudenberg98theory,basar99dynamic}. Such avenues have been applied on
engineering applications~\cite{scutari2010mimo}, where  the BR is expressible
in closed form.  Proximal BR schemes appear to have been  first
discussed by Facchinei and Pang~\cite{FPang09}, where they showed that the  set
of fixed points of  the proximal BR map is equivalent to the set of Nash
equilibria under convexity of the player-specific problems.  Asynchronous
BR schemes have been shown to be convergent by Altman et
al.~\cite{altman2007evolutionary}.  Recently,
in~\cite{pang2017two},   two synchronous schemes were proposed for computing an
equilibrium of a noncooperative game with risk-averse players under a
contractivity assumption on the proximal BR map. Under related
assumptions, we develop rate and asymptotic guarantees for randomized
synchronous and asynchronous variants~\cite{lei20synchronous}. In 2011,
Facchinei et al.~\cite{facchinei2011decomposition}  proposed  several
regularized Gauss-Seidel BR schemes for  generalized potential games, where it
was shown that limit points are Nash equilibria when each player's subproblem is
convex. Extensions to stochastic regimes were considered
in~\cite{lei20asynchronous} where almost-sure convergence guarantees were
provided for an efficient asynchronous best-response scheme where the
best-responses were solved with increasing accuracy under the assumption that
player-specific objectives were $L_i$-smooth uniformly in rival decisions.}

{{\bf Gaps in prior schemes.} Unfortunately, the scheme
in~\cite{lei20asynchronous} cannot be applied since it requires player-specific
smoothness properties and does not incorporate a relaxation. This motivates the
development of a scheme that can accommodate (i) nonsmoothness and (ii)
relaxation.}  

{Accordingly, we} develop an asynchronous relaxed {inexact} smoothed BR scheme
{(ARSPBR)}. At every step in (ARSPBR),  player $i$ is randomly selected based
on a prescribed probability ${p_i}$. Then player $i$ takes an inexact relaxed
best-response step based on $\gamma_k$ and $\epsilon^{i,k}$ while other players
do not update their strategy. If $\gamma_k = 1$ and $c$, the proximal weight,
is sufficiently large, then  this step reduces to an inexact best-response
step. Step (2) of the algorithm necessitates an inexact solution to the
hierarchical problem (Player$_{\eta}(\x^{-i,k})$). We propose a zeroth-order
scheme recently developed in a parallel paper and articulate both the scheme
and its error analysis in Section~\ref{sec:4.2}. 
\begin{tcolorbox}{{\bf Asynchronous relaxed smoothed  proximal best-response (ARSPBR) scheme}}
\be
\item[(0)] Let $k = 0$, $\z^{i,0} = \x^{i,0} \in \Xscr_i$ for $i = 1, \cdots, \Nbold$, and $p_i \in (0,1)$ for $i = 1, \cdots, \Nbold$ with $\sum_{i=1}^\Nbold p_i = 1$. {Given $\eta > 0$ and relaxation sequence $\{\gamma_k\}.$ }  
\item[(1)] Select a player $i_k = i \in \{1,\cdots, \Nbold\}$ with probability $p_i > 0$. 
\item[(2)] Update $\z^{k+1}$ and $\x^{k+1}$ as follows.
    \begin{align*} \tag{{ARSPBR}}
& \begin{aligned}
    \z^{i,k+1} & :=  \begin{cases} (1-\gamma_{k}) \x^{i,k} + \gamma_k \left(B_{i,\eta} (\x^k)\right); & i = i_k \\
    \x^{\ic{i},k}; & i \neq i_k \end{cases} \\ 
        \x^{i,k+1} & := \begin{cases} \z^{i,k+1} + \epsilon^{i,k+1}; & \qquad \qquad \qquad \quad i= i_k  \\ 
        {\z^{i,k+1}}; & \qquad \qquad \qquad \quad i \neq i_k. 
        \end{cases}
        \end{aligned}
\end{align*}
\item[(3)] Stop if $k > K$, Stop; else return to Step 1, $k:=k+1.$
\ee
\end{tcolorbox}

It can be observed that the update for $\x^{i,k+1}$ for $i = i_k$ can be rewritten as follows.
\begin{align} \label{rel-inex-BR}
   \notag  \x^{i,k+1} &:= \z^{i,k+1} + \epsilon^{i,k+1} \\
   \notag       & = (1-\gamma_{k}) \x^{i,k} + \gamma_k B_{i,\eta}(\x^k) + \epsilon^{i,k+1} \\
                & = (1-\gamma_{k}) \x^{i,k} + \gamma_k \left(B_{i,\eta}(\x^k) + \tfrac{\epsilon^{i,k+1}}{\gamma_k}\right).
\end{align}
In effect, $\x^{i,k+1}$ is a consequence of averaging between the previous belief $\x^{i,k}$ and an inexact best-response using the relaxation weight $\gamma_k$. When $\gamma_k = 1$, this reduces to the unrelaxed scheme and \eqref{rel-inex-BR} reduces to 
    \begin{align} \label{inex-BR}
            \x^{i,k+1} = B_{i,\eta}(\x^k)+\epsilon^{i,k+1}. 
    \end{align}
  This inexact best-response is an  
an $\tfrac{\epsilon^{i,k}}{\gamma_k}$-optimal solution of the best-response problem. Prior schemes on resolving protypical hierarchical problems of this form (i.e. MPECs) are not equipped with non-asymptotic rate guarantees. However, in Section~\ref{sec:4.2}, we develop a zeroth-order scheme with non-asymptotic rate guarantees when each player's objective is convex given rival decisions.  Before proceeding, we define the history of the process. Suppose $\Fscr_0 \triangleq \{\x^0\}$. Suppose $\Fscr'_k$ and $\Fscr_{k}$ are defined as
\begin{align*}
    \Fscr'_1 & = \Fscr_0 \cup \{i_1\}, \\
        \Fscr_1 & = \Fscr'_0 \cup \{\omega_{1,j_1}, \cdots, \omega_{1,j_1}\}, \\
                & \vdots \\ 
        \Fscr'_k & = \Fscr_{k-1} \cup \{i_{k-1}\}, \\ 
        \Fscr_k & = \Fscr'_{k-1} \cup \{\omega_{k-1,j_1}, \cdots, \omega_{k-1,j_{k-1}}\}. 
            \end{align*}
{Note that the samples $\{\omega_{k,j_1}, \cdots, \omega_{k,j_k}\}$ are
employed in computing an approximate best-response in iteration $k$ via a
zeroth-order Monte-Carlo sampling scheme.} We are now ready
to derive asymptotic guarantees for the asynchronous relaxed inexact
best-response scheme. 

\begin{proposition}[{\bf Almost-sure convergence for asynchronous relaxed inexact best-response scheme}] \em Consider a game $\Gcal \in \Gscr^{\rm chl}$. For any $i \in \{1, \cdots, \Nbold\}$, suppose $f_i(\bullet,\x^{-i})$ is a convex function for any $\x^{-i} \in \Xscr_{-i}$ and $\Xscr_i \subseteq \Real^{n_i}$ is a closed and convex set. Consider the smoothed counterpart of $\Gcal$, denoted by $\Gcal_{\eta}$ where $\Gcal_{\eta} \in \Gscr^{\rm chl}_{\rm pot}$; for any $i \in \{1, \cdots, \Nbold\}$, suppose $f_{i,\eta}(\bullet,\x^{-i})$ denotes the $\eta$-smoothing of $f_i(\bullet,\x^{-i})$. Suppose $P_{\eta}$ denotes the potential function of $\Gcal_{\eta}$ where $P_{\eta}(\x) \geq \tilde{P}$ for any $\x \in \Xscr + \eta \mathbb{B}$. Here $\tilde{P}$ denotes a lower bound on $P_{\eta}(\x)$. Suppose B$_i(\x)$ and B$_{i,\eta}(\x)$ denote the proximal best-response and smoothed proximal-response for $i \in \{1, \cdots, \Nbold\}.$ Then the following hold for any $i \in \{1, \cdots, \Nbold\}$.  Consider a sequence $\{\x^k\}$ generated by (ASRPBR) scheme. Then the following hold. 

    \smallskip
    \noindent (a) For $k \geq {0}$, the following holds almost surely. 
    \begin{align} \label{main_rec}
        \mathbb{E}[P_{\eta}(\x^{k+1})- \tilde{P}_{\eta} \mid \Fscr_k]  & \leq  
        (P_{\eta}(\x^{k}) - \tilde{P}_{\eta})   - \gamma_k\left(c - \tfrac{L\gamma_k}{2}\right) \|B_{i,\eta}(\x^k)-\x^{i,k}\|^2 \notag
                                                                    \\ & + \sum_{i=1}^{\Nbold} M_i\mathbb{E}[\|\epsilon^{i,k+1}\| \mid \Fscr_k].
    \end{align}

    \noindent (b) Suppose one of the following hold. (i) $\{\gamma_k\}$ is a decreasing non-summable but square-summable sequence where $\gamma_k < \tfrac{2c}{L}$ for every $k$; (ii) $\gamma_k = \gamma = 1$ and $c > \tfrac{L}{2}$. {Furthermore, suppose $\sum_{k=0}^{\infty} \sum_{i=1}^{\Nbold} M_i\mathbb{E}[\|\epsilon^{i,k+1}\| \mid \Fscr_k] < \infty$}. Then 
    \begin{align} \label{fp} \lim_{k \to \infty} \sum_{i=1}^{\Nbold} \|\x^{i,k} - B_{i,\eta}(\x^{k})\|^2 = 0 \mbox{ almost surely}. \end{align} 

    \smallskip

    \noindent (c) Suppose \eqref{fp} holds. Then $\{\x^k\}$ converges to the set of Nash equilibria of $\Gcal_{\eta}$ in an a.s. sense.

\end{proposition}
\begin{proof}
    (a) For ease of exposition, we let $i_k$, the player selected at the $k$th iteration, be denoted by $i$.   
    Since $f_{i,\eta}$ is $L$-smooth where $L = \tfrac{\alpha}{\eta}$, we have that 
    \begin{align}\notag
        f_{i,\eta}(\z^{i,k+1},\x^{-i,k}) &  \leq f_{i,\eta}(\x^{i,k},\x^{-i,k}) + \gamma_k \nabla_{\x^i} f_{i,\eta}(\x^{i,k},\x^{-i,k})^\mathsf{T}(B_{i,\eta}(\x^k)-\x^{i,k}) \\
        \label{eq1}                               & +  \tfrac{L\gamma_k^2}{2}\|B_{i,\eta}(\x^k)-\x^{i,k}\|^2.
    \end{align}
    Furthermore, by the optimality conditions of (SPBR$_{i,\eta}(\x^k)$), we have that 
    \begin{align}\label{eq2}
        \notag 0  & \leq (\nabla_{\x^i} f_{i,\eta}(B_{i,\eta}(\x^k),\x^{-i,k}) + c(B_{i,\eta}(\x^k)-\x^{i,k}))^\mathsf{T}(\x^{i,k}-B_{i,\eta}(\x^k)) \\
                 & {\le -\nabla_{\x^i} f_{i,\eta}(\x^{i,k},\x^{-i,k})}^\mathsf{T}(B_{i,\eta}(\x^k)-\x^{i,k}) - c\|B_{i,\eta}(\x^k)-\x^{i,k}\|^2,
    \end{align}
    a consequence of the monotonicity of $\nabla_{\x^i} f_{i,\eta}(\bullet,\x^{-i,k})$. By adding \eqref{eq1} and $\gamma_k \times $\eqref{eq2}, we obtain that 
    \begin{align}\label{bd-1}
        f_{i,\eta}(\z^{i,k+1} ,\x^{-i,k}) &  \leq f_{i,\eta}(\x^{i,k},\x^{-i,k}) - \gamma_k\left(c - \tfrac{L\gamma_k}{2}\right) \|B_{i,\eta}(\x^k)-\x^{i,k}\|^2. 
        \end{align} 
    Next, we derive a bound on $f_{i,\eta}(\x^{i,k+1},\x^{-i,k}) - f_{i,\eta}(\z^{i,k+1},\x^{-i,k})$ by the mean value theorem. 
    \begin{align} \notag 
        f_{i,\eta}(\x^{i,k+1},\x^{-i,k}) - f_{i,\eta}(\z^{i,k+1},\x^{-i,k}) & = \nabla_{\x^i} f_{i,\eta}(\tilde{x}_i,\x^{-i,k+1})^\mathsf{T} (\x^{i,k+1}-\z^{i,k+1}) \\
        \implies \vert f_{i,\eta}(\x^{i,k+1},\x^{-i,k}) - f_{i,\eta}(\z^{i,k+1},\x^{-i,k}) \vert & \leq  M_i \| \epsilon^{i,k+1}\|, 
        \label{bd-2}
    \end{align}
    where $\tilde{x}_i \in [\x^{i,k+1},\z^{i,k+1}]$.  Consequently, we have that 
       \begin{align*}
           P_{\eta}(\x^{k+1}) - P_{\eta}(\x^{k}) &  = 
           P_{\eta}(\x^{i,k+1},\x^{-i,k}) - P_{\eta}(\x^{i,k},\x^{-i,k})\\
                                                 & =   f_{i,\eta}(\x^{i,k+1},\x^{-i,k}) - f_{i,\eta}(\x^k) \\ 
                                                 & =   f_{i,\eta}(\x^{i,k+1},\x^{-i,k}) - f_{i,\eta}(\z^{i,k+1},\x^{-i,k})+ f_{i,\eta}(\z^{i,k+1},\x^{-i,k}) - f_{i,\eta}(\x^k) \\
        & \overset{\eqref{bd-1}}{\leq} - \gamma_k\left(c - \tfrac{L\gamma_k}{2}\right) \|B_{i,\eta}(\x^k)-\x^{i,k}\|^2 
         + f_{i,\eta}(\x^{i,k+1},\x^{-i,k}) - f_{i,\eta}(\z^{i,k+1},\x^{-i,k}) \\ 
        & \overset{\eqref{bd-2}}{\leq} - \gamma_k\left(c - \tfrac{L\gamma_k}{2}\right) \|B_{i,\eta}(\x^k)-\x^{i,k}\|^2 + M_i\|\epsilon^{i,k+1}\|. 
        \end{align*}
        Since $\epsilon^{j,k+1} = 0$ for $j \neq i$, we have that 
            \begin{align*}
                P_{\eta}(\x^{k+1}) - P_{\eta}(\x^{k}) & \leq - \gamma_k\left(c - \tfrac{L\gamma_k}{2}\right) \|B_{i,\eta}(\x^k)-\x^{i,k}\|^2 + \sum_{i=1}^{\Nbold} M_i\|\epsilon^{i,k+1}\|.
            \end{align*} 
            By taking expectations with respect to $\Fscr_k$, we have that in an a.s. sense that 
    \begin{align*}
        \notag \mathbb{E}[P_{\eta}(\x^{k+1})- \tilde{P}_{\eta} \mid \Fscr_k]  & \leq  
        (P_{\eta}(\x^{k}) - \tilde{P}_{\eta})   - \gamma_k\left(c - \tfrac{L\gamma_k}{2}\right) \mathbb{E}[\|B_{i,\eta}(\x^k)-\x^{i,k}\|^2 \mid \Fscr_k] \\
           \notag                                                            & + \sum_{i=1}^{\Nbold} M_i\mathbb{E}[\|\epsilon^{i,k+1}\| \mid \Fscr_k] \\
                \notag                                                        & =  
                                                                        (P_{\eta}(\x^{k}) - \tilde{P}_{\eta})   - \gamma_k\left(c - \tfrac{L\gamma_k}{2}\right) \mathbb{E}[\mathbb{E}[\|B_{i,\eta}(\x^k)-\x^{i,k}\|^2 \mid {\Fscr'_{k+1}} ] \mid \Fscr_k ] \\
                     \notag                                                  & + \sum_{i=1}^{\Nbold} M_i\mathbb{E}[\|\epsilon^{i,k+1}\| \mid \Fscr_k] \\
& = \notag 
(P_{\eta}(\x^{k}) - \tilde{P}_{\eta})   - \gamma_k\left(c - \tfrac{L\gamma_k}{2}\right) \sum_{i=1}^{\Nbold} p_i\|B_{i,\eta}(\x^k)-\x^{i,k}\|^2 \\
                                                                       & + \sum_{i=1}^{\Nbold} M_i \mathbb{E}[\|\epsilon^{i,k+1}\|\mid \Fscr_k], 
    \end{align*}
    where the second equality follows from the tower law of conditional expectation and the last equality arises from recalling that $\|B_{i,\eta}(\x^k)-\x^{i,k}\|^2$ is adapted to $\Fscr_k$ for every $i \in \{1, \cdots, \Nbold\}$.  

    \noindent (b(i)) By choice, $\gamma_k < \tfrac{2c}{L} = \tfrac{2c\eta}{\alpha}$ for every $k$. Since $\{\gamma_k\}$ is a diminishing sequence with $\sum_{k} \gamma_k^2 < \infty$, for sufficiently large $K$, $c - \tfrac{L\gamma_k}{2} \geq \tilde{c}$. Consequently, it suffices to consider a shifted recursion to claim that for $k > K$, the following holds a.s.   
\begin{align*} 
        \notag \mathbb{E}[P_{\eta}(\x^{k+1})- \tilde{P}_{\eta} \mid \Fscr_k]  & \leq  
        (P_{\eta}(\x^{k}) - \tilde{P}_{\eta})   - \gamma_k \tilde{c} {\sum_{i=1}^{\Nbold}p_i}\|B_{i,\eta}(\x^k)-\x^{i,k}\|^2\\
                                                                       & + \sum_{i=1}^{\Nbold} M_i\mathbb{E}[\|\epsilon^{i,k+1}\| \mid \Fscr_k].
    \end{align*}
    Since $\{(P_{\eta}(\x^{k}) - \tilde{P}_{\eta})\}$ is a nonnegative sequence and $\sum_{k=0}^{\infty} \sum_{i=1}^{\Nbold} M_i\mathbb{E}[\|\epsilon^{i,k+1}\| \mid \Fscr_k] < \infty$, we have that $\{(P_{\eta}(\x^{k}) - \tilde{P}_{\eta})\}$ is convergent a.s. and $\sum_{k=0}^{\infty} \gamma_k \sum_{i=1}^{\Nbold} \|B_{i,\eta}(\x^k)-\x^{i,k}\|^2 < \infty$ a.s. Since $\sum_{k=0}^{\infty} \gamma_k = \infty$, we have that    
    $$ \liminf_{k \to \infty} \sum_{i=1}^{\Nbold}  \|B_{i,\eta}(\x^k)-\x^{i,k}\|^2 = 0. $$ 
Consequently, along some subsequence $\Kscr$, we have that $\lim_{k \in \Kscr, k \to \infty}  \sum_{i=1}^{\Nbold} \|B_{i,\eta}(\x^k)-\x^{i,k}\|^2 = 0.$ It remains to show that $\sum_{i=1}^{\Nbold} \|B_{i,\eta}(\x^k)-\x^{i,k}\|^2 \xrightarrow[k \to \infty]{k \in \Kscr(\omega)} 0$ in an a.s. sense for almost every  $\omega \in \Omega$. We proceed by contradiction. Suppose for $\omega \in \Omega^c \subseteq \Omega$ and $\mathbb{P}(\omega \mid \omega \in \Omega^c) > 0$ , we have that $$\liminf_{k \in \Kscr(\omega)} \sum_{i=1}^{\Nbold} \|B_{i,\eta}(\x^k)-\x^{i,k}\|^2 \geq \bar{v}.$$   
Therefore, for every $\Kscr(\omega)$, there exists a $K(\omega)$ such that $\sum_{i=1}^{\Nbold} \|B_{i,\eta}(\x^k)-\x^{i,k}\|^2 \geq \tfrac{\bar{v}}{2}$ for $k \geq K(\omega)$. This implies that with finite probability, $\sum_{k \in \Kscr(\omega)} \sum_{i=1}^{\Nbold} \gamma_k \|B_{i,\eta}(\x^k)-\x^{i,k}\|^2 \geq \sum_{k \geq K(\omega), k \in \Kscr(\omega)} \gamma_k \|B_{i,\eta}(\x^k)-\x^{i,k}\|^2 \geq  
\sum_{k \geq K(\omega), k \in \Kscr(\omega)}  \tfrac{\gamma_k\bar{v}}{2} = \infty.$ But this contradicts the claim that  
$\sum_{k=1}^{\infty} \gamma_k \sum_{i=1}^{\Nbold} \|B_{i,\eta}(\x^k)-\x^{i,k}\|^2 < \infty$ almost surely. Therefore, 
$\sum_{i=1}^{\Nbold} \|B_{i,\eta}(\x^k)-\x^{i,k}\|^2 \xrightarrow[k \to \infty]{a.s.} 0$.

\medskip

    \noindent (b(ii)) Since $\gamma_k = 1$ for every $k$ and $c > L/2$, \eqref{main_rec} reduces to 
    \begin{align*}
        \notag \mathbb{E}[P_{\eta}(\x^{k+1})- \tilde{P}_{\eta} \mid \Fscr_k]  & \leq  
        (P_{\eta}(\x^{k}) - \tilde{P}_{\eta})   - \left(c - \tfrac{L}{2}\right) \sum_{i=1}^{\Nbold} \|B_{i,\eta}(\x^k)-\x^{i,k}\|^2\\
                                                                       & + \sum_{i=1}^{\Nbold} M_i\mathbb{E}[\|\epsilon^{i,k+1}\| \mid \Fscr_k].
    \end{align*}
    By invoking the Robbins-Siegmund lemma, we have that $\sum_{i=1}^{\Nbold} \|B_{i,\eta}(\x^k)-\x^{i,k}\|^2 < \infty$ a.s., implying that $\sum_{i=1}^{\Nbold}\|B_{i,\eta}(\x^k)-\x^{i,k}\|^2 \xrightarrow[k \to \infty]{a.s.} 0$ for $i = 1, \cdots, \Nbold$.  
   
    \noindent (c) This follows by ~\cite[Th.~1(b)]{lei20asynchronous}.

\end{proof}

\subsection{A zeroth-order scheme for resolving SPBR$_{i,\eta}(\x^{-i})$} \label{sec:4.2}
At the $k$th step of the (ARSPBR) scheme, the relaxed inexact scheme and its unrelaxed counterpart require computing an $\left(\tfrac{\epsilon^{i,k}}{\gamma_k}\right)$-solution to (SPBR$_{i,\eta}(\x^k)$). We develop a scheme for computing such a solution in this subsection; in particular, we consider the inexact resolution of the smoothed best-response problem given by (SPBR$_{i,\eta}(\x^{-i})$). 
\begin{align}\label{prox-prob}
    \min_{\vv^i \in \Xscr_i} \phi_{i,\eta}(\vv^i,\x) \triangleq  \left[f_{i,\eta}(\vv^i,\x^{-i}) + \tfrac{c}{2} \|\vv^i -\x^i\|^2\right].  
\end{align}
We denote an optimal solution to this problem by $\vv^{i,*}$ and our goal lies in developing a scheme that generates a sequence $\{\vv^{i,t}\}$ such that $\mathbb{E}[\| \vv^{i,t}-\vv^{i,*} \|^2 \mid \x] \leq Cq^t$, where $C$ and $q$ are positive scalars and $q \in (0,1)$. We observe  that
$\phi_{\eta}$ is an $\mathcal{O}(\tfrac{1}{\eta})$-smooth and $c$-strongly convex
expectation-valued function. Since $\phi_{i,\eta}$ is
$\mathcal{O}(\tfrac{1}{\eta})$-smooth, one might imagine that a standard
stochastic approximation scheme can be applied for computing an approximate
solution of \eqref{prox-prob}. However, this requires computing a sampled
gradient of $f_{i,\eta}(\bullet,\x^{-i})$. Unfortunately, one may recall that a sampled gradient of $f_{i,\eta}(\bullet,\x^{-i})$ 
 requires computing the
sampled gradient of $\tilde{h}_\eta(\x^i,\y^i(\x,\omega),\omega)$ where $\y^i(\x,\omega)$
represents the solution of a lower-level parametrized variational inequality
problem. Instead, we construct a {\em zeroth-order} scheme that relies only on
function values to approximate the gradient of
$\tilde{h}_{\eta}(\x^i,\y^i(\x,\omega))$. To this end, we develop a randomized
smoothing-based zeroth-order scheme inspired by~\cite{nesterov17}. In
particular, we define $\phi_{\eta}$ as 
\begin{align*} \notag
    & \, \phi_{i,\eta}(\vv^i,\x)   = \mathbb{E}_{u^i \in \mathbb{B}_i}\left[ f_{i,\eta}(\vv^i+\eta u^i,\x^{-i})+ \tfrac{c}{2} \|\vv^i+\eta u^i -\x^i\|^2\right]  \\
       \notag                     & =  \mathbb{E}_{u^i \in \mathbb{B}_i} \left[\mathbb{E}\left[ \tilde{g}_i(\vv^i+\eta u^i,\x^{-i},\omega)+\tilde{h}_{i,\eta}(\vv^i+\eta u^i,\y^i(\vv^i+\eta u^i,\x^{-i},\omega),\omega) \mid u^i \right]+ \tfrac{c}{2} \|\vv^i+\eta u^i -\x^i\|^2\right] \\
    & =  \mathbb{E}_{u^i, \omega} \left[ \tilde{g}_i(\vv^i+\eta u^i,\x^{-i},\omega)+\tilde{h}_{i,\eta}(\vv^i+\eta u^i,\y^i(\vv^i+\eta u^i,\x^{-i},\omega),\omega) + \tfrac{c}{2} \|\vv^i+\eta u^i -\x^i\|^2\right],  
\end{align*}
where $\omega$ and $u^i$ are independent random variables, $\mathbb{B}_i \triangleq \{u^i \in \Real^{n_i} \mid \|u^i\| \leq 1\}$, the inner expectation is with respect to $\omega$, conditional on $u^i$ while the outer expectation is with respect to $u^i$. {The gradient of $\phi_{\eta}(\vv^i,\x)$ is given by the following 
    \begin{align} \label{grad-phi}
    \nabla_{\vv^i} \phi_{i,\eta}(\vv^i,\x)=  \mathbb{E}_{v^i \in \eta \mathbb{S}_i} \left[\left(f_{i,\eta}(\vv^i+v^i,\x^{-i})+ \tfrac{c}{2}\|\vv^i+v^i-\x^i\|^2\right) \tfrac{n_iv^i}{\eta \|v^i\|}\right],
\end{align}
where $\mathbb{S}_i$ denote the surface of the ball $\mathbb{B}_i$, i.e., $\mathbb{S}_i\triangleq \{v^i \in \mathbb{R}^{n_i}\mid \|v^i\| = 1\}$. }
A mini-batch approximation of the zeroth-order approximation of the gradient by using $N_t$ samples $\{\omega_{j},v^{i,j}\}_{j=1}^{N_t}$ is denoted by 
{
\begin{align*} 
    \gb_{i,\eta,\phi,N_t} (\vv,\x) \triangleq \sum_{j=1}^{N_t} \tfrac{ \tfrac{n_i}{\eta }\left[ \tilde{g}_i(\vv^i+\eta v^{i,j},\x^{-i},\omega_j)+\tilde{h}_{i,\eta}(\vv^i+\eta v^{i,j},\y^i(\vv^i+\eta v^{i,j},\x^{-i},\omega_j),\omega_j) + \tfrac{c}{2} \|\vv^i+\eta v^{i,j} -\x^i\|^2\right]\left(\frac{v^{i,j}}{\|v^{i,j}\|}\right)}{N_t}. 
\end{align*}}
We observe that this mini-batch approximation satisfies suitable unbiasedness and moment assumptions in an almost-sure sense, a standard requirement in stochastic approximation approaches.

\begin{lemma}\em ~\cite[Lemma~3]{cui21zeroth}  For any $i \in \{1, \cdots, \Nbold\}$, suppose  
    $w_{i,t} =   {\bf g}_{i,\eta,\phi,N_t}(\vv^{i,t},\x)-\nabla_{\vv^i} \phi_{i,\eta}(\vv^{i,t},\x)$. Then the following hold for any $\vv^{i,t} \in \Xscr_i+\eta \mathbb{B}_i$ and any $\x \in \Xscr$.    

    \noindent (a) $\mathbb{E}_{v^i,\omega}  \left[ w_{i,t} \mid \vv^{i,t},\x\right] = 0$ almost surely. 

    \noindent (b) $\mathbb{E}_{v^i,\omega} \left[   \|w_{i,t}\|^2 \mid \vv^{i,t}, \x \right]\leq  \tfrac{\nu^2}{N_t}$ almost surely. 
\end{lemma}

Next, we recall that  $\phi_{i,\eta}(\bullet, \x)$ is Lipschitz continuous on a compact set $\Xscr_i+\eta \mathbb{B}_i$ (which follows from convexity over a compact set $\Xscr_i+\eta \mathbb{B}$ uniformly in $\x$ on $\Xscr$. Further from~\cite[Lemma~1]{cui21zeroth}, we recall that $\phi_{i,\eta}(\bullet, \x)$ is $(\tfrac{L_0 n_i}{\eta}+c)$-smooth on $\Xscr_i+\eta \mathbb{B}_i$ uniformly in $\x$. Both claims are formalized in the next Lemma.   

\begin{lemma} \em Consider the game $\Gcal \in \Gscr^{\rm chl}_{\rm pot}$ and its smoothed counterpart $\Gcal_{\eta}$. Then the following hold. Suppose $\Xscr_i$ is bounded for $i = 1, \cdots, \Nbold.$ 

    \noindent (a) For any $i \in \{1, \cdots, \Nbold\}$, the function $\phi_{i,\eta}(\bullet,\x)$ is convex and Lipschitz continuous on $\Xscr_i+\eta \mathbb{B}_i$ with constant $L_0$ uniformly in $\x$ on $\Xscr$.  

    \noindent (b) The gradient  of  $\phi_{i,\eta}(\bullet,\x)$, defined as \eqref{grad-phi}, is Lipschitz continuous on $\Xscr$ uniformly in $\x$ on $\Xscr$ with constant $\tfrac{L_0n_i}{\eta} + c$.

%
\end{lemma}

We now consider the application of the following scheme to \eqref{prox-prob}. Given a  $\vv^{i,0} \in \Xscr_i$, a sequence $\{\vv^{i,t}\}$ is constructed as follows. 
\begin{align}
    \tag{ZSOL}
    \vv^{i,t+1} := \Pi_{\Xscr_i} \left[ \vv^{i,t}- \zeta_t \left({\bf g}_{i,\eta,\phi,N_t}(\vv^{i,t},\x)+w_{i,t}\right) \right], \qquad t > 0 
\end{align}
where $w_{i,t} =  {\bf g}_{i,\eta,\phi,N_t}(\vv^{i,t},\x)-\nabla_{\vv^i} \phi_{i,\eta}(\vv^{i,t},\x)$. Suppose $\tilde{\Fscr}_0 = \{\vv^{i,0}\}$ and $\tilde{\Fscr}_t = \tilde{\Fscr}_{t-1} \cup\{\vv_t\}$. 

\begin{lemma} [{\bf Rate statement for zeroth-order scheme for $B_{i,\eta}(\x^k)$}]\em Suppose the scheme (ZSOL) is applied on \eqref{prox-prob} where $\phi_{i,\eta}(\bullet, \x)$ is $c$-strongly convex and $\alpha$-smooth, where $\alpha=\tfrac{L_0n_i}{\eta} + c$. Suppose $\zeta_t = \zeta < \tfrac{c}{\alpha^2}$, $q = (1-2c\zeta +2\zeta^2\alpha^2) < 1$,   and $N_t = \lceil q^{-(t+1)} \rceil$ for every $t$. Then the following holds for a suitable positive scalar $C$,  
    $\mathbb{E}[\|\vv^{i,t} - \vv^{i,*}\|^2 \mid \x] \leq q^t C, \mbox{ for } t > 0.$ 
\end{lemma}
\begin{proof}
Recall that 
\begin{align*}
    \|\vv^{i,t+1} - \vv^{i,*}\|^2 & \leq \|\vv^{i,t} - \vv^{i,*}\|^2
    - 2\zeta_t(\vv^{i,t}-\vv^{i,*})^\mathsf{T}( {\bf g}_{i,\eta,\phi,N_t}(\vv^{i,t},\x)  - {\bf g}_{i,\eta,\phi,N_t}(\vv^{i,*},\x)) \\
                                  & + 2\zeta_t^2 \|w_{i,t}\|^2 - 2\gamma_t w_{i,t}^\mathsf{T}( \vv^{i,t}-\vv^{i,*})
                                  +  2\zeta_t^2\alpha^2 \|\vv^{i,t}-\vv^{i,*}\|^2 \\
                                  & \leq (1-2c\zeta_t + 2\zeta_t^2\alpha^2) \|\vv^{i,t}-\vv^{i,*}\|^2- 2\gamma_t w_{i,t}^T(\vv^{i,t} - \vv^{i,*})
                                  +  2\zeta_t^2 \|w_{i,t}\|^2. 
\end{align*}
Taking expectations conditioned on $\x$, we obtain that 
\begin{align*}
    \mathbb{E}\left[   \|\vv^{i,t+1} - \vv^{i,*}\|^2 \mid \x \right] & \leq
    \left(1-2c\zeta_t + 2\zeta_t^2\alpha^2\right) \mathbb{E}[ \|\vv^{i,t}-\vv^{i,*}\|^2\mid \x]- 2\zeta_t \mathbb{E}[\underbrace{\mathbb{E}[w_{i,t}  \mid \tilde{\Fscr}_t, \x]}_{ \ = \ 0}\mid \x]^T( \vv^{i,t} - \vv^{i,*}) \\
                                                                     & +  2\zeta_t^2 \mathbb{E}[\underbrace{\mathbb{E}[ \|w_{i,t}\|^2 \mid \tilde{\Fscr}_t, \x]}_{ \ \leq \ \tfrac{\nu^2}{N_t}}\mid \x]
                                                                     \leq  \left(1-2c\zeta_t + 2\zeta_t^2\alpha^2\right) \mathbb{E}[ \|\vv^{i,t}-\vv^{i,*}\|^2 \mid \x]+  \tfrac{2\zeta_t^2 \nu^2}{N_t}.
\end{align*}
By setting $\zeta_t = \zeta$ such that $ (1-2c\zeta + 2\zeta^2\alpha^2) = q < 1$, we have that 
\begin{align*}
    \mathbb{E}\left[   \|\vv^{i,t+1} - \vv^{i,*}\|^2 \mid \x\right] & \leq  q \mathbb{E}[ \|\vv^{i,t}-\vv^{i,*}\|^2 \mid \x]+  2\zeta^2 \nu^2q^{t+1} 
                                                       \leq q^{t+1} C, 
\end{align*}
where $C$ is a suitably defined positive scalar. 
\end{proof}

{We observe that in \eqref{rel-inex-BR},  B$_{i,\eta}(\x^k) = \vv^{i,*}$ and
$\epsilon^{i,k+1} \triangleq \vv^{i,t} - \vv^{i,*}$. Therefore by employing
Jensen's inequality, we may show that  $\mathbb{E}[\|\epsilon^{i,k+1}\| \mid
\x^k] \leq \sqrt{C} q^{t/2}$. We conclude with a comment on the relationship between the two proposed schemes.} 

\medskip

\noindent {\bf Comment on the relationship between ({\bf ARSPBR}) and ({\bf VR-SPP})}.\\ 

\noindent {(i) \em Monotonicity vs Potentiality.} Section 3 focuses on the resolution of a monotone hierarchical game where the ``monotonicity'' of the game corresponds the monotonicity of the concatenated  player-specific subdifferential maps. However, Section 4 considers a class of potential hierarchical game where the potentiality is again with respect to the implicit player-specific objectives.\\ 

\noindent {(ii) \em Gradient-response  vs Best-response.} Section 3 develops a partially distributed stochastic proximal-point scheme for resolving the associated stochastic inclusion problem, where players take gradient-response steps (with a modified proximal term).  Section 4 presents an inexact best-response scheme that can be implemented in a partially distributed regime.

\section{Numerical Results}\label{sec:numerics}
In Section~\ref{sec:5.1}, we apply the ({\bf VR-SPP}) scheme to resolving the class of multi-leader multi-follower games considered in Section~\ref{sec:2.3}(b). In addition, we also examine how such schemes cope with expectation-valued constraints. In Section~\ref{sec:5.2}, we apply ({\bf ARSPBR}) to a class of hierarchical games  in uncertain settings as described in Section~\ref{sec:5.2}. 
\subsection{A multi-leader multi-follower problem under uncertainty}\label{sec:5.1}
In this section, we apply ({\bf VR-SPP}) on a 
multi-leader multi-follower game  described in Section \ref{sec:2.3} (Example b). 

\noindent {\bf 5.1.1. Problem parameters and algorithm specifications.} 
Suppose $\Nbold=13$ leaders and
$\Mbold=10$ followers and let $C_i$ is generated from the distribution $\mathcal{U}(0,100)$ for $i = 1,\cdots,\Nbold$, where $\mathcal{U}(l,u)$ denotes the uniform distribution on the interval $[l,u]$. Furthermore, $c_j=50$, for $j=1,\cdots,\Mbold$, $b=7$ and $a(\omega) \sim \mathcal{U}(33,37)$.  We compare our proposed scheme with a more standard stochastic approximation scheme applicable on monotone inclusion $0 \in T(\x)$ and specify their algorithm parameters. Solution quality is compared by estimating the residual function  $\texttt{res}(\x)=\norm{T_\lambda(\x)}$.  

\noindent {(i) ({SG}): Stochastic subgradient framework.}  Here, we employ the following stochastic subgradient scheme  to generate $\{\{\x^{k,i}\}_{i=1}^{\Nbold}\}$.  
        \begin{align}\tag{SG} \left\{ \begin{aligned}
                \x^{k+1,1}& \coloneqq \Pi_{\mathcal{X}_1} \left[\x^{k,1} - \alpha_k u_{k,1} \right] \\
                    & \qquad \vdots \\
                \x^{k+1,\Nbold}& \coloneqq \Pi_{\mathcal{X}_{\Nbold}} \left[\x^{k, \Nbold} - \alpha_k u_{k,\Nbold} \right]
        \end{aligned}\right\}, \mbox{ where } u_{k,i} \in \partial_{\x^i} \tilde{f}_i(\x,\y(\x,\omega),\omega) 
                \end{align}
            for $i = 1, \cdots, \Nbold.$
    In (SG), $\alpha_k \triangleq \tfrac{\alpha_0}{\sqrt{k}}$, where $\alpha_0=0.1$. $\x^0$ is randomly generated in $[0,1]^M$.

    \noindent {(ii) {\bf (VR-SPP)}. } We apply the ({\bf VR-SPP}) scheme defined in Section~\ref{sec:3.2.3} in which we employ $N_k=\lfloor 1.1^{k+1}\rfloor$, a proximal parameter $\lambda=0.1$ and a diminishing steplength $\tfrac{\alpha_0}{k}$ with $\alpha_0=0.1$ to approximate the resolvent via the (SA) scheme  (also presented in Section~\ref{sec:3.2.3}).

\smallskip 

\noindent {\bf 5.1.2. Performance comparison and insights.} In Fig.~\ref{tra1},
we compare the numerical performance between ({SG}) and ({\bf VR-SPP}) with
various parameters under the same number of samples. The thick line indicates
the average performance and the transparent area is the variability over 20
simulations. We examine their sensitivities to the number of players,
variability and steplength, respectively, in Table~\ref{time1}. First, both
Fig.\ref{tra1} and Table~\ref{time1} show that on this class of problems, ({\bf
VR-SPP}) significantly outperforms ({SG}) schemes. Second, ({\bf VR-SPP}) takes
far less time than ({SG}) while providing far more accurate solutions. The
distinctions in time emerge since ({\bf VR-SPP}) utilizes an increasing
sample-size policy and thus it takes far fewer resolvent steps than ({SG}). In
each iteration, we use a ({SG}) scheme to evaluate $\norm{T_\lambda(x_k)}$;
therefore ({\bf VR-SPP}) uses far fewer outer iterations, leading to far
shorter runtimes.  
\begin{figure}[htbp]
\centering
\includegraphics[width=.48\textwidth]{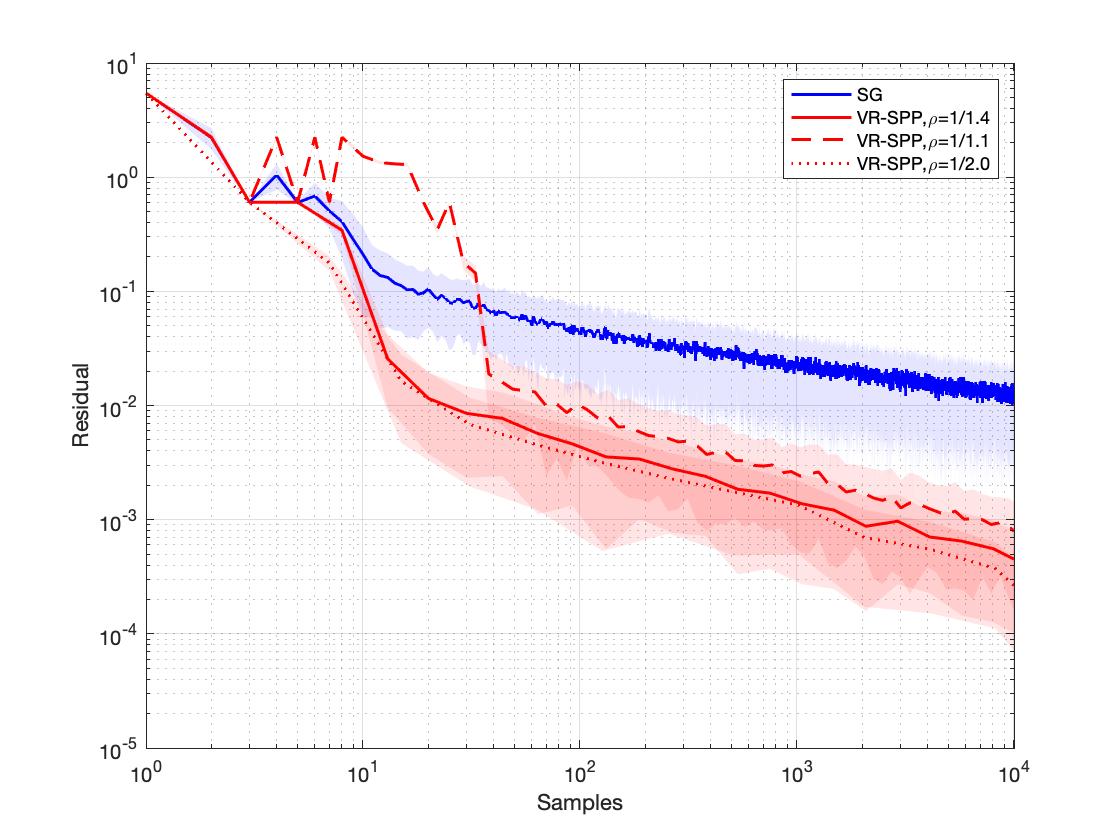}
\includegraphics[width=.48\textwidth]{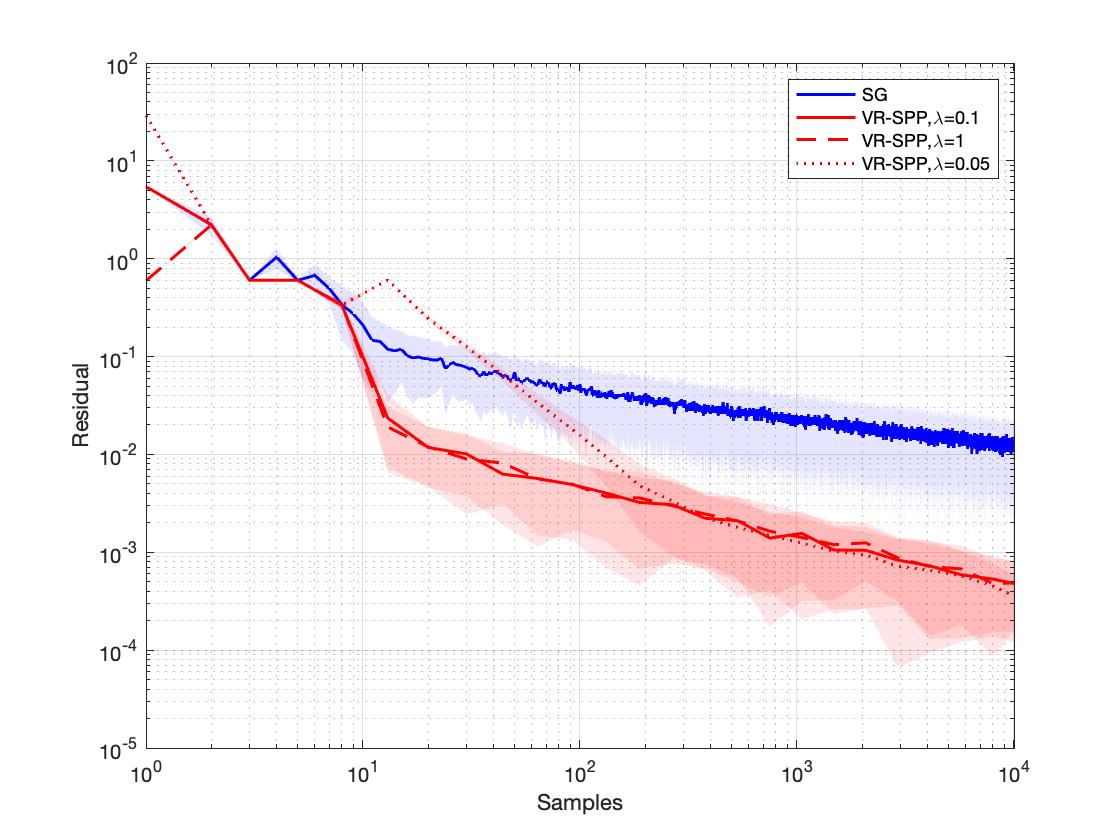} \hfill
\caption{Trajectories for ({SG}) and ({\bf VR-SPP})}
\label{tra1}
\end{figure}

\begin{table}[htb]
\scriptsize
\caption{Errors and time comparison of ({SG}) and ({\bf VR-SPP}) with various parameters}
\begin{center}
\begin{threeparttable} 
\begin{tabular}{c}
    \begin{tabular}[t]{ c | c | c | c |  c }
    \hline
    \multirow{2}{*}{$\Nbold$} & \multicolumn{2}{c|}{{SG}} & \multicolumn{2}{c}{\bf VR-SPP}   \\ \cline{2-5}
      & $\texttt{res}(\x^k)$ & Time & $\texttt{res}(\x^k)$ & Time  \\ \hline
      13 & 1.3e-2 & 6.7 & 5.0e-4 & 0.26 \\
      23 & 1.6e-2 & 13.8 & 5.2e-4 & 0.45  \\
     33 & 1.7e-2 & 28.8 & 5.8e-4 & 0.53 \\
     43 & 1.8e-2 & 41.1 & 5.7e-4 & 0.61 \\ \hline
    \end{tabular}  \hspace{0.1in}
    
            \begin{tabular}[t]{ c | c | c | c |  c }
    \hline
    \multirow{2}{*}{$\alpha_0$} & \multicolumn{2}{c|}{{SG}} & \multicolumn{2}{c}{\bf VR-SPP}   \\ \cline{2-5}
      & $\texttt{res}(\x^k)$ & Time & $\texttt{res}(\x^k)$ & Time  \\ \hline
      0.1 & 1.3e-2 & 6.7 & 5.0e-4 & 0.26  \\
      0.2 & 2.5e-2 & 6.6 & 5.7e-4  & 0.26  \\
     0.5 & 2.9e-2 & 6.6 & 1.3e-3 & 0.26 \\
     1 & 3.5e-2 & 6.7 & 1.7e-3 & 0.26  \\ \hline
    \end{tabular}  \\
    
    \begin{tabular}[t]{ c | c | c | c |  c }
    \hline
    \multirow{2}{*}{$a$} & \multicolumn{2}{c|}{{SG}} & \multicolumn{2}{c}{\bf VR-SPP}   \\ \cline{2-5}
      & $\texttt{res}(\x^k)$ & Time & $\texttt{res}(\x^k)$ & Time  \\ \hline
      $[33,37]$ & 1.3e-2 & 6.7 & 5.0e-4 & 0.26  \\
      $[30,40]$ & 3.5e-2 & 6.7 & 7.6e-4 & 0.27  \\
     $[25,45]$ & 5.0e-2 & 6.7 & 1.7e-3 & 0.26 \\
     $[20,50]$ & 6.6e-2 & 6.7 & 2.5e-3 & 0.26 \\ \hline
    \end{tabular} 
   
    \end{tabular}
    
\begin{tablenotes}
\small
\item The errors and time in the table are the average results of 20 runs
    \end{tablenotes}
  \end{threeparttable}
\end{center}
\label{time1}
\end{table}

\noindent {\bf 5.1.3. Incorporating expectation-valued constraints.} We now consider an extension of this game where each player is faced by expectation-valued constraints. Specifically, we impose a constraint $\mathbb{E}[c_i(\x^i,\omega_i)] \leq 0$ where $c_i(\x^i,\omega_i) = \x_i - U_i+\omega_i$ for $i=1, \cdots, \Nbold$.   We choose $U_i=5$ and $\omega_i \sim \mathcal{U}(-1,1)$ for $i =1 , \cdots, \Nbold$. With these additional expectation-valued constraints, we again compare the ({SG}) and ({\bf VR-SPP}) schemes in Table~\ref{timed}. It can be seen that akin to earlier, the ({\bf VR-SPP}) scheme is not overly impaired by the presence of expectation-valued constraints.

\begin{table}[htb]
\scriptsize
\caption{Comparison of ({SG}) and ({\bf VR-SPP}) for games with expectation-valued constraints}
\begin{center}
\begin{threeparttable} 
\begin{tabular}{c}
    \begin{tabular}[t]{ c | c | c | c |  c }
    \hline
    \multirow{2}{*}{$\Nbold$} & \multicolumn{2}{c|}{{SG}} & \multicolumn{2}{c}{\bf VR-SPP}   \\ \cline{2-5}
      & $\texttt{res}(\x^k)$ & Time & $\texttt{res}(\x^k)$ & Time  \\ \hline
      13 & 1.6e-2 & 6.8 & 5.9e-4 & 0.28 \\
      23 & 1.9e-2 & 14.3 & 6.3e-4 & 0.46  \\
     33 & 2.0e-2 & 29.4 & 6.8e-4 & 0.54 \\
     43 & 2.2e-2 & 42.3 & 7.4e-4 & 0.64 \\ \hline
    \end{tabular}  \hspace{0.1in}
    
            \begin{tabular}[t]{ c | c | c | c |  c }
    \hline
    \multirow{2}{*}{$\alpha_0$} & \multicolumn{2}{c|}{{SG}} & \multicolumn{2}{c}{\bf VR-SPP}   \\ \cline{2-5}
      & $\texttt{res}(\x^k)$ & Time & $\texttt{res}(\x^k)$ & Time  \\ \hline
      0.1 & 1.6e-2 & 6.8 & 5.9e-4 & 0.28  \\
      0.2 & 2.1e-2 & 6.9 & 7.7e-4  & 0.29  \\
     0.5 & 4.6e-2 & 6.9 & 1.3e-3 & 0.28 \\
     1 & 7.1e-2 & 7.0 & 1.8e-3 & 0.30  \\ \hline
    \end{tabular}  \\
    
    \begin{tabular}[t]{ c | c | c | c |  c }
    \hline
    \multirow{2}{*}{$a$} & \multicolumn{2}{c|}{{SG}} & \multicolumn{2}{c}{\bf VR-SPP}   \\ \cline{2-5}
      & $\texttt{res}(\x^k)$ & Time & $\texttt{res}(\x^k)$ & Time  \\ \hline
      $[33,37]$ & 1.6e-2 & 6.8 & 5.9e-4 & 0.28   \\
      $[30,40]$ & 4.7e-2 & 6.8 & 1.3e-4 & 0.28  \\
     $[25,45]$ & 5.3e-2 & 6.9 & 2.6e-3 & 0.28 \\
     $[20,50]$ & 6.9e-2 & 6.8 & 4.4e-3 & 0.28 \\ \hline
    \end{tabular} 
   
    \end{tabular}
    
\begin{tablenotes}
\small
\item The errors and time in the table are the average results of 20 runs
    \end{tablenotes}
  \end{threeparttable}
\end{center}
\label{timed}
\end{table}

\subsection{A monotone stochastic bilevel game}\label{sec:5.2}
We apply ({\bf ARSPBR}) on the game in Section~\ref{sec:2.3}(a). We evaluate the solution quality of
player $i$ by the residual function
$\texttt{res}(\x^i)=\norm{\x^i-B_{i,\eta}(\x^{-i})}$ and use
$\texttt{res}(\x)=\sum_{i=1}^\Nbold\texttt{res}(\x^i)/\Nbold$ to denote the residual
across multiple players in a game.

\noindent {\bf 5.2.1.  Problem parameters and algorithm specifications.} We assume that there are $\Nbold=13$ players and
each with a single follower. Furthermore, let $Q_i=3$ and we reuse the symbol
of $b_i$ and $l_i$, letting $b_i(\x^i)=b_i\x^i$ and $l_i(\x^i)=l_i\x^i$, for $i =1, \cdots, \Nbold$, $b_i$ and $l_i$ are generated from $\mathcal{U}(0,3)$ and
$\mathcal{U}(0,1)$, respectively. Suppose
$\tilde{g}_i(\x^i,\x^{-i})=\tfrac{1}{2}d_i \cdot (\x^i)^2+3\x^i\cdot{\sum_{j=1}^\Nbold}\x^j$,
where $d_i$ is randomly generated from $\mathcal{U}(0,100)$ for $i = 1, \cdots, \Nbold$. The
random parameter is $a_i(\omega) \sim \mathcal{U}(33,37)$, for $i = 1, \cdots, \Nbold$.  At iteration $k$, we run $T_k = \lceil
\log(k^{1.5}) \rceil$ steps in (ZSOL) and we use $N_t=\lceil 1.5^{t+1} \rceil$
samples for $t>0$. In addition, we assume steplength $\zeta_t=0.01, \
\forall t>0$ and smoothing parameter $\eta=0.1$. 

\smallskip 

\begin{figure}[htbp]
\centering
\includegraphics[width=.48\textwidth]{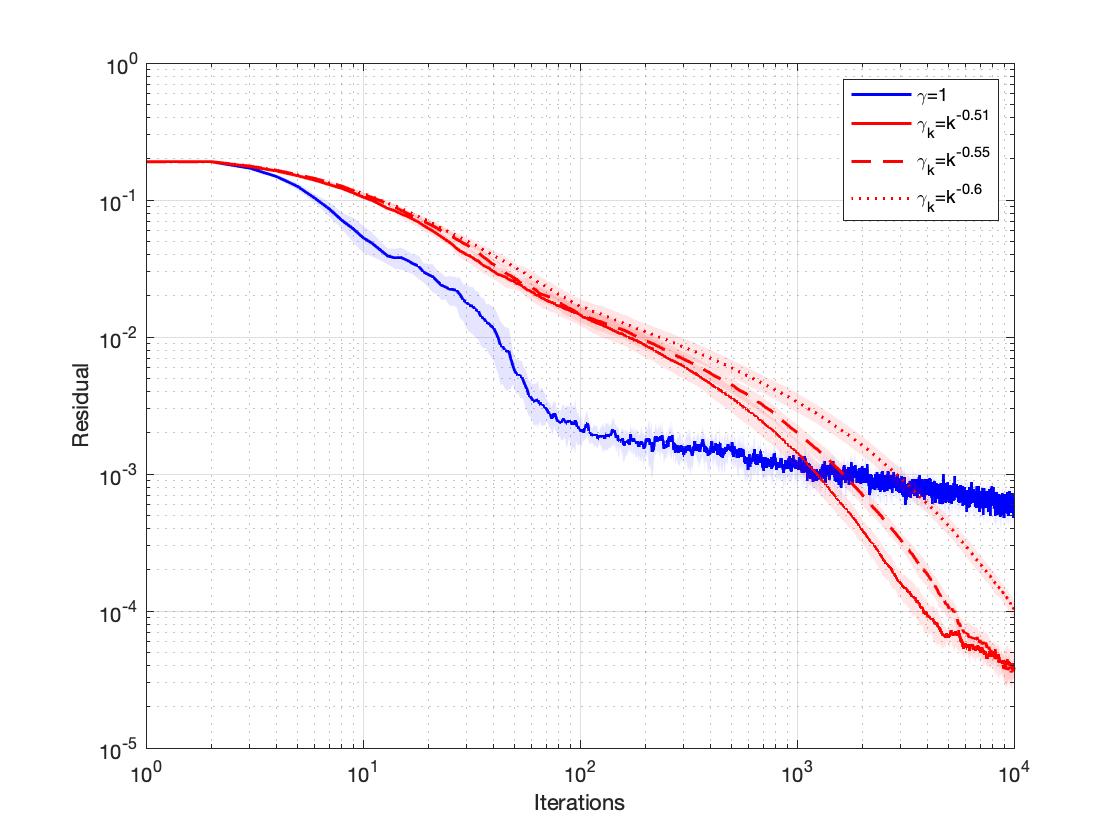}
\includegraphics[width=.48\textwidth]{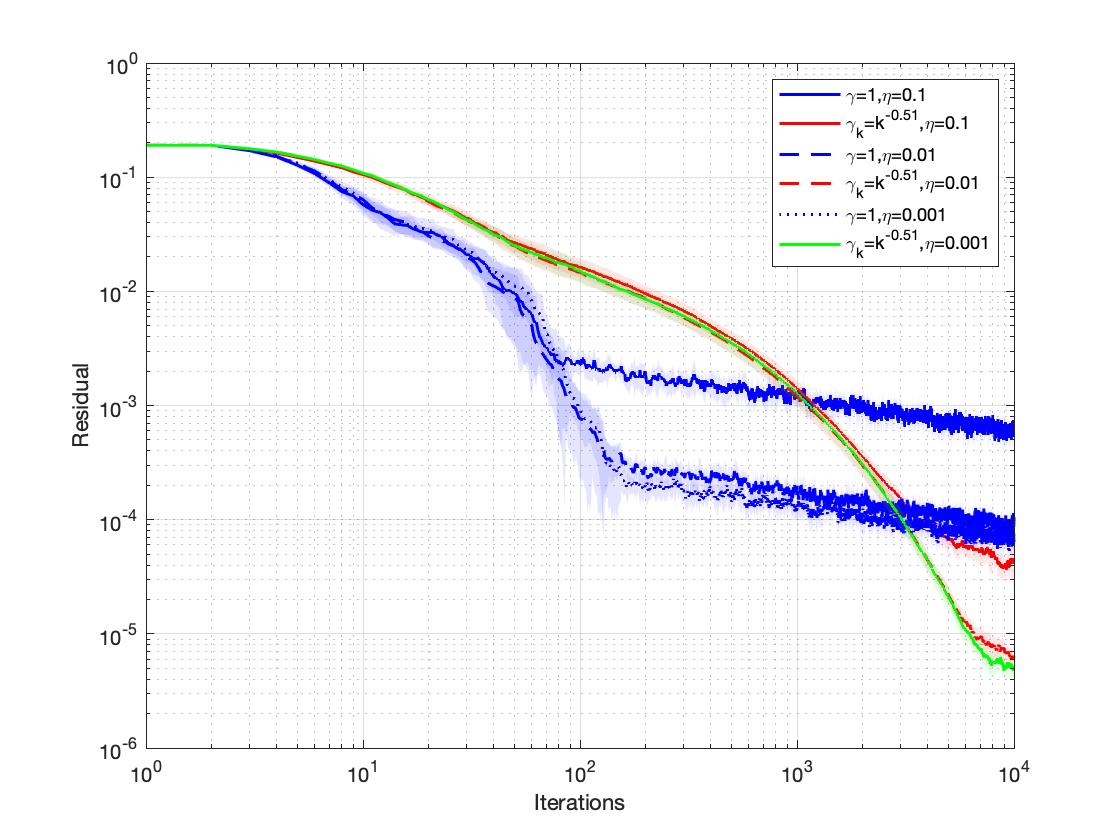} \hfill
\caption{Trajectories for ({\bf ARSPBR}) with different relaxation and smoothing parameters}
\vspace{-0.15in}
\label{tra2}
\end{figure}

\noindent {\bf 5.2.2. Performance comparison and insights.} In Fig.~\ref{tra2},
we compare the plots for ({\bf ARSPBR}) with different relaxation sequences
(left) and varying smoothing parameters (right). All the trajectories clearly
show the schemes converge to the optimal solution. While none of the relaxation schemes
perform better in early stages, they tend to have superior performance and a higher degree of stability
as the process continues. In addition, we examine their sensitivities to
various parameters in Table~\ref{time2}. Again, we note that the relaxation schemes provide more accurate solutions with a similar level of effort.

\begin{table}[htbp]
\scriptsize
\caption{Errors and time of ({\bf ARSPBR}) with ($\gamma_k=1$) and ($\gamma_k=k^{-0.51}$)}
\begin{center}
\begin{threeparttable} 
\begin{tabular}{c}
    \begin{tabular}[t]{ c | c | c | c |  c }
    \hline
    \multirow{2}{*}{$\Nbold$} & \multicolumn{2}{c|}{$\gamma_k=1$} & \multicolumn{2}{c}{$\gamma_k=k^{-0.51}$}  \\ \cline{2-5}
      & $\texttt{res}(\x^k)$ & Time & $\texttt{res}(\x^k)$ & Time  \\ \hline
      13 & 6.1e-4 & 7.6 & 4.0e-5  & 7.6  \\
      23 & 6.2e-4 & 9.3 & 1.3e-4 & 9.2  \\
     33 & 6.4e-4 & 11.4 & 4.8e-4 & 11.4 \\
     43 & 6.0e-4 & 15.4 & 6.3e-4 & 15.3 \\ \hline
    \end{tabular}  \hspace{0.1in}
    
            \begin{tabular}[t]{ c | c | c | c |  c }
    \hline
    \multirow{2}{*}{$\zeta_t$} & \multicolumn{2}{c|}{$\gamma_k=1$} & \multicolumn{2}{c}{$\gamma_k=k^{-0.51}$}   \\ \cline{2-5}
      & $\texttt{res}(\x^k)$ & Time & $\texttt{res}(\x^k)$ & time  \\ \hline
      1e-2 & 6.1e-4 & 7.6 & 4.0e-5  & 7.6  \\
      5e-3 & 6.8e-4 & 7.4 & 4.8e-5 & 7.3  \\
      2e-3 & 7.2e-4 & 7.5 & 5.2e-5 & 7.4  \\
      1e-3 & 7.6e-4 & 7.7 & 5.6e-5 & 7.6  \\
       \hline
    \end{tabular}  \\
    
    \begin{tabular}[t]{ c | c | c | c |  c }
    \hline
    \multirow{2}{*}{$a$} & \multicolumn{2}{c|}{$\gamma_k=1$} & \multicolumn{2}{c}{$\gamma_k=k^{-0.51}$}   \\ \cline{2-5}
      & $\texttt{res}(\x^k)$ & Time & $\texttt{res}(\x^k)$ & Time  \\ \hline
      $[33,37]$ & 6.1e-4 & 7.6 & 4.0e-5  & 7.6  \\
      $[30,40]$ & 6.8e-4 & 7.7 & 4.3e-5  & 7.4  \\
     $[25,45]$ & 7.2e-4 & 7.5 & 5.0e-5  & 7.6  \\
     $[20,50]$ & 8.0e-4 & 7.7 & 5.1e-5 & 7.8 \\ \hline
    \end{tabular} 
   
    \end{tabular}
    
\begin{tablenotes}
\small
\item The errors and time in the table are the average results of 20 runs
    \end{tablenotes}
  \end{threeparttable}
\end{center}
\label{time2}
\end{table}

\noindent {\bf 5.2.3. Convergence of smoothed equilibria to the true equilibrium.} To show that the sequence of equilibria $\{\x^*_{\eta}\}_{\eta \downarrow 0}$ converges to $\x^*$, an equilibrium of the original game, we provide values of two metrics in Table~\ref{time3}. that are the best-response residual for the smoothed game $\|\x_{\eta}^{*}-B(\x_\eta^*)\|$, where $B(\x_\eta^*) \triangleq \left\{B_1(\x_\eta^*), \cdots, B_\Nbold(\x_\eta^*) \right\}$ and the residual $\|\x_{\eta}^{*}-\x^*\|$. To compute the equilibrium of the original game, we make a slight modification to some algorithm parameters. Here we assume $b_i=3$ and $l_i=1$ for all $i$. It is not difficult to see that in player $i$'s optimization problem, the optimal solution $\x^{i,*}$ should be negative. It follows that $\y^i(\x^{i,*},\omega) \triangleq \max\left\{ Q_i(\omega)^{-1} b_i(\x^{i,*},\omega), \ell_i(\x^{i,*},\omega) \right\}=Q_i(\omega)^{-1} l_i(\x^{i,*},\omega)$. Therefore, in (SGE-a), $\partial_{\x^i} \mathbb{E}[\tilde{{g}}_i(\x^i,\x^{-i},\omega) + {\tilde{h}}_i(\x^i,\omega)]$ is linear and single-valued which means we can use $\texttt{PATH}$ to compute the true equilibrium. 

{\bf Insights.} It can be observed that the distance to the true equilibria diminishes to zero as one gets increasingly accurate equilibria of games with progressively smaller $\eta$. This aligns with the theoretical claim in Section~\ref{sec:4.1}.   

\begin{table}[htbp]
\scriptsize
\caption{Residuals of equilibria of the smoothed games under various smoothing parameters}
\begin{center}
\begin{threeparttable} 
\begin{tabular}{c}
    \begin{tabular}[t]{ c | c | c | c |  c  | c | c}
    \hline
    & $\eta$ & $0.2$ & $0.1$ & $0.01$ & $0.001$ & $0.0001$ \\ \hline
     \multirow{ 2}{*}{$\Nbold = 13$} &$\|\x_{\eta}^{*}-B(\x_\eta^*)\|$ & 1.1e-3 & 3.3e-4 & 5.6e-5 & 6.0e-5 & 2.3e-5 \\ \cline{2-7}
     & $\|\x_{\eta}^{*}-\x^*\|$ & 1.2e-3 & 3.4e-4 & 5.2e-5 & 3.6e-5 & 3.1e-5 \\\hline
     \multirow{ 2}{*}{$\Nbold = 23$} &$\|\x_{\eta}^{*}-B(\x_\eta^*)\|$ & 5.5e-1 & 2.9e-4 & 5.0e-5 & 3.9e-5 & 3.6e-5 \\ \cline{2-7}
     & $\|\x_{\eta}^{*}-\x^*\|$ & 1.1e-1 & 3.8e-4 & 5.4e-5 & 4.8e-5 & 4.3e-5 \\\hline
     \multirow{ 2}{*}{$\Nbold = 33$} &$\|\x_{\eta}^{*}-B(\x_\eta^*)\|$ & 7.8e-1 & 7.6e-3 & 8.2e-5 & 7.6e-5 & 5.7e-5 \\ \cline{2-7}
     & $\|\x_{\eta}^{*}-\x^*\|$ & 1.8e-1 & 1.2e-3 & 1.0e-4 & 9.2e-5 & 7.4e-5 \\\hline
   \end{tabular}
    \end{tabular}
    
\begin{tablenotes}
\small
    \end{tablenotes}
  \end{threeparttable}
\end{center}
\label{time3}
\end{table}

\section{Concluding remarks} We consider a class of hierarchical convex games
under uncertainty, a class of games in which  the implicit form of the
player-specific problems is convex, given rival decisions. In fact, certain subclasses of
multi-leader multi-follower games are known to lie in the considered class of
games. We present two sets of schemes for computing equilibria of such games.
Of these, the first  is a variance-reduced proximal-point framework and can
contend with monotone regimes, admitting  optimal deterministic rates of
convergence and near-optimal sample complexities. The second can process
smoothed potential variants of such games via an asynchronous relaxed smoothed
proximal best-response scheme. Notably, sequences produced by such schemes
converge almost surely to an $\eta$-approximate Nash equilibrium of the
original game where $\eta$ denotes a fixed smoothing parameter. We develop a
geometrically convergent zeroth-order scheme for computing the best response
which reduces to resolving a mathematical program with equilibrium constraints,
a problem that is known to be strongly convex in its implicit form. While
preliminary numerics are promising, we believe that this is but a first step in
developing a rigorous foundation for a broad class of hierarchical games
complicated by risk, nonsmootheness, and nonconvexity. 

\appendix
\section{Appendix}
\noindent{\bf A.1. Variational inequality problems, Inclusions, and monotonicity.}\\ 

\noindent {\bf (a) Variational inequality problems and inclusions.} 
Consider a variational inequality problem VI$(\mathcal{X}, F)$ where $\mathcal{X}$ is a closed and convex set and $F: \Real^n \to \Real^n$ is a single-valued continuous map. Such a problem requires an $\x$ such that 
    $$ (\tilde{\x}-\x)^\mathsf{T} F(\x) \geq 0, \qquad \forall \tilde{\x} \in\Xscr.$$
Furthermore, VI$(\Xscr,F)$ can also be written as an inclusion problem, i.e.
\begin{align*}
    \x \mbox{ solves } \mbox{VI}(\Xscr,F) \qquad \iff \qquad 0 \in F(\x) + \mathcal{N}_{\Xscr}(\x). 
\end{align*}
     Consider an $\Nbold$-player game $\Gcal$  where for $i=1, \cdots, \Nbold$, the $i$th player minimizes the parametrized smooth convex optimization problem defined as 
   \begin{align} \tag{Agent$_i(\x^{-i})$}
	\min_{\x^i \in \Xscr_i} \ f_i(\x^i, \x^{-i}). 
\end{align}
By convexity assumptions,  the set of Nash equilibria of $\Gcal$ is equivalent to the solution set of the variational inequality problem VI$(\mathcal{X}, F)$ where  $ \mathcal{X} \triangleq \prod_{i=1}^\Nbold \Xscr_i$ and  
\begin{align}
    \label{def-F} F(\x) \triangleq \pmat{ \nabla_{\x^1} f_1(\x) \\
				\vdots \\
        \nabla_{\x^\Nbold} f_\Nbold(\x)}.\end{align}
               If $f$ is a nonsmooth convex function, then the subdifferential $\partial f$ is also a monotone set-valued (or multi-valued) map on $\Xscr$. In addition, if $f_i(\bullet, \x^{-i})$ is not necessarily smooth, then the associated set of equilibria are given by the solution of VI$(\Xscr, T)$ where 
               \begin{align}
               \label{def-t} T(\x) \triangleq \prod_{i=1}^\Nbold \partial_{\x^i} f_i(\x^i,\x^{-i}).\end{align}

               \medskip

               \noindent {\bf (b) Monotonicity properties.}  Consider VI$(\Xscr, F)$.  Then the map  $F$ is monotone on $\mathcal{X}$ if $(F(\x)-F(\y))^\mathsf{T}(\x-\y) \geq 0$ for all $\x, \y \in \Xscr$.  Monotonicity may also emerge in  the context of $\Nbold$-player noncooperative games. In particular, one may view $\Gcal$ as being monotone if and only if the associated map $F$, defined as \eqref{def-F},  is  monotone on $\Xscr$.  In the special case when $\Nbold = 1$, this reduces to the gradient map of a smooth convex function $f$, denoted by $\nabla f$, being monotone. This can also be generalized to set-valued regimes. For instance, the map $T$, defined as \eqref{def-t}, arising from a noncooperative game $\Gcal$ with nonsmooth player-specific objectives is said to be monotone if for any $\x, \y \in \Xscr$ and any $u \in T(\x)$ and $v \in T(\y)$, we have $(u-v)^\mathsf{T}(\x-\y) \geq 0$.\\

          \noindent (c) {\bf Monotonicity in the context of single-leader single-follower.} Consider a single-leader single-follower problem in which the follower's objective $g(\x,\bullet)$ is a strongly convex function on $\Yscr$, a closed and convex set while $\Xscr$ is also a closed and convex.   
\begin{align}
    \tag{Leader} \min_{\x \in \Xscr} \  &f(\x,\y(\x)),  \mbox{ where }\\ 
    \tag{Follower} \y(\x) = \mbox{arg}\hspace{-0.02in}\min_{\y \in \Yscr} \ & g(\x,\y). 
\end{align}
There are many instances when  $f(\bullet, \y(\bullet))$ is a convex function on $\Xscr$ (see~\cite{sherali1984multiple,su2007analysis,sherali83stackelberg,demiguel2009stochastic} for some instances) implying that $\partial f(\bullet, \y(\bullet))$ is a monotone map on $\Xscr$. In other words, the {\em implicit} problem in leader-level decisions can be seen to be characterized by a  convex objective with a {\em monotone} map. However, when viewing the problem in the {\bf full space} of $\x$ and $\y$, i.e.
\begin{align*}
\begin{aligned}
    \min_{\x \in \Xscr, \y} \  f(\x,\y) & \\
    (\tilde{\y}-\y)^\mathsf{T} \nabla_{\y} g(\x,\y) & \geq 0, \qquad \forall \tilde{\y} \ \in \ \Yscr. 
\end{aligned}
\end{align*}
In the full space of $\x$ and $\y$, this is indeed a nonconvex optimization problem~\cite{luo96mathematical}; however, the implicit problem in $x$ may be convex under some assumptions and the resulting subdifferential map is then monotone.\\

\noindent {\bf A.2. Proofs.} \\

\noindent {\bf Proof of Proposition~\ref{mono-g}.} 
\begin{proof}
    In both cases, it is  not difficult to see that $H$ is a monotone map where $ H(\x) \triangleq \prod_{i=1}^\Nbold \partial_{\x^i} h_i(\x^i,\x^{-i})$. Consequently, if $T$ is defined as $T(\x)=G(\x)+H(\x) $, then $T$ is a monotone map which follows from the monotonicity of $G$, defined as
    $G(\x) \triangleq \prod_{i=1}^\Nbold \partial_{\x^i} g_i(\x^i,\x^{-i})$.
\end{proof}

\noindent {\bf Proof of Proposition~\ref{pot-g}.} 
\begin{proof}
    For (a), potentiality follows by noting that for any $\x^i, \tilde{\x}^i \in \Xscr_i$ and $\x^{-i} \in \Xscr^{-i}$, we have
\begin{align*} \widehat{P}(\x^i,\x^{-i}) - \widehat{P}(\tilde{\x}^i,\x^{-i}) &=  P(\x^i,\x^{-i}) - P(\tilde{\x}^i,\x^{-i}) + h_i(\x^i,\y^i(\x^i)) - h_i(\tilde{\x}^i,\y^i(\tilde{\x}^i)) \\&= g_i(\x^i,\x^{-i})+h_i(\x^i,\y^i(\x^i,\x^{-i}))-(g_i(\tilde{\x}^i,\x^{-i})+h_i(\tilde{\x}^i,\y^i(\tilde{\x}^i,\x^{-i}))). \end{align*} \\
For (b), proceeding in a similar fashion, it follows that for any $\x^i, \tilde{\x}^i \in \Xscr_i$ and $\x^{-i} \in \Xscr^{-i}$, we have
\begin{align*} \widehat{P}(\x^i,\x^{-i}) - \widehat{P}(\tilde{\x}^i,\x^{-i}) &=  P(\x^i,\x^{-i}) - P(\tilde{\x}^i,\x^{-i}) + h(\x^i,\x^{-i}) - h(\tilde{\x}^i,\x^{-i})\\  &= g_i(\x^i,\x^{-i})+h_i(\x^i,\y^i(\x^i,\x^{-i}))-(g_i(\tilde{\x}^i,\x^{-i})+h_i(\tilde{\x}^i,\y^i(\tilde{\x}^i,\x^{-i}))). \end{align*}
\end{proof}

\noindent {\bf Proof of Lemma~\ref{sa-rate-quad}.} 
\begin{proof}
Suppose $J_2$ denotes a positive integer such that $(1-2c\alpha_j) \geq 0$ for $j \geq J_2$, i.e. $J_2 = \lceil 2c\theta \rceil \geq 2c\theta.$ Let $J \triangleq \max\{J_1,J_2\}$ and $\mathcal{D} \triangleq \max\left\{ \tfrac{\mathcal{M}^2 \theta^2}{2(2c\theta-1)},J \mathcal{A}_J\right\}.$  For $j = J$, the inductive hypothesis holds trivially. If it holds for some $j >  J$, 
\begin{align*}
\mathcal{A}_{j+1} & \leq (1-2c\alpha_j) \mathcal{A}_j + \tfrac{\alpha_j^2 \mathcal{M}^2}{2} 
 \leq (1-2c\alpha_j) \tfrac{\mathcal{D}}{j} + \tfrac{\alpha_j^2 \mathcal{M}^2}{2}\\
& = (1-2c\alpha_j) \tfrac{\mathcal{D}}{j} + \tfrac{2(2c\theta-1)}{2j}\tfrac{\theta^2 \mathcal{M}^2}{2(2c\theta-1)j} 
 \leq (1-2c\alpha_j) \tfrac{\mathcal{D}}{j} + \tfrac{2c\theta-1}{j}\tfrac{\mathcal{D}}{j} \\
& \leq (1-\tfrac{2c\theta}{j}) \tfrac{\mathcal{D}}{j} + \tfrac{2c\theta-1}{j}\tfrac{\mathcal{D}}{j} 
 = \tfrac{\mathcal{D}}{j} - \tfrac{2c\theta \mathcal{D}}{j^2} + \tfrac{2c\theta \mathcal{D}}{j^2} - \tfrac{\mathcal{D}}{j^2} \le \tfrac{\mathcal{D}}{j} - \tfrac{\mathcal{D}}{j(j+1)} = \tfrac{\mathcal{D}}{(j+1)}. 
\end{align*}
It remains to get a bound on $\mathcal{A}_J$. 
\us{\begin{align} 
\notag \mathcal{A}_J & \leq (1-2c\alpha_{J-1})\mathcal{A}_{\us{J-1}} + \tfrac{\alpha_{J-1}^2 \mathcal{M}^2}{2} 
 \leq \mathcal{A}_{\us{J-1}} + \tfrac{\alpha_{J-1}^2 \mathcal{M}^2}{2}\\		
			& \leq  \left((1-2c\alpha_{J-2}) \mathcal{A}_{J-2} + \tfrac{\alpha_{J-2}^2\mathcal{M}^2}{2}\right)+  \tfrac{\alpha_{J-1}^2\mathcal{M}^2}{2}
\leq \mathcal{A}_{\us{1}} + \mathcal{M}^2\sum_{\ell = 1}^{J-1} \tfrac{\alpha_{\ell}^2 }{2} \notag \\
\label{def-BJ}
& \leq  \mathcal{A}_{\us{1}} + \tfrac{\mathcal{M}^2\theta^2\pi^2}{12} \triangleq \us{\mathcal{A}_1 + B \mathcal{M}^2},  \us{\mbox{ since } \sum_{\ell=0}^{J-1} \tfrac{1}{\ell^2} \leq \tfrac{ \pi^2}{6}.}
\end{align}}
Consequently, for $j \geq J$, 
$\mathcal{A}_{j} \leq \frac{\max\left\{\tfrac{\mathcal{M}^2\theta^2}{2(2c\theta-1)}, J \mathcal{A}_\ssc{J}\right\}}{2j} \leq \frac{\tfrac{\mathcal{M}^2\theta^2}{2(2c\theta-1)}+ J (\us{\mathcal{A}_{\us{1}} + B\mathcal{M}^2})}{2j}.$
\end{proof}

\noindent {\bf Proof of Proposition~\ref{prop-res-error}.} 

\begin{proof}
Throughout this proof, we refer to $J_{\lambda}^T(\x^k)$ by $\z^{k,*}$ to ease the exposition.  Consider the update rule given by (SA), given $\z_{0}  = \x^k$.  We have that 
\begin{align*}
	\|\z_{j+1}-\z^{k,*}\|^2 & = \|\z_j - \alpha_j u_j - \z^{k,*}\|^2 
	  = \|\z_j - \z^{k,*}\|^2 + \alpha_j^2 \|u_j\|^2 - 2\alpha_j u_j^\mathsf{T}(\z_j-\z^{k,*}).
\end{align*} 
Taking expectations on both sides, we obtain that 
\begin{align*}
\mathbb{E}[\|\z_{j+1} &- \z^{k,*}\|^2 \mid \mathcal{F}_{k,j}]  =  
	  \|\z_j - \z^{k,*}\|^2 + \alpha_j^2 \mathbb{E}[\|u_j\|^2 \mid \mathcal{F}_{k,j}]\\
		&  - 2\alpha_j \mathbb{E}[u_j^\mathsf{T}(\z_j-\z^{k,*}) \mid \mathcal{F}_{k,j}] \\
		& = \|\z_j - \z^{k,*}\|^2 + \alpha_j^2 \mathbb{E}[\|u_j\|^2 \mid \mathcal{F}_{k,j}] - 2\alpha_j\mathbb{E}[u_j \mid \mathcal{F}_{k,j}]^\mathsf{T}(\z_j-\z^{k,*}) \\
		& =  \|\z_j - \z^{k,*}\|^2 + \alpha_j^2 \mathbb{E}[\|u_j\|^2\mid\mathcal{F}_{k,j}] - 2\alpha_j \bar{u}_j^\mathsf{T}(\z_j-\z^{k,*}) \\
	& - 2\underbrace{\mathbb{E}[\alpha_j(u_j-\bar{u}_j)^\mathsf{T}(\z_j-\z^{k,*})\mid \mathcal{F}_{k,j}]}_{\ = \ 0}\\
	& = \|\z_j - \z^{k,*}\|^2  + \alpha_j^2 \mathbb{E}[\|u_j\|^2 \mid \mathcal{F}_{k,j}] - 2\alpha_j (\bar{u}_j-\bar{u}_{k}^*)^\mathsf{T}(\z_j-\z^{k,*}),
\end{align*}
where $\bar{u}_j \in F_k(\z_j)$, $\mathbb{E}[\alpha_j(u_j-\bar{u}_j)^\mathsf{T}(\z_j-\z^{k,*})\mid \mathcal{F}_{k,j}] =  \alpha_j(\mathbb{E}[u_j \mid \mathcal{F}_{k,j}]-\bar{u}_j)^\mathsf{T}(\z_j-\z^{k,*}) = 0$, $0 = \bar{u}_k^* \in F_k(\z^{k,*})$ and $(\bar{u}_j-\bar{u}_{k}^*)^\mathsf{T}(\z_j-\z^{k,*}) \geq \tfrac{1}{\lambda}\|\z_j-\z^{k,*}\|^2$ by the $\tfrac{1}{\lambda}$-strong monotonicity of $F_k$. Consequently, we have that 
\begin{align*}
\mathbb{E}&[\|\z_{j+1} - \z^{k,*}\|^2 \mid \mathcal{F}_{k,j}] \le
 (1-\tfrac{2\alpha_j}{\lambda})\|\z_j - \z^{k,*}\|^2  + \alpha_j^2 \mathbb{E}[\|u_j\|^2 \mid \mathcal{F}_{k,j}] \\
& \overset{{\eqref{vsa1}}}{\leq} (1-\tfrac{2\alpha_j}{\lambda})\|\z_j - \z^{k,*}\|^2  + \alpha_j^2 (4M_1^2 \|\x^k\|^2   + 2M_2^2 + (4M_1^2+\tfrac{2}{\lambda^2})\|\z_j-\x^k\|^2) \\
& \leq (1-\tfrac{2\alpha_j}{\lambda})\|\z_j - \z^{k,*}\|^2   + \alpha_j^2 ((\us{8}M_1^2+\tfrac{\us{4}}{\lambda^2}) \|\z_j-\z^{k,*}\|^2 + 4M_1^2\|\x^k\|^2 \\
	& + 2M_2^2 + (8M_1^2+\tfrac{4}{\lambda^2})\|\z^{k,*}-\x^k\|^2) \\
& \leq (1-2\alpha_j(\tfrac{1}{\lambda}-\alpha_j (4M_1^2+\tfrac{2}{\lambda^2}))\|\z_j - \z^{k,*}\|^2\\
&  + \alpha_j^2 (4M_1^2\|\x^k\|^2+ 2M_2^2 + (8M_1^2+\tfrac{4}{\lambda^2})\|\z^{k,*}-\x^k\|^2) \\
& \leq  (1-\tfrac{\alpha_j}{\lambda})\|\z_j - \z^{k,*}\|^2+ \alpha_j^2 (4M_1^2\|\x^k\|^2+ 2M_2^2 + (8M_1^2+\tfrac{4}{\lambda^2})\|\z^{k,*}-\x^k\|^2),
\end{align*}
where the last inequality follows from $\alpha_j (\us{8}M_1^2+\tfrac{\us{4}}{\lambda^2}) \leq \tfrac{1}{2\lambda}$ for $j \geq J_1$ where $j \geq J_1 \triangleq \lceil 2\lambda\theta (\us{8}M_1^2+\tfrac{\us{4}}{\lambda^2})\rceil$. Taking expectations conditioned on $\mathcal{F}_k$ and recalling that $\mathbb{E}[[\|\z_{j+1} - \z^{k,*}\|^2 \mid \mathcal{F}_{k,j}]\mid  \mathcal{F}_k] = \mathbb{E}[\|\z_{j+1} - \z^{k,*}\|^2\mid  \mathcal{F}_k]$ since $\mathcal{F}_k \subset \mathcal{F}_{k,j}$, we obtain the following inequality for \ssc{$j \geq J_1$},
\begin{align*}
\mathbb{E}[\|\z_{j+1} - \z^{k,*}\|^2 \mid  \mathcal{F}_k] &\leq 
(1-\tfrac{\alpha_j}{\lambda})\mathbb{E}[\|\z_j - \z^{k,*}\|^2\mid \mathcal{F}_k]\\
	 + \alpha_j^2 (4M_1^2\mathbb{E}[\|\x^k\|^2 \mid \mathcal{F}_k]+ 2M_2^2 & + (8M_1^2+\tfrac{4}{\lambda^2})\mathbb{E}[\|\z^{k,*}-\x^k\|^2 \mid \mathcal{F}_k]).
\end{align*}
 Consequently, if $\alpha_j = \tfrac{\theta}{j}$, we have a recursion given by 
\begin{align*}
\mathcal{A}_{j+1} \leq (1-2c \alpha_j) \mathcal{A}_j + \tfrac{\alpha_j^2 \mathcal{M}^2}{2}, \quad j \geq J_1
\end{align*}
where $\mathcal{A}_j \triangleq \mathbb{E}[\|\z_{j}-\z^{k,*}\|^2 \mid \mathcal{F}_k]$, $c = \tfrac{1}{2\lambda}$, $\alpha_j = \tfrac{\theta}{j}$, and $\mathcal{M}^2/2 = 4M_1^2\mathbb{E}[\|\x^k\|^2\mid \mathcal{F}_k]+ 2M_2^2 + (8M_1^2+\tfrac{4}{\lambda^2})\mathbb{E}[\|\z^{k,*}-\x^k\|^2\mid\mathcal{F}_k]$. By Lemma~\ref{sa-rate-quad}, we have that 
 \begin{align}\label{sa-bound}
\mathcal{A}_j \leq \frac{\tfrac{\mathcal{M}^2\theta^2}{2(2c\theta-1)}+ J \us{(\mathcal{A}_1+B\mathcal{M}^2)}}{2j}, \quad j \geq J
\end{align}
where $J \triangleq \max\{J_1,J_2\}$, $J_2 \triangleq \lceil 2c\theta \rceil$, and $\us{B \triangleq \tfrac{\theta^2 \pi^2}{12}}$. Since $\mathcal{A}_1 = \mathbb{E}[\|\z^{k,*}-\x^k\|^2 \mid \mathcal{F}_k]$, the numerator in \eqref{sa-bound} may be further bounded as follows.
\begin{align}
& \quad \tfrac{\mathcal{M}^2\theta^2}{2(2c\theta-1)}+J \us{(\mathcal{A}_1+B\mathcal{M}^2)} 
\notag
 = 
\left(\tfrac{\theta^2 }{2(2c\theta-1)}+J \us{B}\right) \mathcal{M}^2  + \us{J\mathcal{A}_1}  \\ 
\notag
	& \leq \left(\tfrac{\theta^2}{2(2c\theta-1)}+J\us{B}\right)(8M_1^2\|\x^k\|^2+ 4M_2^2 + (16M_1^2+\tfrac{8}{\lambda^2})\mathbb{E}[\|\z^{k,*}-\x^k\|^2\mid\mathcal{F}_k]) \\
\notag
	& + J\mathbb{E}[\|\z^{k,*}-\x^k\|^2 \mid \mathcal{F}_k] \\
\notag
	& =  \left(\tfrac{\theta^2}{2(2c\theta-1)}+J \us{B}\right)(8M_1^2\|\x^k\|^2+ 4M_2^2)
\notag
	 \\& + \left(\left(\tfrac{\theta^2}{2(2c\theta-1)}+J\us{B}\right)\left(16M_1^2+\tfrac{8}{\lambda^2}\right)+J\right)\mathbb{E}[\|\z^{k,*}-\x^k\|^2\mid\mathcal{F}_k]. 
\label{ineq-num}
\end{align}
We have that
\begin{align*} & \quad \mathbb{E}[\|\z^{k,*}-\x^k\|^2 \mid \mathcal{F}_k]  \leq 2\|\x^k - \x^*\|^2 + 2\mathbb{E}[\|\z^{k,*} - \x^*\|^2 \mid \mathcal{F}_k] \\
	& = 2\|\x^k - \x^*\|^2 + 2\mathbb{E}[\|J_{\lambda}^T(\x^k) - \x^*\|^2 \mid \mathcal{F}_k] \\
 & \leq 2\|\x^k - \x^*\|^2 + 2\|\x^k - \x^*\|^2 \leq 8 \|\x^k\|^2+8 \|\x^*\|^2, 
 \end{align*} 
where the second inequality follows from $\|J_{\lambda}^T(\x^k)-\x^*\| = \|J^T_{\lambda}(\x^k)-J^T_{\lambda}(\x^*)\| \leq \|\x^k-\x^*\|.$
Consequently, from \eqref{ineq-num},  
$\tfrac{\mathcal{M}^2\theta^2}{2(2c\theta-1)}+J \us{(\mathcal{A}_1+B\mathcal{M}^2)} 
	\leq \nu_1^2 \|\x^k\|^2 + \nu_2^2$, where 
\begin{align*}
\nu_1^2  &  \triangleq 
\left(\left(\tfrac{\theta^2}{2(2c\theta-1)}+JB\right)\left(\us{136}M_1^2+\tfrac{64}{\lambda^2}\right)+\us{8J}\right)\mbox{ and } \\  
 \nu_2^2 &  \triangleq  4\left(\tfrac{\theta^2}{2(2c\theta-1)}+J \us{B}\right)M_2^2  
   +  8\left(\left(\tfrac{\theta^2}{2(2c\theta-1)}+J\us{B}\right)\left(16M_1^2+\tfrac{8}{\lambda^2}\right)+J\right)\|\x^*\|^2. 
\end{align*}
\end{proof}

\begin{proposition}\label{prop-lip-func}\em For $i=1,\cdots, \Nbold$, consider  the problem (Player$_i(\x^{-i}$). Suppose for $i =1, \cdots, \Nbold$, (a.i) and (a.ii) hold.  

    \noindent (a.i) $\Xscr_i \subseteq \Real^{n_i}$ and $\Yscr_i \subseteq \Real^{m_i}$ are closed and convex sets.

    \smallskip

\noindent (a.ii) $F_i(\x,\bullet,\omega)$ is a $\mu_F(\omega)$-strongly monotone and $L_F(\omega)$-Lipschitz continuous map on $\Yscr$ uniformly in $\x \in \Xscr$ for every $\omega \in \Omega$, and there exist scalars $\mu_F,L_F  > 0$ such that $\inf_{\omega \in \Omega}\mu_F(\omega) \geq  \mu_F$ and $\sup_{\omega \in \Omega}L_F(\omega) \leq L_F$.

\smallskip

Suppose $\tilde{f}_i(\x,\y_i,\omega)$ is continuously differentiable on $\Cscr \times \Real^{m_i}$ for every $\omega \in \Omega$ where $\Cscr$ is an open set containing $\Xscr$ and $\Xscr$ is bounded. Then the function {$f_i^{\bf imp}$}, defined as ${f_i^{{\bf imp}}(\x)} \triangleq \mathbb{E}[\tilde{f}(\x,\y(\x,\omega), \omega)]$, is Lipschitz continuous and directionally differentiable on $\Xscr$.
\end{proposition}


\bibliographystyle{ieeetr}
\bibliography{extrasvi,../Revision_one/extrasvi}


\end{document}